\documentclass[a4paper,12pt,reqno]{article}
\usepackage[english]{babel}
\usepackage{amsmath,amsfonts,amssymb,amsthm,bbm,nicefrac}
\usepackage{graphics,epsfig,psfrag} 
\usepackage{paralist,scrextend,subfigure}
\usepackage[dvipsnames]{xcolor}
\usepackage{stmaryrd} 	
\usepackage{hyperref} 
\usepackage{float}

\usepackage{esint} 

\usepackage[active]{srcltx} 
\usepackage{color,soul} 
\usepackage[latin1]{inputenc}
\usepackage[OT1]{fontenc}
\usepackage{authblk}

\newtheorem{theorem}{Theorem}[section]
\newtheorem{corollary}[theorem]{Corollary}
\newtheorem{lemma}[theorem]{Lemma}
\newtheorem{proposition}[theorem]{Proposition}

\theoremstyle{definition}
\newtheorem{definition}[theorem]{Definition}
\newtheorem{remark}{Remark}

\def\N{\mathbb{N}}

\def\R{\mathbb{R}}

\let\e=\varepsilon
\let\vp=\varphi
\let\vt=\vartheta
\let\t=\tilde
\let\ol=\overline
\let\ul=\underline
\let\.=\cdot
\let\0=\emptyset

\let\mc=\mathcal

\def\1{\mathbbm{1}}

\def\SMP{strong maximum principle}

\def\be{\begin{equation}}
\def\ee{\end{equation}}

\def\m{\noalign{\medskip}}
\def\rife#1{\eqref{#1}}
\def\vep{\varepsilon}
\def\de{\delta}

\def\vfi{\varphi}

\def\thm#1{Theorem~\ref{thm:#1}}
\def\seq#1{(#1_n)_{n\in\N}}
\def\limn{\lim_{n\to\infty}}

\def\as{\quad\text{as }\;}

\newenvironment{formula}[1]{\begin{equation}\label{#1}}
             {\end{equation}\noindent}

\def\Fi#1{\begin{formula}{#1}}
\def\Ff{\end{formula}\noindent}

\title{\bf Traveling waves for a nonlocal KPP equation and mean-field game models of knowledge diffusion}

\begin{document}

\author{Alessio Porretta\footnote{Dipartimento di Matematica, Universit\`a di Roma Tor Vergata, Via della Ricerca Scientifica 1, 00133 Roma, email: porretta@mat.uniroma2.it.  Partially supported by Indam Gnampa projects 2018 and 2019 and by Fondation Sciences Math\'ematiques de Paris} \quad 
and  \quad Luca Rossi\footnote{Ecole des Hautes Etudes en Sciences Sociales, PSL Research University,  
Centre d'Analyse et Math\'ematiques Sociales, 54 boulevard Raspail, 75006 Paris, France, email: rossi@ehess.fr}}
%
%

\maketitle


\begin{abstract}
{\footnotesize
We analyze  a mean-field game model proposed by economists R.E.~Lucas and B.~Moll \cite{LuMo} to describe economic systems where production is based on knowledge growth and diffusion. This model reduces to  a  PDE system where  a backward Hamilton-Jacobi-Bellman equation is coupled with a forward KPP-type equation with nonlocal reaction term. 
We  study the existence of  traveling waves for this mean-field game system, obtaining the existence of  both critical and supercritical waves.
In particular we prove a conjecture raised by economists on the existence of a critical balanced growth path for the described economy, supposed to be  the expected  stable  growth in the long run.  We also provide nonexistence results which clarify the role of parameters in the economic model.  

In order to prove these results, we build fixed point arguments on the sets of  critical waves for the forced speed problem arising from the coupling in the KPP-type equation. To this purpose, we provide a full characterization of the whole family of traveling waves for a new class of KPP-type equations with nonlocal and nonhomogeneous reaction terms. This latter analysis has an independent interest  since it shows new phenomena induced by the nonlocal effects and a different picture of critical waves, compared to the classical literature on Fisher-KPP equations.}
\end{abstract}


\tableofcontents

\section{Introduction}

There is  a huge literature in macroeconomics devoted to the analysis of knowledge-based economic systems, where production and learning play key roles. As a sample reference, we only  cite here  \cite{arrow}  among the pioneering papers on this topic. The most recent contributions in this field have renovated the interest in the quantitative analysis of this kind of models, see  e.g.~\cite{alvarez+al}, \cite{konig+al}, \cite{Luttmer}, \cite{Perla}. On the lines of this research, in 2014 R.E. Lucas and B. Moll introduced  a new refined model to describe an economy of knowledge growth and diffusion (\cite{LuMo}). This model, resulting in a  system of PDEs, proved to be  a source of many  interesting mathematical questions which are the object of this work.  

Compared to other previous  models in macroeconomics,  R.E. Lucas and B. Moll  put new emphasis on the interaction between the individual optimization and the evolution of the economic environment which results from individual behaviors. In their model, the agents are characterized by their level   of productivity-related  knowledge (or technology) and split their time between producing and meeting other people in order to exchange ideas and  improve their knowledge. The evolution of this economy is globally described by the productivity distribution function, which is driven by people's choices. Conversely, the individual strategies search for an optimal equilibrium between  the time devoted to producing and the time spent to increase the technological level of production; this choice obviously depends itself on the global status of the economic environment. 

This kind of interaction  is  typical in mean-field game models, which aim at studying the interplay (and the occurrence of  Nash equilibria)  between individual decisions and collective behavior.  So far, mean field game theory, introduced by 
J.-M.~Lasry and P.-L.~Lions (see \cite{GLL,LL-jap}), has been rapidly spreading in many fields of applications and currently leads to 
new interesting problems in the theory of PDEs. Nowadays, mean-field game theory attracts more and more interest among economists, since~it provides support to develop models for heterogeneous agents. 
We~refer the reader to \cite{Achdou+al} for  a discussion of several mean-field game models in macroeconomics and we borrow from this paper the following short presentation of the Lucas-Moll model. 
 
In this model, any single agent has some level of knowledge/productivity $z$
and decides to allocate  a fraction $s\in [0,1]$ of his/her time (a unit of labor per year) to search for new ideas or technologies (in order to increase the productivity level) by interacting with other people, with  $\alpha(s)$ Poisson rate of probability to meet another
agent. 
As a result of a meeting, the productivity associated to knowledge level
becomes the maximum of the productivity of the two agents.  
Thus,  the  individual dynamics is  described through  the stochastic process $x_t:= \log (z_t)$
(here $z_t$ is the productivity level), which 
is governed by the SDE 
 $$
 dx_t= \sqrt2\, \kappa \, dB_t + dJ_t
 $$
 where $B_t$ is a standard $1$-dimensional Brownian motion, $\kappa>0$, and $J_t$ is a Poisson process with intensity $\alpha(s_t)$ that jumps when individuals meet someone with a higher level of productivity during the time $s_t$. 
 The Brownian motion accounts for fluctuations in the individual productivity. 
 In the language of economists,  the Brownian noise may also represent the individual process of experimentation and innovation, whereas the learning process is referred to as imitation.
 
The agents' goal is to maximize their production,  which is of course proportional to the time devoted to produce (that is $(1-s)$) and to the current level of  productivity~$z$. Then,  the value function of  a single agent, conditionally to the initial condition $x_t= x$,  is given by
 $$
 v(t,x)= \sup_{s_\tau\in [0,1]} {\mathbb E}_{t,x} \int_t^{+\infty} e^{-\rho(\tau-t)} (1-s_\tau) e^{x_\tau} d\tau.
 $$
Here  $\rho$ is the discount factor, while the control strategy of the agent is the process $\{s_t\}$ taking values in $[0,1]$ (recall that $s_t$ also enters in the Poisson process $J_t$
involved in the dynamics of $x_t$, which has intensity $\alpha(s_t)$). 

The Bellman equation of dynamic programming yields the following Hamilton-Jacobi equation for $v$:
$$
-\partial_t v -  {\kappa^2} \partial_{xx} v + \rho v = \max_{s\in [0,1]}\left\{(1-s)e^x + \alpha(s) \int_x^{+\infty} (v(y,t)-v(x,t)) f(t,y)dy \right\} \,,
$$
 where $f(t,x)$ is the density of the log-productivity distribution function at time $t$ (i.e. the law of $x_t$). 
 As derived by Lucas and Moll, the equation for $f$ reads as
 \be\label{fp}
 \partial_t f -\kappa^2  \partial_{xx} f = f(t,x) \int_{-\infty}^x \alpha(s^*(t,y)) f(t,y)dy - \alpha(s^*(t,x)) f(t,x) \int_x^{+\infty} f(t,y)dy,
 \ee
where $s^*(t,x)$ is the optimal feedback strategy of the  agents.  

It is not difficult to understand the equation of $f$ as a balance of mass.  Indeed, the  density $f(t,x)$ changes according not only  to the individual noise of the agents, but also to the exchange of  knowledge among the population. To this respect,  the right-hand side should be understood as  a balance (at time $t$)  between 
new people who  upgrade  their knowledge  up to level $x$ by meeting someone with such level, and people who 
leave the level $x$ because they increase their knowledge by learning from someone with higher technology.  
In particular, the $L^1(\R)$ norm of $f$ is preserved. {We point out that the above description applies to an equilibrium configuration, in the spirit of Nash equilibria: indeed, the density $f$ appears {a priori} as an exogenous  datum in the optimization of the agents,  and the equilibrium is  realized a posteriori by assuming that $f$ is  actually driven by the optimal  strategy used by the  agents.}
 
Summing up, the mean-field game system proposed by  Lucas and Moll in their knowledge-production model can be stated as follows:
\be\label{system}
\begin{cases}
\displaystyle
-\partial_t v -\kappa^2\partial_{xx} v + \rho \, v = \max_{s\in [0,1]} \left\{ (1-s)e^x+\alpha(s)\int_x^{+\infty} [v(y)-v(x)]f(y)dy\right\}
\\
\m
\displaystyle \partial_t f -\kappa^2 \partial_{xx} f = f(x) \int_{-\infty}^x \alpha(s^*)f(y)\, dy
 -f(x) \alpha(s^*)\int_x^{+\infty}f(y)dy \\
\m
\displaystyle s^*= {\rm argmax} \,  \left\{ (1-s)e^x+\alpha(s)\int_x^{+\infty} [v(y)-v(x)]f(y)dy\right\} \\
\m
f(0)=f_0
\end{cases}
\ee 
which is set for $t>0\,,\, x\in \R$ and the normalization condition $\int_{\R} f(t,y)dy = 1$.
\vskip1em

Among the most important questions raised by Lucas and Moll in the analysis of this model, they addressed the  problem of existence of traveling waves solutions for system \rife{system}.  These solutions, which are called {\it balanced growth paths}  in the language of economics, are solutions of the type
\be\label{twz}
v (t,x)= e^{ct}\nu(x-ct)\,,\quad  f(t,x) =\varphi(x-ct).
\ee
As explained in \cite{LuMo},  this kind of solutions (usually rephrased in terms of the productivity variable $z$) plays a very crucial role to understand  the behavior of the economy in the long run  and the existence of sustainable growth strategies. See also Remark~\ref{interp} where we discuss the interpretation of  our results in terms of the original model. 

In \cite{LuMo}, Lucas and Moll introduced a numerical algorithm to show the existence of balanced growth paths in the case that the agents are not affected by individual noise, which can be called the deterministic case ($\kappa=0$) for system \rife{system}. Further results for the case without diffusion were given in 
\cite{Wolfram1,Wolfram2}.  

On one hand, introducing a diffusion term in the form of individual noise for the agents looks very natural for the model, since it allows one to consider  fluctuations in the individual productivity and  prevents some additional constraint for balance growth paths (like an a priori prescription of  a Pareto tail for the initial distribution), see e.g.~the discussion in \cite{Luttmer}. 

On another hand,  in the diffusive case  the analysis of  traveling waves for system~\rife{system}  looks more challenging and intriguing. 
In the case of constant learning technology function $\alpha(s)= \alpha_0$ (new ideas arrive without 
the need of going in search for other people) the cumulative distribution function $F(t,x)= \int_{-\infty}^x f(t,y)dy$
satisfies the classical Fisher-KPP equation (\cite{KPP}). This case was extensively discussed in \cite{Luttmer}.  


In the case of variable learning technology function $\alpha(s)$, it was conjectured in~\cite{Achdou+al,LuMo} that system~\rife{system} admits balanced growth paths and, in particular, the  limiting profile distribution  in the long time  should be a solution of the form \rife{twz} satisfying 
\be\label{ccri}
c= 2\kappa \, \sqrt{\int_{\R} \alpha(s^*(y))\varphi(y)dy}.
\ee
This question is very relevant for the economic model because this would identify a critical growth rate in the long run for balanced growth paths. 
\vskip1em
The purpose of this article is to prove, under fairly general assumptions,  the existence of such a critical traveling wave for system \rife{system}. From a PDEs viewpoint, this is especially interesting because it involves both a nontrivial  extension  of the standard analysis of Fisher-KPP equations and the construction of critical equilibria for the mean-field game system,  namely a fixed point argument  on  a family of traveling~waves.

Results in this direction were given in \cite{cinesi} for the case of a linear function $\alpha(s)=\alpha s$.
By contrast, in the original model  suggested by Lucas and Moll, $\alpha(\cdot)$ is supposed to be a  strictly concave, increasing function such that $\alpha'(0)=+\infty$ and $\alpha'(1)>0$. 
According to \cite{LuMo}, this setting of assumptions seems to match  real situations on account of experimental data, and power type functions like $\alpha(s)= s^\eta$, $\eta\in (0,1)$ are typical examples. 

In the economic interpretation, assuming $\alpha'(0)=+\infty$ implies that people will never stop searching for new ideas
and a possibly small but not trivial fraction of time is devoted to search for new technology, even at a large productivity level $z$. This results into the condition that the optimal policy $s^*$ in \rife{system} satisfies
$$
s^*(t,x)>0\,, \,\, x\in \R\,, 
$$ 
but of course $s^*(t,x) \to 0$ as $x\to \infty$. The second condition  $\alpha'(1)>0$ also has a clear interpretation in the model, namely that people with a sufficiently low  level of knowledge should devote all their time to go in search for new technology; this means that there exists a threshold $z_0>0$ such that it is not convenient (or not possible) to start producing if the knowledge level is smaller than $ z_0$
(this typically happens for new producers). In the logarithmic variable $x= \log z$, this implies that there exists $x_0\in \R$ such that 
$$
s^*(t,x) \equiv 1 \qquad \forall x \leq x_0\,.
$$ 
Under the above constitutive assumptions on the learning technology  function $\alpha(\cdot)$, in this paper we derive the following  results:
\begin{itemize}

\item if $\rho\geq 2\kappa \sqrt{\alpha(1)}$ and $\alpha(1)>\kappa^2$, there exists a balanced growth path (i.e.~a solution of  \rife{system} in the form \rife{twz}) with a growth rate $c$ satisfying the critical identity \rife{ccri}.  
Moreover there holds that $2\kappa^2<c<2\kappa \sqrt{\alpha(1)}$.

\item for every $c$ such that $2\kappa\sqrt{\alpha(1)}\leq c < \alpha(1)+\kappa^2$ and $c<\rho$, there exist  balanced growth paths with growth rate $c$ (which are not critical).

\item there are no balanced growth paths with growth rate  $c\leq 2\kappa^2$ nor  $c\geq \alpha(1)+\kappa^2$.

\end{itemize}

The first item above is our main contribution and proves the conjecture in \cite{Achdou+al} about the existence of traveling waves with critical growth. 
We refer to Theorem~\ref{main} for a precise statement, where we also   discuss   the optimality of the conditions 
on~$\rho,\kappa,\alpha$.
Let us mention that the existence of a critical traveling wave for system \rife{system} is also  proved independently in the very recent paper \cite{ryzhik} under the assumption
	that the discount factor $\rho$ and the intensity $\alpha$ of the technology function are sufficiently~large.

In the second item we show that there is a whole family of other traveling waves with supercritical speed.
This proves to be consistent with the typical behavior of KPP-type equations.
However, the existence of an upper bound ($\alpha(1)+\kappa^2$) for the velocities,  which is optimal owing to the third item,
is not an intrinsic feature of KPP equations and it is rather an outcome of the coupling with the value function $v$ through the
optimal feedback strategy $s^*$.

Unfortunately, the picture of all possible waves of system~\rife{system} is not  yet completely understood, as we will discuss later. However, even if many questions remain open for the system, we believe that our analysis makes a significant advance towards the study of the long time convergence to a stable profile. 


\vskip1em
As it is very typical in mean-field game systems, the construction of equilibria is a consequence of some fixed point argument.  In this context, this leads us to  a careful analysis of traveling waves for  a nonlocal KPP-type equation. Indeed, if $F(t,x):= \int_{-\infty}^x f(t,y)dy$ is the cumulative distribution function, a direct computation
(which we postpone to Section~\ref{results}) reveals that \rife{fp} rewrites for $W:=1-F$ as
\be\label{rd}
\partial_t W- \kappa^2 \partial_{xx} W= W \int_{-\infty}^x A(t,y)(-\partial_x W) dy\,,\qquad
\text{with }\;A:= \alpha\circ s^*,
\ee
 together with the limiting conditions 
 $$W(t,-\infty)=1\,,\qquad W(t,+\infty)=0.$$
 This is a nonlocal reaction-diffusion equation which, in the case $A$ constant, reduces to the 
 classical Fisher-KPP equation
 $$\partial_t W- \kappa^2 \partial_{xx} W= AW(1-W).$$ 
 Traveling waves for the system \rife{system} yield the special case $A=A(x-ct)$ (see Section~\ref{sec:derivation}),
  hence we are led to consider 
 solutions of the form $W(t,x)=w(x-ct)$,~i.e. 
\be\label{forced}
\begin{cases}
- \kappa^2 w''- c w' = w \int_{-\infty}^x A(y)(-w'(y))dy\,,\quad x\in \R\,, \\
w(-\infty)=1\,, \quad w(+\infty)=0\,, \quad w'<0\,.
\end{cases}
\ee
We point out that $w(x-ct)$ is not  just a wave for the equation~\eqref{rd}
because we additionally assume that the nonlocal kernel $A$ is also moving with an imposed velocity~$c$; 
this is why~\eqref{forced}  has to be understood as a {\it forced speed} problem.

 A major part of our work consists in the analysis of solutions to \rife{forced}. 
This corresponds to the traveling wave problem for the cumulative distribution function with a given imposed policy $s$ 
 (and $A=\alpha\circ s$).
 Despite the large literature about nonlocal KPP equations, problem~\rife{forced} presents some peculiar features 
 which had not appeared in previous models.  
  To this respect, we give several new contributions, of independent interest, 
  to the study of forced speed waves for nonlocal KPP equations. 
  
  Assuming that $A(\cdot)$ is a nonnegative nonincreasing  function satisfying 
  $\bar A:= A(-\infty)> A(+\infty)=:\underline A$, 
  we can summarize as follows our  results, to be compared with  what is known for the standard local KPP case:
\begin{itemize}

\item problem \rife{forced} admits waves for all $c> 2\kappa \sqrt{\underline A}$. If $c\geq 2\kappa\sqrt{\bar A}$, there exist waves with speed $c$ and arbitrary normalization at  any point $x_0\in \R$. By contrast, if $2\kappa \sqrt{\underline A}<c<2\kappa \sqrt{\bar A}$, for any  given point $x_0$ there is a minimal height $\theta=\theta(x_0,c)$  such that waves with velocity $c$ only exist with $w(x_0)\geq \theta$.

\item For fixed speed $c\in (2\kappa \sqrt{\underline A}, 2\kappa\sqrt{\bar A})$, all possible waves are an ordered foliation indexed by the value  $\int_{\R} A(y)(-w'(y))dy$, whose maximum is given by 
\be\label{criti}
\frac{c^2}4= \kappa^2 \int_{\R} A(y)(-w'(y))dy.
\ee
The unique wave which satisfies \rife{criti} is called {\em critical};
this is the wave of velocity~$c$ which, at any point, runs at the possible lowest height.

\end{itemize} 

Let us  point out how  the analysis of \rife{forced} proves to be crucial in the study of system~\rife{system}.  In fact, 
our approach is built on a fixed point argument which 
requires the understanding of the full picture of possible waves for the single nonlocal KPP equation.
Then imposing condition \rife{criti} will lead us to a wave for \rife{system}
satisfying the criticality condition~\eqref{ccri}. 
To~this respect, our construction of the critical wave for the mean-field game system~\rife{system} looks completely different from the method employed in~\cite{ryzhik}, where the authors use a  topological degree argument and a suitable approximation procedure which automatically provides a wave for~\rife{system}
satisfying the criticality condition~\eqref{ccri}, 
assuming the parameters $\rho$,$\alpha$ to be sufficiently large.  
The essential difference of the two approaches even raises    
the question of whether the obtained critical waves~coincide.

\medskip

The organization of this paper runs as follows. We leave to the next Section \ref{results} 
the derivation of the traveling wave system and a more precise statement of our main results. As we mentioned, they involve both the single nonlocal  KPP equation and the mean-field game system. Further comments on the optimality of our results are also given below.  Then, Section \ref{nonlocalKPP}   is devoted to the detailed analysis of  solutions to~\rife{forced}.
In Section \ref{mfg-proofs} we come back to the system \rife{system} and we prove the results on the mean-field game model. 


\section{Assumptions and  main results}\label{results}



We come back to the mean field game system \rife{system} in order to make precise the setting of our assumptions.
We assume  that  the learning technology function $\alpha(s)$ satisfies
\be\label{alpha1}
\begin{split}
& \hbox{$\alpha\in C^0([0,1]) \cap C^2((0,1])$ is  increasing,  strictly concave,}
\end{split}
\ee
together with
\Fi{alpha0}
\alpha(0)=0,
\Ff
\Fi{alpha1'}
\alpha(1)>\kappa^2,
\Ff
\be\label{alpha2}
\lim_{s\to 0^+}\alpha'(s)= +\infty\,, 
\ee
\be\label{alpha3}
\alpha'(1) >0\,.
\ee
We already explained in the Introduction the interpretation of conditions \rife{alpha2} and~\rife{alpha3}  in terms of the knowledge diffusion-growth model.  
Besides, condition~\rife{alpha1'}  will turn out to be necessary in order for balanced growh paths to exist (see Proposition~\ref{pro:NC} and Theorem \ref{main}). 
A natural interpretation is that there should be enough probability to meet people and enhance the individual level of knowledge in  order for this model of  economy to reach a significant balanced growth.

Another necessary assumption involves the  discount rate $\rho$, that is,
\be\label{ro}
\rho>\kappa^2\,.
\ee
It will soon appear clear that this is a minimal condition even for the existence of solutions to \rife{system}. Further conditions will be needed  on the discount rate in order to
guarantee the existence of balanced growth paths, which  we will discuss after Theorem \ref{main}.


\subsection{Balanced growth paths and  traveling waves}\label{sec:derivation}

Here we derive the system of traveling waves which is associated to  balanced growth paths for the Lucas-Moll model. Before giving a proper definition of admissible 
solutions, we start by making a few heuristic remarks on the 
solutions of system  \rife{system}.

First of all, we stress that the Hamiltonian function  
\be\label{Ham}
H(t,x;v):=   \max_{s\in[0,1]} \big[(1-s)e^x+ \alpha(s) \int_{x}^{+\infty} [v(t,y)-v(t,x)]f(t,y)dy\big]
\ee
requires the condition $v(t)\in L^1(f(t) dx)$ in order to be finite. Since 
$$
-\partial_t v - \kappa^2\partial_{xx} v + \rho v \geq e^x
$$
by comparison (and the condition $\rho> \kappa^2$) we have $v\geq \frac{e^x}{(\rho-\kappa^2)}$, hence we are led to require $f(t)e^x \in L^1(\R)$. This is to point out that a  natural functional setting for the system \rife{system} should require
\be\label{fex}
  f(1+e^x) \in C^0([0,\infty);L^1(\R))\,,\quad \frac v{1+e^x} \in C^0([0,\infty);  L^\infty(\R))\,
\ee
plus the natural condition that $f(t)$ be a probability density for all $t$.

It is also natural to guess that $v$ be monotone with respect to $x$. This can be observed by differentiating the Hamilton-Jacobi equation.  In fact, 
by standard parabolic regularity, locally bounded solutions $(v,f)$ are at least of class $C^{(1+\theta)/2, 1+\theta}$, in particular $v$ is $C^1$ in the $x$ variable.
Then, the strict concavity assumption on $\alpha$ allows us to use some form of the envelope theorem  (see e.g. \cite[Lemma 1]{CP-CIME})  which implies that $H(t,x;v)$ is differentiable in $x$. Differentiating the Bellman equation we 
deduce that~$v_x$ solves the equation 
\be\label{vxeq}
-\partial_t v_x- \kappa^2\partial_{xx} v_x + \rho v_x = (1-s^*) e^x - \alpha(s^*) v_x\, (1-F)
\ee
where 
$$F(t,x):= \int_{-\infty}^x f(t,y) dy$$
is the cumulative distribution function, and $s^*$ is given in  \rife{system}. Heuristically, this equation yields $v_x\geq 0$ and $v_xe^{-x} \in L^\infty((0,\infty)\times \R)$.  

The monotonicity of $v$  also implies that  $ \int_{x}^{+\infty} [v(y)-v(x)]f(y)dy\geq 0$ and that the function $s^*$ defined in \rife{system}  is a nonincreasing function of $x$. Indeed, since $\alpha(\cdot)$ is concave,  the function 
	$$
	g(s;x):= (1-s)e^x+\alpha(s)\int_x^{+\infty} (v(t,y)-v(t,x))f dy
	$$
	is also concave with respect to $s$ and so either $g(\cdot ;x)$ is decreasing in $[0,1]$ or $s^*(t,x):= \sup\{ \tau\in [0,1]\,:\, g(\cdot;x) \,\, \hbox{is increasing in} \, [0,\tau)\}$. But one can readily check that
	\begin{align*}
	& \{ \tau\in [0,1]\,:\, g(\cdot ; x_2) \,\, \hbox{is increasing in} \, [0,\tau)\} \\
	& \qquad \quad \subset \{ \tau\in [0,1]\,:\, g(\cdot ; x_1) \,\, \hbox{is increasing in} \, [0,\tau)\}  \quad \forall x_1<x_2
	\end{align*}
	hence $s^*(t,x_2)\leq s^*(t,x_1)$.

In fact,  it is possible  to build a solution  $(v,f)$ of \rife{system} satisfying the above properties, provided the discount rate is sufficiently large; however we postpone to a forthcoming article a more detailed analysis about the existence of solutions to the system, which depends both on conditions on initial data and on the range of the discount factor. 

Here, we only concentrate on balanced growth path solutions. The above discussion eventually  leads us to the following definition. 
\vskip0.5em

\begin{definition}\label{BGP} A {\it balanced growth path} (BGP) solution of \rife{system} with growth
	rate~$c>0$ is a triple $(f,v, s^*)$ such that 
$$
f= \vfi(x-ct),\qquad v= e^{ct} \nu(x-ct), \qquad s^*= \sigma(x-ct),
$$
and the following properties are satisfied:
\begin{itemize}

\item $\vfi, \nu \in C^2(\R)$, $\sigma\in W^{1,\infty}_{loc}(\R)$

\item 
$\vfi (1+e^x)\in L^1(\R)$, 
$ e^x \int_x^{\infty}  \vfi(y)dy \in L^1(\R)$

\item $\nu$ is increasing, nonnegative and  $\nu' e^{-x} \in  L^\infty(\R)$

\item $f,v$ are classical solutions of the MFG system \rife{system} (with $f_0= \vfi(x)$) and $ s^*(t,x)= {\rm argmax} \,  \left\{ (1-s)e^x+\alpha(s)\int_x^{+\infty} [v(t,y)-v(t,x)]f(t,y)dy\right\}$.
\end{itemize}

\end{definition}

We now proceed by showing that BGP solutions   can be conveniently reformulated in terms of $v_x$ and the CDF function $F$, and this formulation is well suited for traveling waves. 
Let $(f,v)$ be a solution to~\rife{system}.
We first observe that,  integrating by parts,  we can rewrite (omitting the $t$ variable) 
\be\label{hamvx}
\begin{split}
\int_{x}^{+\infty} [v(y)-v(x)]f(y)dy & =  - \lim\limits_{y\to \infty} (1-F(y))(v(y)-v(x))  + \int_{x}^{+\infty} v_x(1-F) dy
\\
&  = \int_{x}^{+\infty} v_x(1-F) dy\,,
\end{split}
\ee
because $(1-F(y)) v(y)\leq \int_y^{\infty} f(s)v(s) ds \to 0$ as $y\to \infty$.    Due to \rife{vxeq}, and using \rife{hamvx} in the definition of $s^*$, we see that 
%
%
the function $\zeta:=  v_xe^{-x}$  is a bounded solution~to the equation
\be\label{exvxeq}
\begin{cases} 
-\partial_t \zeta- \kappa^2\partial_{xx} \zeta - 2 \kappa^2 \partial_{x}\zeta + (\rho-\kappa^2) \zeta = 
(1-s^*)   - \alpha(s^*)\zeta\, (1-F) & \\
  s^*(t,x)=  \mathop{{\rm argmax}}\limits_{s\in [0,1]} \left\{(1-s)e^x + \alpha(s) \int_x^{\infty} 
  \zeta(t,y)e^{y}(1-F(t,y))dy \right\}.  & 
\end{cases}
\ee
%
The equation of $F$ is also readily found. Integrating the equation of $f$ (and neglecting the terms at infinity),  we have
\begin{align*}
\partial_t F-\kappa^2\partial_{xx} F &=    \int_{-\infty}^x f(\xi) \int_{-\infty}^\xi \alpha(s^*)f(y)\, dy - 
\int_{-\infty}^x \alpha(s^*(\xi))f(\xi) \int_\xi^{+\infty} f(y)dy \\
& = - \big[(1-F(\xi)\int_{-\infty}^\xi \alpha(s^*)f(y)\, dy\big]^x_{-\infty} \\
& = - (1-F(x)) \int_{-\infty}^x \alpha(s^*)f(y)\, dy
\end{align*}
where we just used integration by parts. Therefore, the function $F$ solves the nonlocal KPP equation
\be\label{Feq}
 \partial_t F- \kappa^2 \partial_{xx} F+ (1-F)   \int_{-\infty}^x \alpha(s^*(y))(\partial_xF(y))\, dy=0\,,
\ee  
that is, $W:=1-F$ satisfies~\eqref{rd}.  Now, if $W$ is a traveling wave, i.e.~it is of the form
$W(t,x)= w(x-ct)$,
one can  look for  $\zeta$ and $s^*$ in the form of traveling waves too. This is consistent because if 
$ \zeta(t,x)\!:=\!v_x(t,x)e^{-x}= z(x-ct) $ for some function~$z$,~then
\be\label{strav}
\begin{split}
s^*(t,x) & = \mathop{{\rm argmax}}_{s\in [0,1]} \left[(1-s)e^x + \alpha(s) \int_x^{+\infty}  v_x(t,y)(1-F(t,y))dy \right] 
\\ & = \mathop{{\rm argmax}}_{s\in [0,1]}\, e^{ct}\left[(1-s)e^{x-ct} + \alpha(s) \int_{x-ct}^{+\infty}  z(y)e^yw(y)dy \right] 
\end{split}
\ee
which implies that $s^*$ is a function of $x-ct$, i.e., it is itself a traveling wave. 

Summing up, in the case of BGP solutions, we have that
$$W(t,x)=1-\int_{-\infty}^x \vp(y-ct) dy,\qquad
\zeta(t,x)=v_x(t,x)e^{-x}=e^{ct-x}\nu'(x-ct),$$
are traveling wave solutions of~\eqref{rd}, \rife{exvxeq}, that is, 
%
%
%
$W=w(x-ct)$, $\zeta=z(x-ct)$ and $s^*= \sigma(x-ct)$ are solutions to
\be\label{zw-waves}
\begin{cases}
	\displaystyle
	\kappa^2 \, w''+cw'+w\int_{-\infty}^x A(y)(-w'(y))dy=0,\quad x\in \R
	\\
	w\geq 0,\quad w(-\infty)=1,\quad w(+\infty)=0\,,
	\\
	\displaystyle
	-\kappa^2\, z''+ (c-2\kappa^2) z' + (\rho-\kappa^2) z+ A(x)w\, z= 1-\sigma(x),\quad x\in \R \phantom{\int_{-\infty}^x} 
	\\
	z\geq 0, 
	\quad z\text{ is bounded},
	\quad { z e^x w\in L^1(\R)},\\ 
	 \displaystyle \sigma(x)= \mathop{{\rm argmax}}_{s\in [0,1]} \,  \left\{ (1-s)e^x+\alpha(s)\int_x^{+\infty} z(y) e^y\, w(y)dy\right\} \,,\quad A:= \alpha\circ\sigma\,.
\end{cases}
\ee
This will be the framework where traveling waves will be sought for. 
Let us notice that the conditions at infinity for $w$ are induced by mass conservation in the original system, with the normalization condition   $\int_{-\infty}^{+\infty} f =1$.  The conditions for $z$ follow from  the  conditions on $v_x$ discussed  above.  {The condition $e^x w \in L^1(\R)$ is necessary to give proper sense to $\sigma$ in \rife{zw-waves}, this is why this condition is required in the definition of  BGP solutions; we stress that, using elliptic estimates and Harnack inequality,  the first equation in \rife{zw-waves} implies  $|w'(x)| \leq C w(x)$, so the condition  $e^x w \in L^1(\R)$ implies itself a similar condition for $w'$ (which is  the requirement $\vfi\, e^x\in L^1(\R)$ appearing in Definition \ref{BGP}).}
Let us further recall that $\sigma$ and $A$ are nonincreasing and, as we will see in Proposition~\ref{pro:NC} below,
they are also locally Lipschitz-continuous on $\R$. 

The connection between  BGP solutions of \rife{system} and traveling wave solutions of \rife{zw-waves} will be rigorously analyzed in Proposition \ref{bgp-tw}.
We only stress here that  a one-to-one correspondence is easily given, following the above derivation,  between the solutions $(z,w)$ of \rife{zw-waves} and  the couple $(v_x,f)$. However, an extra condition ($\rho>c$) will be needed in order to build the {\it balanced growth} value function $v$. This specifically comes from the requirement that $v$ be positive, and is  a natural condition in  the described model,  see also Remark \ref{interp}.

\subsection{Statement of the main results}

We now state the main result of the~paper.

\begin{theorem}\label{main} Assume that hypotheses \rife{alpha1}--\rife{alpha3} and \rife{ro} hold true. Then we have:

\begin{enumerate}[$(i)$]

\item  If there exists a BGP solution (with growth $c$) of \rife{system}, then necessarily  
$$
 2\kappa^2<c< \alpha(1)+\kappa^2\qquad \hbox{and } \qquad c<\rho
 $$
(hence \rife{alpha1'} and \rife{ro} are necessary for a BGP to exist)

\item   If $\rho \geq 2\kappa\sqrt{\alpha(1)}$, there exists  a BGP solution   of \rife{system} with growth $c\in (2\kappa^2, 2\kappa\sqrt{\alpha(1)})$ and such that
\rife{ccri} is satisfied.

\item For every $c\in [2\kappa\sqrt{\alpha(1)}, \alpha(1)+\kappa^2)$ such that $c<\rho$, there exist BGP solutions of \rife{system} with growth $c$   (which do not satisfy \rife{ccri}).

\end{enumerate}


\end{theorem}


\begin{remark}\label{main-rem} Several comments are in order to describe the above statement.

\begin{itemize}

\item[(a)] The condition  $\alpha(1)>\kappa^2$ proves to be necessary to leave room for the existence of {\it some}  traveling wave, solution of \rife{zw-waves}, hence for BGP solutions as well.   By contrast, the restriction $\rho>c$ is not needed for  the solutions $(z,w)$ of \rife{zw-waves} to exist. But this restriction is necessary for BGP solutions. In particular this is needed to build a consistent value function $v$ once the traveling waves $v_xe^{-x}$ and $f$ are proved to exist; we refer the reader to Proposition \ref{bgp-tw} for a better comprehension. 

It is interesting to notice that this necessary condition $\rho>c$, linking the discount rate to the possible balanced growth, is very common in the economic literature. Indeed,  as  pointed out to us by B. Moll, this condition  is usually needed both for neoclassical growth models (see e.g. \cite{Ace}) and for balanced growth equilibria induced by endogenous growth (see e.g. \cite{Romer}).


\item[(b)] As already mentioned in the previous item, there is a small gap between the pure analysis of system \rife{zw-waves} and the BGP solutions of \rife{system}. For example, without the requirement $z e^x w\in L^1(\R)$,   
solutions of~\rife{zw-waves} may be found with $\sigma\equiv1$ in the (larger) range of parameters 
$$
\frac 34 \kappa^2\leq {\rho}\leq \kappa^2,\qquad
\kappa^2-\rho\leq \alpha(1)<\kappa^2,
$$
and velocities $c\in[2\kappa\sqrt{\alpha(1)},2\kappa^2)$. These solutions however do not correspond to balanced growth paths because the derivation of
the equations in~\rife{zw-waves} from~\rife{system} crucially relies on the condition $z e^x w\in L^1(\R)$.
	
\item[(c)]	The most important output of Theorem \ref{main} is  the existence
	of at least one wave with velocity   $c\in (2\kappa^2, 2\kappa\sqrt {\alpha(1)})$, which in addition is 
	{\em critical}, in the sense that it fulfills~\eqref{ccri}. 
	
	Notice that  the speed  $c$ of this critical wave   is not precisely known, and since $c<\rho$ is necessary for a BGP to exist, we have to assume $\rho\geq 2\kappa\sqrt {\alpha(1)}$ in order to guarantee the existence of at least one critical wave.  	
	
	Unfortunately, not only we do not know whether this is the unique critical wave, but we also do not know yet  if there are other traveling waves in this range of velocities (but we conjecture that  other noncritical waves exist for $c$ in this range). 
	
By contrast, we know much better what happens for $c\geq 2\kappa\sqrt{\alpha(1)}$; indeed, for every $c\in [2\kappa\sqrt{\alpha(1)}, \alpha(1)+\kappa^2)$ there are traveling waves with speed $c$, and they can have arbitrary normalization at any given point $x_0\in \R$. This is  a whole family of traveling waves with {\it supercritical speed}, because they cannot satisfy condition \rife{ccri}, since this latter condition implies $c<2\kappa\sqrt{\alpha(1)}$.

 \item[(d)] The critical wave found in Theorem \ref{main}-$(ii)$ also satisfies the expected decay as $x\to \infty$, namely that $\frac{-w'}w \to \frac c{2\kappa^2}$.
 
  \end{itemize}
\end{remark}

\begin{remark}\label{interp} 
	Let us recall, from \cite{LuMo}, that the solutions constructed in Theorem \ref{main} have a clear interpretation in terms of the productivity variable $z=e^x$. Indeed, for a balanced growth path solution, the  cumulative distribution function $F$, given in terms of $z$, takes the form
$$
F= \Phi(e^{-ct} z)
$$
for some increasing function $\Phi$. This implies that all level sets of $F$ (the $q-$th quantiles of the CDF function) grow with the same exponential rate, because 
$$
\{z\,: F(t,z)=q\} = \{ z= e^{ct} \Phi^{-1}(q)\}\,.
$$
The fact that all   level sets of $F$  have   the same exponential growth rate justifies the name of {\it balanced growth path} solutions, in terms of the economy.

Let us also mention that the decay rate of the critical wave, mentioned in Remark \ref{main-rem}(d), is also significant for the economic model. This is usually interpreted by 
economists in terms of the Pareto tail of the CDF function; indeed, if  $\frac{-w'}w \to \frac c{2\kappa^2}$, this means that $F(t,z)$ has a tail which decays (in polynomial scale)
as $z^{- \frac c{2\kappa^2}}$ (the precise behavior for the KPP equation
would actually suggest $F= O\left( z^{- \frac c{2\kappa^2}}\, \log z\right)$). In the language of economists, the value $\frac{2\kappa^2}c$
is called the {\it tail inequality} associated to the Pareto-like distribution.  To this respect, our result also proves  the conjecture in~\cite{Achdou+al}  that the critical balanced growth path for system \rife{system} should have tail inequality equal to $\kappa \left( \int_{\R} \alpha(s^*(y)) \vfi(y)dy\right)^{-1/2}$.
\end{remark}

As it is typical in mean field game systems, the solutions we find in  Theorem~\ref{main} arise from a fixed point argument. To this purpose, we first develop a deep study of traveling waves for the single nonlocal KPP equation 
\Fi{nlKPP}
\begin{cases}
\displaystyle
w''+cw'+w\int_{-\infty}^x A(y)(-w'(y))dy=0,\quad x\in\R\\
0\leq w\leq1,\quad w(-\infty)=1,\quad w(+\infty)=0 .
\end{cases}
\Ff
Here we have set the diffusion coefficient $\kappa=1$; this is no loss of generality, up to rescaling $c$ and $A$ 
by $1/\kappa^2$. 
The main difficulties we have to face, compared with the classical KPP equation,
come from the facts that this equation is inhomogeneous and nonlocal in the reaction term, 
which entail, respectively, 
that it is not translation invariant and that the comparison principle fails. 

Equation \rife{nlKPP} is obtained from the mean field game system with $A:=\alpha\circ\sigma$. 
This motivates the setting of assumptions we are interested in, namely,
$A$ is bounded and nonincreasing. We also exclude the case $A$ constant because this reduces to
the standard Fisher-KPP equation (for which basically everything is known).

In our analysis of problem \rife{nlKPP}, we completely characterize   the whole family of traveling waves.

\begin{theorem}\label{thm:nlKPP}
	Assume that $A\in W^{1,\infty}_{loc}(\R)$ is bounded and nonincreasing and that
	$$
	\bar A:= \lim_{s\to -\infty} A(s)\,>\, \underline A:= \lim_{s\to +\infty} A(s)\geq0.
	$$
	The traveling wave problem~\eqref{nlKPP} admits solution if and only if $c>2\sqrt{\ul{A}}$. 
	For any $c>2\sqrt{\ul{A}}$ the family of solutions is given by 
	$$\mc{F}:=(w_\vt)_{\vt\in\Theta},$$
	with $w_\vt$ satisfying $w_\vt(0)=\vt$ and
	$$\Theta=\begin{cases} 
		\displaystyle [\vt_c,1) & \text{ if }\;c\in\big(2\sqrt{\ul{A}},2\sqrt{\bar A}\,\big)\\
		\\
		(0,1) & \text{ if }\;c\in[ 2\sqrt{\bar A},+\infty).
	\end{cases}$$
	The $(w_\vt)_{\vt\in\Theta}$ are strictly ordered and $\vt\mapsto w_\vt$ is a continuous 
	bijection from $\Theta$ to $\mc{F}$ equipped with the $L^\infty(\R)$ norm.
	
	Finally, the ``critical'' waves $w_{\vt_c}$ depend continuously on $c\in(2\sqrt{\ul{A}},2\sqrt{\bar A})$ 
	with respect to the~$L^\infty(\R)$ norm, and the values $\vt_c=w_{\vt_c}(0)$ satisfy
	\Fi{vtclimits}
		\vt_c\nearrow1 \as c\searrow2\sqrt{\ul{A}},\qquad
	\vt_c\searrow0 \as c\nearrow 2\sqrt{\bar A}.
	\Ff
\end{theorem}	

Of course, the choice of the point $0$ for parametrizing $\mc{F}$ is purely arbitrary.

The fact that the waves with equal speed are ordered seems remarkable, because this 
property is typically out of reach
for nonlocal, inhomogeneous problems, due to the lack of comparison principle.
We stress out that the main interest of Theorem~\ref{thm:nlKPP} lies in the range of velocities
$\big(2\sqrt{\ul{A}},2\sqrt{\bar A}\,\big)$, which reduces to the empty set when $A$ is constant. 
So this is the range of traveling waves which come from the genuinely inhomogeneous (and nonlocal) forced speed term $A$.
Outside this range, the picture is similar to the classical KPP equation: for any $c\geq 2\sqrt{\bar A}$ 
the graphs of the family of waves (which in the classical case are simply translations of the same profile) foliate the whole strip $\R\times(0,1)$. By contrast, for $c\in\big(2\sqrt{\ul{A}},2\sqrt{\bar A}\,\big)$, 
the foliation does not fill the whole strip, but only the region to the right of the 
``critical'' wave.
The situation is depicted in Figure~\ref{fig:KPPwaves}.
\begin{figure}[H]
	\centering
	\subfigure[$c<2\sqrt{\bar A}$]
	{\includegraphics[width=6.5cm]{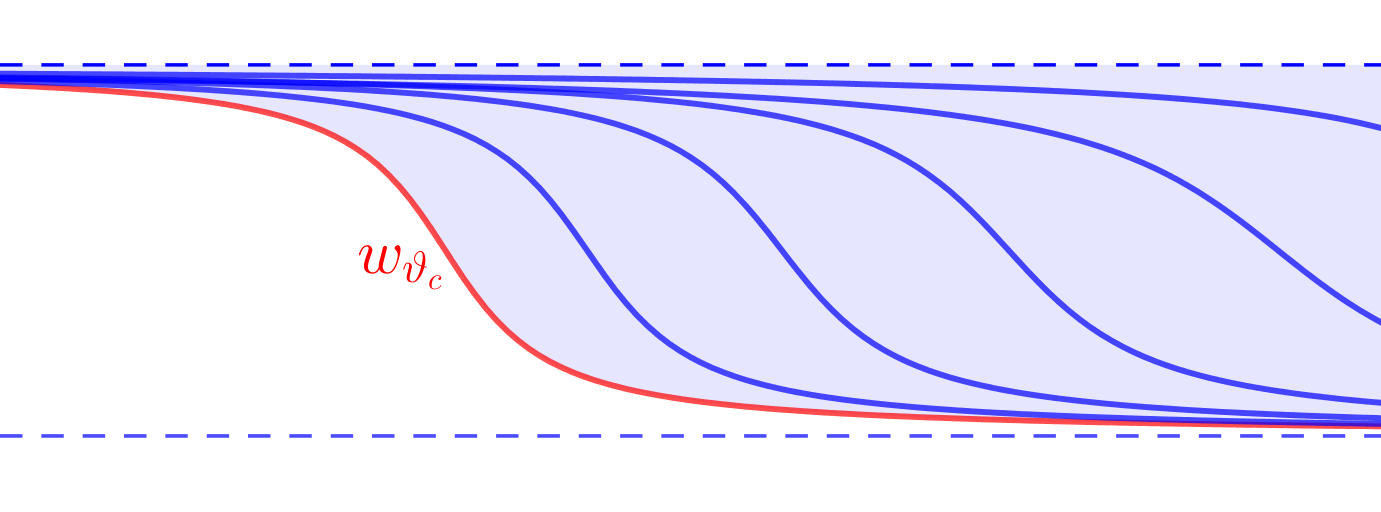}}
	\hspace{1cm}
	\subfigure[$c\geq 2\sqrt{\bar A}$]
	{\includegraphics[width=6.5cm]{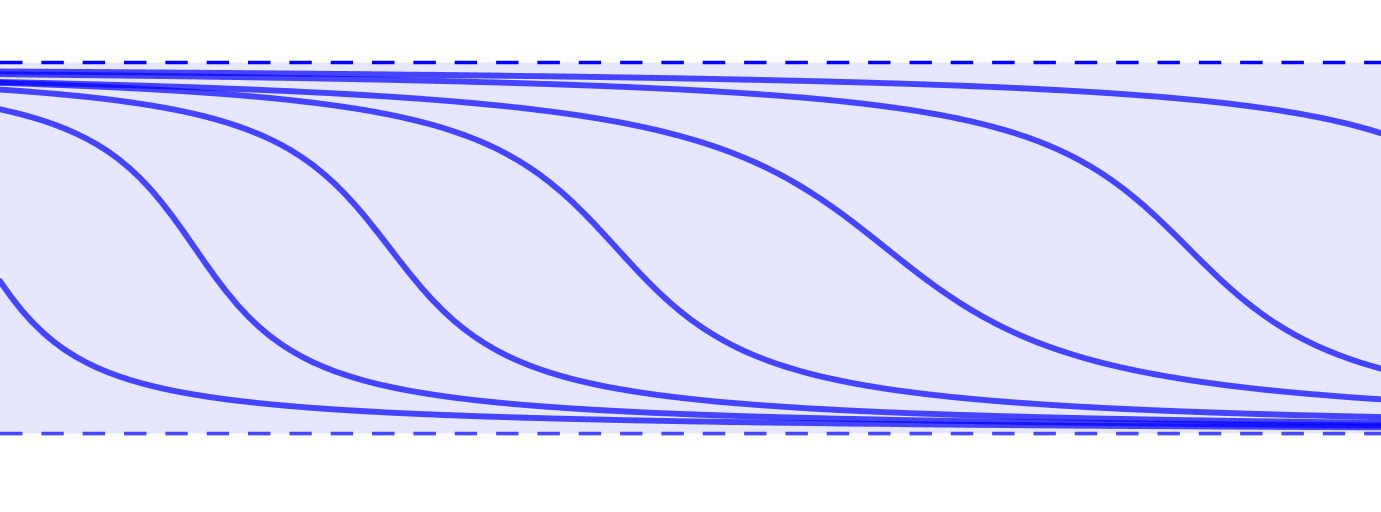}}
	\caption{The two different types of foliation.}\label{fig:KPPwaves}
\end{figure}	 

The two different scenarios can be heuristically explained as follows:
on one hand, if the transition of the wave from $1$ to $0$
takes place (for its main part) far to the right then it would be faced to values of $A$ close to $\underline A$, 
and for such value the range of admissible speeds is classically $c\geq 2\sqrt{\ul{A}}$; this is why 
fronts can be found for any $c>2\sqrt{\ul{A}}$, and they converge pointwise to $1$ as $c\searrow 2\sqrt{\ul{A}}$.
Conversely, if the transition occurs far to the left then $A$ would be close to $\bar A$ and then necessarily 
$c\gtrsim2\sqrt{\bar A}$; hence, for a given speed 
$c<2\sqrt{\bar A}$, the transition cannot occur too much to the left, or, equivalently,
there must exist a pointwise lower bound for the wave. 

In order to understand what happens for $c$ in the range $\big(2\sqrt{\ul{A}},2\sqrt{\bar A}\,\big)$, 
the following operator
will be of crucial importance:
$$
\mc{I}(w):=\int_{\R}A(y)(-w'(y))dy.
$$
From the modeling point of view, $\mc{I}(w)$ is related to the total expectation of meetings 
for the given policy $A:=\alpha\circ s$.
Because of the condition $w(-\infty)=1$, it can be equivalently written as
$$
\mc{I}(w)=\bar A+\int_{\R}A'(y)w(y)dy.
$$
This formulation enlightens the continuity and monotonicity of $\mc{I}$.
The value of $\mc{I}$ on the critical wave $w_{\vt_c}$ turns out to encode the speed in
a very transparent way:
$\mc{I}(w_{\vt_c})=c^2/4$. This immediately
shows that there exist no traveling waves with speed $c\leq2\sqrt{\ul{A}}$. The relationship between 
the waves and the functional $\mc{I}$ is summarized in the following result.  

\begin{theorem}\label{thm:I}
The mapping $\vt\to\mc{I}(w_\vt)$ is a decreasing homeomorphism between
$\Theta$ and $J$, where
$$
J=\begin{cases} 
\displaystyle (\ul A,c^2/4] & \text{if }\;2\sqrt{\ul{A}}<c<2\sqrt{\bar A}\\
\\
(\ul A,\bar A) & \text{if }\;c\geq 2\sqrt{\bar A}.
\end{cases}
$$
\end{theorem}

Another key feature of the operator $\mc{I}$ is that
it encodes the exponential rate of decay of the wave, see
Proposition~\ref{implicit-speed} below. 
Unfortunately, there is one property that we are still missing: the ordering of critical waves with different speeds.
This would be of great help to construct a wave for the system~\eqref{zw-waves} through 
a fixed point argument. 
Nevertheless, we are able to derive the ordering for large $|x|$, 
c.f.~Propositions~\ref{implicit-speed} and \ref{pro:leftordering},
and we use this to cook up a suitable selection principle for the fixed point argument.
%


\section{The nonlocal KPP equation}\label{nonlocalKPP}

This section is devoted to the study of the single traveling wave problem~\eqref{nlKPP},
which corresponds to an assigned production/research strategy.
Namely, throughout this section we assume that 
$A$ is a given function which fulfills the properties derived in Section~\ref{sec:derivation}, 
which are:
$$A\in W^{1,\infty}_{loc}(\R) \text{ \, is nonnegative, nonincreasing},$$
$$
\bar A:=A(-\infty)>\ul A:=A(+\infty).
$$
In the next section, we start by collecting  some tools on the nonlocal equation~\rife{nlKPP}. 

\subsection{Preliminary toolbox}

As a first step, we show basic properties of solutions to \rife{nlKPP}.

\begin{proposition}\label{Harnack}  If~\eqref{nlKPP} admits solution for some $c\in\R$, then necessarily $c>0$ and $w'<0$ in $\R$. In addition, for  any $R>0$, there exists a constant  $C_R$, only depending on $R$, $\bar A=A(-\infty)$ 
	and an upper bound for $c$, such that, for any $x_0\in \R$, there holds 
\Fi{Harna}
\max_{[x_0-R,x_0+R]}w\leq C_Rw(x_0)\,, \qquad \max_{[x_0-R,x_0+R]}(1-w)\leq C_R(1-w(x_0)).
\Ff
\end{proposition}

\begin{proof} 
We preliminarily observe that $w>0$ thanks to the elliptic strong maximum principle.
We then divide the equation in~\eqref{nlKPP} by $w$ and differentiate to get
$$
\frac{w'''}{w}-\frac{w'w''}{w^2}+c\frac{w''}{w}-c\frac{(w')^2}{w^2}
-Aw'=0,\quad x\in\R.
$$
Hence, the function $u:=w'$ satisfies the equation
\Fi{u''}
u''+\Big(c-\frac{u}{w}\Big)u'-Aw u=c\frac{u^2}{w},\quad x\in\R,
\Ff
with zero order coefficient $-Aw\leq0$.
Moreover, by the boundedness of $w$, we know that there exist two sequences 
$\seq{x^\pm}$ diverging to $\pm\infty$ respectively, such that $u(x^\pm_n)\to0$ as $n\to\infty$. 
Applying the weak maximum principle to the equation~\eqref{u''} in $(x_n^-,x_n^+)$
and letting $n\to\infty$, we deduce that $u\leq0$ in $\R$ if $c\geq0$,
whereas $u\geq0$ in $\R$ if $c\leq0$. Then, by the limiting conditions in~\eqref{nlKPP}, we
necessarily have that $c>0$ and $u\leq0$.
The strict inequality $w'=u<0$ follows by applying the 
strong maximum principle to~\eqref{u''}.

As for the Harnack inequalities \rife{Harna}, the first one comes from  standard elliptic theory,  because $w$ solves an equation as $w''+ c w' + V w=0$ where the  potential $V$ satisfies $0\leq V\leq \bar A$.  As for the second one, we observe that if $w$ solves~\eqref{nlKPP} than 
$v(x):=1-w(x)$ satisfies
$$
-v''-cv'+(1-v)g(x)v=0,
$$
where 
$$
g(x):=\frac1v\Big(\bar A-A(1-v)+
\int_{-\infty}^x A'(y)(1-v(y))dy\Big).
$$
On one hand, using $A'v\leq0$ in the above integral shows that $g\geq0$. On the other hand, the fact that
$A'v$ is decreasing yields
$$
g(x)\leq\frac1v\Big(\bar A-A(1-v)+
(1-v)(A-\bar A)\Big)=\bar A.
$$
Therefore, we conclude  as before   that $v(x)\leq C_R u(x_0)$ provided $|x-x_0|\leq R$. This gives the second inequality in \rife{Harna}.
\end{proof}

It will be handy to reformulate the equation in~\eqref{nlKPP} in a different way.
Namely, under the condition $w(-\infty)=1$, integrating by parts the nonlocal
term leads to 
\Fi{eq:integrata}
w''+cw'+w\Big(\bar A-Aw+\int_{-\infty}^x A'(y)w(y)dy\Big)=0,
\quad x\in\R.
\Ff
The advantage of this equation, compared with the one in~\eqref{nlKPP},
is that the only identically constant solutions are $0$ and $1$. 
Now we  show  that {\it all  other solutions  to \rife{eq:integrata} are decreasing waves connecting $1$  and $0$} and then, in particular, they are solutions of \rife{nlKPP}.

\begin{lemma}\label{lem:phi'<0}
Let $0\leq w\leq1$ be a solution of \eqref{eq:integrata} for some $c\geq0$. Then either
$w\equiv0$, or $w\equiv1$, or $c>0$ and $w$ satisfies
$$w(-\infty)=1,\quad w(+\infty)=0\quad\text{and \ $w'<0$ in $\R$.}
$$
Moreover, there exists a constant $L$, only depending on $c,\bar A$, such that $\|w'''\|_\infty \leq L$.		
\end{lemma}

\begin{proof}
We first observe that globally bounded solutions of \eqref{eq:integrata} are also bounded in~$C^3$. 
Indeed, using the monotonicity of $A$ and the bounds on $w$, we notice that~$w$ solves a linear equation $w''+cw'+w V=0$ where 
$$0\leq V:=\Big(\bar A-Aw+\int_{-\infty}^x A'(y)w(y)dy\Big)\leq \bar A.$$ 
By elliptic estimates (see e.g., 
\cite[Theorem~9.11]{GT}), given any point $a\in \R$  we have  
$$
|w'(a)| \leq L   \sup_{x\in [a-1,a+1]} |w(x)| \leq L  
$$
where $ L$ only depends on $c,\bar A$. Then the same conclusion 
holds true (with a larger~$L$) for $w''=-cw'-w V$.  
We then bootstrap by differentiating this equation and observing that
$V'=-Aw'$. This shows that $|w'''|\leq L$ for some $ L$ depending on $c,\bar A$.

Let us now show that bounded solutions are decreasing waves.  First  we observe that $A'\leq0$ and $w\leq1$ imply
\Fi{phisuper}
-w''-cw'\geq A(x)w(1-w)\geq0,\quad x\in\R.
\Ff
We treat separately the cases $c=0$ and $c>0$.

\medskip
{\em Case $c=0$. }\\
In such case \eqref{phisuper} yields $w''\leq0$ in $\R$, hence $w$ is constant.
Then, since $A(-\infty)=\bar A>0$,~\eqref{phisuper} shows that the only possibilities are 
$w\equiv0$ or $w\equiv1$.

\medskip
{\em Case $c>0$. }\\
The inequality~\eqref{phisuper} implies that $-(w'e^{cx})'\geq0$,
which, integrated on $(-\infty,x)$ ($w' e^{cx}$ vanishes at $-\infty$ because $w'$ is bounded), 
yields $w'(x)\leq0$ for any $x\in\R$.
Differentiating~\eqref{eq:integrata} we get the following equation for $w'$:
$$(w')''+c(w')'+w'\Big(\bar A-2Aw+
\int_{-\infty}^x A'(y)w(y)dy\Big)=0.$$
We deduce from the elliptic strong maximum principle that either $w'<0$ in $\R$, or $w'\equiv0$.
In the latter case, as before, we infer from~\eqref{phisuper} and~$\bar A>0$ that $w\equiv0$
or $w\equiv1$.

We are left with the case $w'<0$. In such case $w(\pm\infty)$ exist and satisfy
$0\leq w(-\infty)<w(+\infty)\leq1$.
We integrate by parts the integral in~\eqref{eq:integrata} to get 
\Fi{phiinfty}
-w''-cw'\geq w\Big(\bar A\big(1-w(-\infty)\big)+
\int_{-\infty}^x A(y)(-w'(y))dy\Big).
\Ff
The two terms in the right-hand side are nonnegative.
Suppose by contradiction that $w(-\infty)<1$. Then there exists $k>0$ such that
the right-hand side is larger than $k$ for $x\leq0$, that is,
\Fi{phiexp}
-(w'e^{cx})'\geq ke^{cx}.
\Ff
Integrating on $(-\infty,x)$, for given $x<0$, we obtain $-w'(x)\geq \frac k c$, which is impossible.
Therefore, $w(-\infty)=1$. If, on the other hand,
$w(+\infty)>0$, it is the integral term
in~\eqref{phiinfty} to be larger than some $k>0$ for $x\geq x_0$ such that $A(x_0)>0$,
that is,~\eqref{phiexp} holds for $x\geq x_0$.
Integrating on $(x_0,x)$ yields 
$$-w'(x)\geq \frac kc-e^{-c(x-x_0)}\Big(\frac kc-w'(x_0)\Big),$$
which is again a contradiction. This concludes the proof.
\end{proof}

In the next step we study how to build solutions of \rife{eq:integrata} using   shooting and comparison methods for ODEs in truncated domains.  A key point will be played by the Cauchy problem
in the half-line.

\begin{lemma}\label{lem:ivp} 
  Let $P,Q\in W^{1,\infty}_{loc}(\R)$ and $K\in L^\infty_{loc}(\R)$. 
Given  $c,a,\eta\in\R$, $\vt>0$, the Cauchy problem 
	\Fi{ivp-gen}
	\begin{cases}
		\displaystyle
		w''+cw'+w\Big(P(x)+Q(x)w+\int_{a}^x K(y)w(y)dy\Big)=0,\quad x>a\\
		w(a)=\vt\\
		w'(a)=\eta,
	\end{cases}
	\Ff
	admits a unique (classical) positive solution $w$ in some interval $[a,b)$, 
	with either $b=+\infty$, or $b\in(a,+\infty)$ and
	$w(b^-)=0$ or $+\infty$. 
	Moreover, such solution depends continuously on $c,a,\vt,\eta$ as well as on
	$P,Q$ and $K$ with respect to
	$W^{1,\infty}_{loc}(\R)$ and $L^\infty_{loc}(\R)$ convergences respectively.
\end{lemma}

\begin{proof}
	We formally divide the equation by $w$ and differentiate. We get
	$$\frac{w'''w-w'w''}{w^2}+c\,\frac{w''w-(w')^2}{w^2}
	+P'+Q'w+Qw'+Kw=0,\quad x>a,$$
	that is,
	$$w'''-\frac{w'w''}{w}+cw''-c\frac{(w')^2}{w}
	+Qww'+(Q'+K)w^2+P'w=0,\quad x>a.$$
	We also have that $w''(a)=-c\eta- P(a)\vt-Q(a)\vt^2$. 
	If $P,Q\in C^1(\R)$ and $K\in C^0(\R)$, this is a standard Cauchy problem of the third order, as long as 
	$w$ stays bounded away from $0$. The  existence, uniqueness and
	continuity with respect to the data then follow from the classical theory. 
	In the general case, the same properties are consequences of Carath\'eodory's existence theorem.
	The resulting solution $w$ is such  that $w''$ is absolutely continuous and therefore
	it is a classical solution of \rife{ivp-gen}.
\end{proof}

The next tool is a  comparison principle and will play a crucial role in our analysis. Here and 
in the sequel, whatever elliptic equation is given in the form
$$
w''   = F(x,w, w'),
$$
we say that $w$ is a sub-solution (respectively, super-solution) if $w$ satisfies $w''   \geq  F(x,w, w')$ (respectively, $w''  \leq  F(x,w, w')$).

\begin{lemma}\label{lem:CP}	
	Let $c,a\in\R$ and $P,Q, K$ satisfy the assumptions of Lemma \ref{lem:ivp}. In addition, assume that  
	$Q$ and $K$ are nonpositive.
	
	Let $w_{1}$ and $w_{2}$
	be respectively a positive sub-solution and a positive super-solution to
	the first equation of~\eqref{ivp-gen} 
	in an interval $[a,\beta]$, with   
	$$w_1(a) \geq w_2(a),\qquad
	w_1'(a)\geq w_2'(a),\qquad
	w_2'(a)\leq 0.$$
	Then $w_1/w_2$ is nondecreasing on $[a,\beta]$, and it is  increasing if $w_1'(a)>w_2'(a)$.
\end{lemma}

\begin{proof}	
Suppose first that $w_1'(a)>w_2'(a)$.
Call $\rho:=w_1/w_2$. 
Using all  the information at the initial point, this function satisfies
$$\rho(a)\geq 1,\qquad\rho'(a)=\frac{w_2(a)w_1'(a)-w_1(a)w_2'(a)}{w_1(a)^2}\geq \frac{w_2(a)}{w_1(a)^2}[w_1'(a)-w_2'(a)]>0.
$$
Let $\t\beta$ be the largest value in $(a,\beta]$ such that $\rho>1$ in $(a,\t\beta)$.
	Assume by contradiction that $\rho$ is not  increasing in $[a,\t\beta]$. This
	means that there exist $a\leq x_1<x_2\leq\t\beta$ such that $\rho(x_1)\geq\rho(x_2)$.
	Call 
	$$h:=\max_{[a,x_2]}\rho>1.$$
	Let $\bar x\in[a,x_2]$ be such that $\rho(\bar x)=h$.
	We know that $\bar x\neq a$. Moreover, if $\rho(x_2)=h$ then necessarily $\rho(x_1)=h$.
	Hence, in any case, we can take $\bar x\in(a,x_2)$.
	We define $\psi:=hw_2-w_1$. Then $\psi\geq0$ 
	in $(a,x_2)$ and $\psi(\bar x)=0$.
	%
	The function $\psi$ satisfies the following 
	differential inequality in $(a,\bar x)$:
	\[\begin{split}
	-\psi''-c\psi'\geq &
	hw_2\Big(P+Qw_2+\int_{a}^x Kw_2\Big)
	-w_1\Big(P+Qw_1+\int_{a}^x Kw_1\Big).
	\end{split}\]
	Whence, using the fact that $w_1>w_2>0$ in $(a,\t\beta)\supset(a,\bar x)$
	and that $Q$ and $K$ are nonpositive,
	we eventually deduce that,  in $(a,\bar x)$,
	$$-\psi''-c\psi'\geq 
	\psi\Big(P+Qw_1+\int_{a}^x Kw_1\Big),$$
	which means that $\psi$ is a super-solution of some linear elliptic equation.
	As a consequence, 
	since $\psi$ attains its minimal value $0$ at $\bar x$, the Hopf lemma yields $\psi'(\bar x)<0$
	which implies that $\psi<0$ is some right neighborhood of $\bar x$. 
	This contradicts the definition of~$h$.
	We have thereby shown that $\rho=w_1/w_2$ is strictly increasing in $[a,\t\beta]$, whence 
	in particular $\rho(\t\beta)>1$.
	It follows that $\t\beta=\beta$ and this concludes the proof in the 
	case~$w_1'(a)>w_2'(a)$.
	
	Assume now that $w_1'(a)=w_2'(a)=:\hat\eta$. 
	Fix a number $\vartheta\in[w_2(a),w_1(a)]$ and, for~$\eta\in\R$, 
	let $w^\eta$ be the solution of~\eqref{ivp-gen} provided by
	Lemma~\ref{lem:ivp}.
	Applying the property derived before we deduce from one hand that if $\eta>\hat\eta$ then 
	$w^\eta/w_2$ is   increasing on $[a,\beta]$, and from the other that if 
	$\eta<\hat\eta$ then
	$w_1/w^\eta$ is   increasing on some interval 
	$[a,\beta_\eta]$ on which $w^\eta$ is positive. 
	%
	It follows from the continuous dependence with respect to the data, ensured by
	Lemma~\ref{lem:ivp}, that 
	both $w^{\hat\eta}/w_2$ and $w_1/w^{\hat\eta}$ are nondecreasing on 
	$[a,\beta]$ (and that $w^{\hat\eta}$ is positive there), whence 
	$$\frac{w_1}{w_2}=\frac{w_1}{w^{\hat\eta}}\, \frac{w^{\hat\eta}}{w_2}$$
	is nondecreasing on $[a,\beta]$.
\end{proof}

We now deduce some consequences of the previous comparison principle.  An easy one, readily observed, is that if there exists a (positive) constant sub-solution of the equation, then any super-solution starting below must be nonincreasing.

\begin{corollary}\label{cor:Dp}
	For any $a<b$, $M>0$, and $\vt\geq \gamma >0$, consider the Dirichlet problem
	\Fi{DP-gen}
	\begin{cases}
	\displaystyle
	w''+cw'+w\Big(M - A(x)w+\int_{a}^x A'(y)w(y)dy\Big)=0,\quad x\in(a,b) & \\
	w(a)=\vt,\quad w(b)=\gamma, &  
	\end{cases}
	\Ff
	where $A(x)\in W^{1,\infty}(\R)$ is nonincreasing.
	
	Assume that $M\geq A(a)\vt$ and that there exists  a positive  solution $W$ of the first equation of \rife{DP-gen} such that $W(a)=\vt$ and $W'(a)\leq 0$.
	
	Then, for every $\gamma\leq W(b)$ the problem \rife{DP-gen} admits a unique positive solution, which is nonincreasing and decreasing if $\gamma<W(b)$. Moreover, $\frac{W}{w}$ is   increasing unless $w=W$.
\end{corollary}

\begin{proof}
For $\eta\leq 0$, let $w_\eta$ be the solution to the first equation of~\eqref{DP-gen},
	with initial condition 
	$$w^\eta(a)=\vt,\qquad (w^\eta)'(a)=\eta.
	$$
	Such solution exists and it is positive in a right neighborhood of $a$ thanks to 
	Lemma~\ref{lem:ivp}.  Moreover, since $M\geq A(a)\vt$, the constant function $w^0\equiv \vt$ is a sub-solution; hence, due to $\eta\leq 0$ and Lemma~\ref{lem:CP}, $w^\eta$ is nonincreasing (and decreasing if $\eta<0$).    Lemma~\ref{lem:CP} also implies that the
	$w^\eta$ are increasing with respect to $\eta$ (in the set where they are positive), and it is readily seen that
	$w^\eta$ vanishes before the point $b$ 
	if $-\eta$ is sufficiently large. Finally, if $\eta= W'(a)$, then $w^\eta= W$ by uniqueness. Therefore, for $\eta>W'(a)$, $w^\eta\geq W$ and  remains positive up to $x=b$. 
	It then follows from the continuity of the solution with respect to $\eta$, 
	that for any $\gamma\leq W(b)$ there exists a unique value $\eta\leq W'(a)\leq 0$
	such that $w^\eta(b)=\gamma$. 
	Finally, applying again Lemma~\ref{lem:CP},
	with $w_1=W$, we deduce that $W/w^\eta$ is nondecreasing
	and it is increasing if $\eta<W'(a)$.
	This concludes the proof.
\end{proof}

We complete our toolbox with another lemma showing that two waves  will be arbitrarily close in the future provided they were sufficiently close in the past.

\begin{lemma}\label{past-future} Let $w$ be a solution to \rife{nlKPP}. 
	For every $x_0\in \R$ and $\vep>0$, there exists  $\de>0$ such that
	if $\tilde w$ is another solution satisfying
	$ |w(x)-\tilde w(x)| \leq \de$ for all  $x\in  (-\infty,  x_0)$, then 
$$
|w(x)-\tilde w(x)| \leq \vep \qquad \forall x\in\R\,.
$$
\end{lemma}

\begin{proof}  The difference $\psi:=\t w-w$ of two solutions $w$, $\t w$ to~\rife{nlKPP}
	satisfies
	$$
	\psi''+c\psi'+\psi\big(\bar A-A(\t w+w)\big)=
	-\t w\int_{-\infty}^{x} A'\t w+w\int_{-\infty}^{x} A'w.
	$$
	Assume now that $ |w-\tilde w| \leq \de$ in an interval $(-\infty,  x_0)$.
	We can then estimate the right-hand side for $x\leq x_0$ as follows:
	$$
	\left|-\t w\int_{-\infty}^{x} A'\t w+w\int_{-\infty}^{x} A'w\right|\leq
	|w-\t w|\int_{-\infty}^{x} |A'\t w|+w\int_{-\infty}^{x}(-A')|\t w-w|\leq
	2\delta\bar A.
	$$
	Thus, we can take $a=x_0-1$ and deduce from interior elliptic estimates
	that $|\psi'(a)|\leq C\delta$ for some $C$ only depending on $c$ and $\bar A$.
	We therefore have that at the point $a$ 
	both $\t w-w$ and $\t w'-w'$
	are of order $\delta$.	
	
	The function $\t w$ satisfies an initial value problem 
	of the type~\eqref{ivp-gen}, with $P=P_{\t w}$, $Q$ and
	$K$ given by
	$$P_{\t w}:=\bar A+\int_{-\infty}^{a} A'\t w,\qquad
	Q(x):=-A(x),\qquad
	K(x):=A'(x).$$ 
	The function $w$ satisfies the same type of problem, but with 
	$$P_w:=\bar A+\int_{-\infty}^{a} A'w.$$
	Using that $|w-\tilde w| \leq \de$ in $(-\infty,a)$ we see that
	$$\int_{-\infty}^{a} A'w-\delta(\bar A-A(a))\leq
	\int_{-\infty}^{a} A'\t w\leq\int_{-\infty}^{a} A'w+\delta(\bar A-A(a)),$$
	hence $|P_{\t w}-P_{w}|\leq\delta(\bar A-A(a))$. Recall that
	the values of $\t w$, $\t w'$ at~$a$ are close to the ones of 
	$w$, $w'$ up to an order $\delta$. As a consequence, by 
	Lemma~\ref{lem:ivp}, for any $\e>0$ and $x_1>a$,
	we can find $\delta<\e/2$ small enough so that $|\t w-w|<\e/2$ in $(-\infty,x_1]$.
	Choosing $x_1$ such that $w(x_1)<\e/2$, and reminding that $w$ and $\t w$ are decreasing by
	Proposition~\ref{Harnack}, we conclude that 
	$|w-\tilde w| \leq \vep$ in $\R$.
\end{proof}

\subsection{Construction of the traveling waves}

This section is devoted to the construction of   waves -- that is,   solutions to~\eqref{nlKPP}. We will distinguish two cases depending on the range of the velocity $c$.

\vskip1em

{\bf Case (A)} $c\geq 2\sqrt{\bar A}$.

The construction of solutions of \rife{nlKPP} is easier in this case because, for such values of $c$, we know that there is a traveling wave solution~$\psi$ for the classical KPP equation:
$$
\begin{cases}
	\displaystyle
	\psi''+c\psi'+\bar A\psi(1-\psi)=0,\quad x\in\R\\
	\psi(-\infty)=1,\quad \psi(+\infty)=0.
\end{cases}
$$
Observe that $\psi$ is a super-solution of \eqref{nlKPP}. We further know that $\psi'<0$.
We consider the normalization condition $\psi(0)=\vt$, with $\vt$ arbitrarily fixed in $(0,1)$.
%
%
%
%
%
%
%
%
%

For $n\in\N$ and $\zeta\in\R$, we introduce the truncated problem
\Fi{-nn}
\begin{cases}
	\displaystyle
	w''+cw'+w\Big(A(-n)\psi(-n-\zeta)-Aw+\int_{-n}^x A'(y)w(y)dy\Big)=0,\quad x\in(-n,n)\\
	w(-n)=\psi(-n-\zeta),\quad w(n)=\psi(n-\zeta).
\end{cases}
\Ff

\begin{lemma}\label{lem:Dirichlet}
	Let $c\geq 2\sqrt{\bar A}$.
	For any $n\in\N$ and $\zeta\in\R$, the problem \eqref{-nn} admits a unique positive solution 
	$w_{n,\zeta}$. Moreover, $w_{n,\zeta}(x)$ is decreasing in $x$ and satisfies
	$$\forall x\in[-n,n],\quad
	w_{n,\zeta}(x)\leq\psi(x-\zeta).$$
	Finally, the mapping $\zeta\mapsto w_{n,\zeta}$ is continuous with respect to the 
	$L^\infty((-n,n))$ norm.
\end{lemma}

\begin{proof}
	The existence and uniqueness of the decreasing solution for \eqref{-nn} 
	is given by Corollary~\ref{cor:Dp}. 
		To prove the upper bound, we exploit the fact that $w_{n,\zeta}$ is a sub-solution
		of the local equation satisfied by $\psi$. Indeed, it satisfies in $(-n,n)$,
		\[\begin{split}
		-w_{n,\zeta}''-cw_{n,\zeta}' &=
		w_{n,\zeta}\int_{-n}^{x} A(y)(-w_{n,\zeta}'(y))dy\\
		& \leq \bar Aw_{n,\zeta}(\psi(-n-\zeta)-w_{n,\zeta})\\
		&\leq \bar Aw_{n,\zeta}(1-w_{n,\zeta}).
		\end{split}\]
		We can then use the sliding method to deduce that $w_{n,\zeta}\leq\psi(\.-\zeta)$
		on $(-n,n)$.		
		Indeed, if this were not the case, calling $\bar\zeta$ the value for which
		$$\min_{x\in[-n,n]}\big(\psi(x-\bar\zeta)-w_{n,\zeta}(x)\big)=0,$$
		which exists and is unique by monotonicity,
		we would have that $\bar\zeta>\zeta$. Hence, 
		because $\psi(\.-\bar\zeta)>\psi(\.-\bar\zeta)=w_{n,\zeta}$ 
		on the boundary of the interval
		$[-n,n]$, the minimum would be attained at an interior
		point but not on the boundary, contradicting the elliptic \SMP.
		
		For the last statement, consider a sequence $(\zeta_j)_{j\in\N}$
		converging to some $\t\zeta\in\R$.
		Using elliptic estimates up to the boundary, for any subsequence of $(\zeta_j)_{j\in\N}$
		we can extract another subsequence $(\zeta_{j_k})_{k\in\N}$ such that
		the associated $(w_{n,\zeta_{j_k}})_{k\in\N}$ converges
		in $C^2((-n,n))$ to a solution $\t w$ of problem~\eqref{-nn} with $\zeta=\t\zeta$.
		Then, by uniqueness, $\t w=w_{\t\zeta}$. This concludes the proof.
\end{proof}

Here is the existence result.

\begin{proposition}\label{pro:exc>}
	For any $c\geq 2\sqrt{\bar A}$ and $\vt\in(0,1)$, 
	problem \eqref{nlKPP} admits a solution $w$ satisfying $w(0)=\vt$.
\end{proposition}	
	
\begin{proof}	
Let $(w_{n,\zeta})_{n\in\N,\;\zeta\in\R}$ be the family given by Lemma~\ref{lem:Dirichlet},
associated with the standard traveling wave $\psi$ normalized by~$\psi(0)=\vt\in(0,1)$.
We have that
$$w_{n,0}(0)\leq\psi(0)=\vt=w_{n,n}(n)<w_{n,n}(0).$$
Then, by the continuous dependence with respect to $\zeta$, there exists
$\zeta_n\in[0,n)$ such that $w_{n,\zeta_n}(0)=\vt$. 
Using interior elliptic estimates, one sees as in the proof of Lemma~\ref{lem:phi'<0} that
the family $(w_{n,\zeta_n})_{n\in\N}$ is equibounded in $C^{3}(I)$, for any bounded interval $I$. 
Hence, as $n\to\infty$, it converges (up to subsequences) in $C^2_{loc}(\R)$ to some function $w$.
We know that $w$ is nonincreasing and satisfies $0\leq w\leq1$ and $w(0)=\vt$. 
We can pass to the limit in the equation in~\eqref{-nn} using the dominated convergence 
theorem. Namely, recalling that $\zeta_n\geq0$ and $\psi(-\infty)=1$, we infer that~$w$ is
a solution of~\eqref{eq:integrata}
which satisfies $0\leq w\leq1$ and $w(0)=\vt$.
It then follows from Lemma~\ref{lem:phi'<0} that 
$w$ is decreasing and that $w(-\infty)=1$, $w(+\infty)=0$.
Then, integrating by parts the integral in~\eqref{eq:integrata}
we recover a solution of the original problem~\eqref{nlKPP}.   
%
%
\end{proof}

\vskip1em

{\bf Case (B)} $2\sqrt{\ul A}<c < 2\sqrt{\bar A}$.

This range is more interesting since one cannot rely anymore on the comparison with the waves of the (usual, local) KPP equation.  Indeed, unlike what happens in case (A),  now the wave  will no longer satisfy any arbitrary normalization at a given point.
%
%
%

The first ingredient is to find a super-solution, that in the previous section was 
simply given by a wave for a standard KPP equation. 

\begin{lemma}\label{lem:super}
	For any $c>2\sqrt{\ul A}$, the 
	equation~\eqref{eq:integrata} admits a decreasing super-solution~$\psi$ satisfying
	$$\psi(-\infty)=1,\quad \psi(+\infty)=0.$$
\end{lemma}

\begin{proof}
Let $\t A\in(\ul A,\bar A)$ be such that $c>2\sqrt{\t A}$.
Then call $s:=1-\t A/\ol{A}\in(0,1)$ and define 
$$h(u):=\begin{cases} \t Au & \text{if }u\leq s\\
					  \bar Au(1-u) & \text{if }u>s.
		\end{cases}$$
We know that there is a traveling wave solution $\psi$ for the classical KPP equation
with nonlinear term $h$, that is, a decreasing solution of
$$
\begin{cases}
\displaystyle
\psi''+c\psi'+h(\psi)=0,\quad x\in\R\\
\psi(-\infty)=1,\quad \psi(+\infty)=0.
\end{cases}
$$
We normalize it by $\psi(0)=s$.
Let us show that, for $\zeta$ sufficiently large, the function $\psi(\.-\zeta)$ is a super-solution
of~\eqref{nlKPP}, or equivalently of~\eqref{eq:integrata}.
For $x<\zeta$, we have that $\psi(x-\zeta)>s$ and thus 
$$\psi(x-\zeta)\int_{-\infty}^x A(y)(-\psi'(y-\zeta))dy\leq\bar A\psi(x-\zeta)(1-\psi(x-\zeta))=h(\psi(x-\zeta)).$$
This implies that $\psi(\.-\zeta)$ is a super-solution of \eqref{nlKPP} in $(-\infty,\zeta)$,
for any choice of $\zeta$.
On the other hand, if $\zeta>0$, for $x>\zeta$ we find that
\[\begin{split}
\int_{-\infty}^x A(y)(-\psi'(y-\zeta))dy
&=\int_{-\infty}^{\zeta/2} A(y)(-\psi'(y-\zeta))dy+\int_{\zeta/2}^{+\infty}A(y)(-\psi'(y-\zeta))dy\\
&\leq\bar A(1-\psi(-\zeta/2))+A(\zeta/2)\psi(-\zeta/2).
\end{split}\]
The above right-hand side is independent of $x$ and tends to $\ul A$ as $\zeta\to+\infty$.
It follows that, for $\zeta$ large enough, there holds for $x>\zeta$,
$$\psi(x-\zeta)\int_{-\infty}^x A(y)(-\psi'(y-\zeta))dy<\t A \psi(x-\zeta)=
h(\psi(x-\zeta)).$$
Hence, for such values of $\zeta$, the function $\psi(\.-\zeta)$ is a super-solution of 
\eqref{nlKPP} and thus of~\eqref{eq:integrata}.
\end{proof}

The next step is to show that if \eqref{nlKPP}, or equivalently~\eqref{eq:integrata},
admits a decreasing
super-solution then it also admits a solution.
We would like to follow the same strategy as in the previous section, going through the approximating problems \rife{-nn}. 
However, since we cannot use anymore the comparison with the local equation, we will need the following lemma to guarantee 
%
that solutions stay bounded away from~$1$, 
uniformly in $n$.

\begin{lemma}\label{lem:Dirichlet2}
	Let $c>2\sqrt{\ul A}$ and let $\psi$ be a decreasing super-solution of~\eqref{eq:integrata},
	satisfying $\psi(-\infty)=1$, $\psi(+\infty)=0$.
	For any $n\in\N$ and $\zeta\geq0$, the problem \eqref{-nn} admits a unique positive solution 
	$w_{n,\zeta}$. Moreover, $w_{n,\zeta}(x)$ is decreasing in $x$ and there holds
	$$\sup_{n\in\N}w_{n,\zeta}(0)<1.
	$$
	Finally, the mapping $\zeta\mapsto w_{n,\zeta}$ is continuous with respect to the 
	$L^\infty((-n,n))$ norm.
%
\end{lemma}

\begin{proof}
		Firstly, we check that $\psi(\.-\zeta)$ is still a super-solution of~\eqref{eq:integrata}
		for any $\zeta\geq0$. 
		Because of the condition $\psi(-\infty)=1$, it is equivalent to consider the
		equation~\eqref{nlKPP}. 
		Using the monotonicity of both $A$ and $\psi$,
		we see that for $\zeta\geq0$ and $x\in\R$, 
		\begin{align*}-\psi''(x-\zeta)-c\psi'(x-\zeta) &\geq
		\psi(x-\zeta)\int_{-\infty}^{x-\zeta} A(y)(-\psi'(y))dy\\
		&=\psi(x-\zeta)\int_{-\infty}^{x} A(y-\zeta)(-\psi'(y-\zeta))dy\\
		&\geq\psi(x-\zeta)\int_{-\infty}^{x} A(y)(-\psi'(y-\zeta))dy,
		\end{align*}
		that is, $\psi(\.-\zeta)$ is a super-solution of \eqref{nlKPP}.
		We can therefore restrict ourselves to the case $\zeta=0$.
		
		Corollary~\ref{cor:Dp} implies the existence, uniqueness and strict monotonicity
		of the solution to~\eqref{-nn} with $\zeta=0$. We call it $w_n$.  
		Let us show that $(w_n(0))_{n\in\N}$ stays bounded from above away from $1$.
		
		Assume by contradiction that this is not the case.
		Then, up to extraction of a subsequence, we have that 
		$w_n(0)\to1$ as $n\to\infty$.
		We can further assume that, up to another extraction,
		$w_n(0)>\psi(-1)$ for all $n\in\N$. Let~$b_n$ be the smallest intersection
		point between $w_n$ and~$\psi$ on $(0,n]$.
		Then call		
		$$k_n:=\max_{[-n,b_n]}\frac{w_n}{\psi},$$ and let $x_n\in[-n,b_n]$ be a 
		point where such maximum is reached.
		We see that $k_n>\frac{w_n(0)}{\psi(-1)}>1$, whence 
		$x_n$ lies inside the interval $(-n,b_n)$ because $w_n/\psi$
		is equal to~$1$ on the boundary.
		We also see that
		$$\limn\frac{w_n(0)}{\psi(0)}=\frac1{\psi(0)}>\frac1{\psi(-1)}\geq
		\max_{[-n,-1]}\frac{w_n}{\psi}.$$
		This implies that $x_n>-1$ for $n$ large enough.
			The function $g_n:=k_n\psi$ touches $w_n$ from above at the point 
		$x_n$, whence
		\[\begin{split}
		0 &= w_n''(x_n)+cw_n'(x_n)+w_n(x_n)
		\int_{-n}^{x_n} A(y)(-w_n'(y))dy\\
		&\leq g_n''(x_n)+cg_n'(x_n)+
		g_n(x_n)\int_{-n}^{x_n} A(y)(-w_n'(y))dy\\
		&\leq g_n(x_n)\left(\int_{-n}^{x_n} A(y)(-w_n'(y))dy-
		\int_{-\infty}^{x_n} A(y)(-\psi'(y))dy\right),
		\end{split}\]
		where, for the last inequality, we have used that $\psi$ is a super-solution of~\eqref{nlKPP}.
		We deduce that
		$$\int_{-n}^{x_n} A(y)(\psi'(y)-w_n'(y))dy\geq
		\int_{-\infty}^{-n} A(y)(-\psi'(y))dy>0,$$
		and thus, integrating by parts,
		$$A(x_n)(\psi(x_n)-w_n(x_n))+
		\int_{-n}^{x_n} A'(y)(w_n(y)-\psi(y))dy>0.$$
		We recall that $w_n\geq\psi$ in $[-1,0)$ because $w_n(0)>\psi(-1)$,
		as well as in $[0,b_n]$ by the definition of $b_n$. Thus, for $n$ large enough,
		since $x_n\in(-1,b_n)$, we infer that $w_n\geq\psi$ in $[-1,x_n]$ and therefore		
		the above inequality together with $A'\leq0$ yield
		$$\int_{-n}^{-1} A'(y)(w_n(y)-\psi(y))dy>0.$$
		This implies in particular that $A'\not\equiv0$ in $(-\infty,-1]$.		
		Recall, however, that we are assuming that $(w_n)_{n\in\N}$ converges to $1$ at the point $0$,
		hence uniformly in $(-\infty,-1]$. Passing to the limit as $n\to\infty$ in the 
		above integral inequality we then reach a contradiction.
		
		The last statement of the lemma follows from the uniqueness of the solution,
		exactly as in the proof of Lemma~\ref{lem:Dirichlet}.
\end{proof}

\begin{proposition}\label{pro:exc<}
	Problem \eqref{nlKPP} admits a solution for any $c>2\sqrt{\ul A}$.
\end{proposition}

\begin{proof}
Fix $c>2\sqrt{\ul A}$. Let $\psi$ be the super-solution provided
by Lemma~\ref{lem:super} and $(w_{n,\zeta})_{n\in\N,\;\zeta\geq0}$ be the family constructed from
it in Lemma~\ref{lem:Dirichlet2}.
We know from that lemma that there exists $\vt$ satisfying
$$\sup_{n\in\N}w_{n,0}(0)<\vt<1.$$
For given $n\in\N$, using the fact that  $w_{n,\zeta}(0)>\psi(n-\zeta)>\vt$ for $\zeta$ sufficiently large 
(depending on $n$) together with the continuity of
$w_{n,\zeta}$ with respect to $\zeta$, we can find
$\zeta_n>0$ such that $w_{n,\zeta_n}(0)=\vt$.

		By standard elliptic estimates, 
		$(w_{n,\zeta_n})_{n\in\N}$ converges in $C^{2}_{loc}(\R)$ (up to subsequences)
		to some function $0\leq w\leq1$.
		Thus, using the dominated convergence theorem, we can pass to the limit in 
		the equation of~\eqref{-nn} and deduce that $w$ solves~\eqref{eq:integrata}. 
		Finally, because $w(0)=\vt\in(0,1)$, Lemma~\ref{lem:phi'<0} implies that
		$w$ is a solution to~\eqref{nlKPP}.		
\end{proof}


\subsection{The functional $\mc{I}$}

We investigate now more deeply the structure of the set of traveling waves. A key role will be played by the following quantity associated to a  solution  $w$   of \rife{nlKPP}:
\Fi{I}
\mc{I}(w):=\int_{\R}A(y)(-w'(y))dy=
\bar A+\int_{\R}A'(y)w(y)dy\,.
\Ff
Observe that the second formulation of $\mc{I}$, obtained after integration by parts, shows that $\mc{I}$ is decreasing with respect to $w$.

We start by collecting some properties of the traveling waves  which   involve the functional $\mc{I}$.


\begin{proposition}\label{implicit-speed}
Let $w$ be a solution to \eqref{nlKPP} and $\mc{I}$ be given by \rife{I}. Then we~have

\begin{itemize}

\item[$(i)$]  $\ul A<\mc{I}(w)<\bar A$  

\item[$(ii)$]  $\mc{I}(w)\leq\frac{c^2}4,$

\item[$(iii)$] $w$ satisfies
\Fi{<phi<}
\frac{A(0)-\mc{I}(w)}{A(0)-\ul A}\leq w(0)\leq\frac{\bar A-\mc{I}(w)}{\bar A-A(0)}
\Ff
where the inequalities  are understood to hold provided the corresponding 
denominators are not $0$.

\item[$(iv)$] $w$ is strictly $\log$-concave (that is, $w'/w$ is decreasing)
and satisfies 
$$\lim_{x\to+\infty}\frac{-w'}w(x)=\frac c2-\sqrt{\frac{c^2}4-\mc{I}(w)}=:\lambda>0.$$
\end{itemize}
	In particular, there holds that
	$$w(x)=w(0)e^{-\lambda(x)x},$$
	where $\lambda(x)$ is an increasing function converging to $\lambda$ as
	$x\to+\infty$.
\end{proposition}

\begin{proof} 
Since $w'<0$ from Proposition \ref{Harnack}, the bounds $\ul A<\mc{I}(w)<\bar A$ immediately follow, recalling that
$\ul A\leq A\leq\bar A$ and that both inequalities are strict somewhere. 


The estimates \rife{<phi<} on $w(0)$ easily follow from the definition of $\mc{I}$ as well.
Indeed, on one hand, 
$$
\mc{I}=\int_{-\infty}^0 A(y)(-w'(y))dy+\int_0^{+\infty} A(y)(-w'(y))dy
\geq A(0)(1-w(0))+\ul{A}w(0).
$$
On the other hand, an integration by parts shows that
$$\mc{I}=\bar A+\int_{\R}A'(y)w(y)dy\leq 
\bar A+\int_{-\infty}^0A'(y)w(y)dy\leq \bar A+w(0)(A(0)-\bar A).
$$
Now we investigate the properties of $q:=-w'/w$, which is  a positive function.  
Direct computation reveals that
\be\label{qeq}
q'=q^2-cq+\int_{-\infty}^x A(y)(-w'(y))dy,
\quad x\in\R.
\ee
The integral term is positive, nondecreasing in $x$, 
and converges to $0$ as $x\to-\infty$ and to 
$\mc{I}(w)$ as $x\to+\infty$.  
We now show that $q$ is bounded and increasing.
Recall that $c>0$ by Proposition~\ref{Harnack}.

First of all we observe that necessarily  $q(x)\leq c$, because if $q(x_0)>c$ then 
\rife{qeq} would imply that $q$ blows up at some point $x_1>x_0$. 
The boundedness of $q$ then implies that $\mc{I}(w)\leq c^2/4$, because otherwise 
by \rife{qeq} there 
would exist $\vep>0$ such that, for large $x$,
$$
q' > q^2-cq+c^2/4 + \vep \geq \vep,
$$ 
which is impossible because $q$ is bounded.  So we also proved that $\mc{I}(w)\leq c^2/4$.  
%
This allows us to rewrite \rife{qeq} as 
\be\label{eqder}
q'=(q-\lambda_-(x))(q-\lambda_+(x)),
\quad x\in\R,
\ee
with
$$
\lambda_\pm(x):=\frac c2\pm\sqrt{\frac{c^2}4-\int_{-\infty}^x A(y)(-w'(y))dy}.
$$
Observe that $0<\lambda_-(x)<\lambda_+(x)$ and 
$\lambda_-'(x)>0>\lambda_+'(x)$ for all $x\in\R$, with
$$\lambda_-(-\infty)=0<\lambda_-(+\infty)=
\frac c2-\sqrt{\frac{c^2}4-\mc{I}(w)}=:\lambda,
\qquad \lambda_+(-\infty)>\lambda_+(+\infty)>0.$$
We infer that if $q(x_0)\geq\lambda_+(x_0)$ at some $x_0$ then $q\geq\lambda_+(x_0)$ in $(x_0,+\infty)$,
which implies that $q(+\infty)=+\infty$, thus this case is excluded.
On the other hand, if $\lambda_-(x_0)\leq q(x_0)\leq\lambda_+(x_0)$ at some $x_0$ then 
$\lambda_-(x_0)\leq q\leq\lambda_+(x_0)$ in $(-\infty,x_0)$
and thus $q(-\infty)>0$, which is impossible because, being $q=-w'/w$, we would have that
$w(-\infty)=-\infty$.
The unique possibility left is therefore $q<\lambda_-$ in $\R$. 
We deduce from~\rife{eqder} that $q'>0$ and that 
		$$
		q(+\infty)=\lambda_-(+\infty)=\lambda.
		$$
		So the proof of item $(iv)$ is concluded.
		
		For the last statement of the Theorem, we write $w(x)=w(0)e^{-\lambda(x)x},$
		with
		$$\lambda(x)=-\frac1x\log\frac{w(x)}{w(0)}=
		-\frac1x\int_0^x \frac{w'(y)}{w(y)}dy.$$
		The convergence of $w'/w$ towards $-\lambda$
		implies that $\lambda(x)\to\lambda$ as $x\to+\infty$.
		Moreover, from the monotonicity of $w'/w$, we infer that, for $x\neq0$,
		$$\lambda'(x) = -\frac1x\, \frac{w'(x)}{w(x)}+
		\frac1{x^2}\int_0^x \frac{w'(y)}{w(y)}dy
		>0.$$
\end{proof}


We focus now on the case $2\sqrt{\ul A}<c<2\sqrt{\bar A}$. We  seek for a wave 
for which the bound $(ii)$ in 
Proposition~\ref{implicit-speed} is optimal, i.e., such that
$$
\mc{I}(w)=\frac{c^2}4.
$$
This will be called a ``critical wave'' associated with a given speed. Observe that similar waves can only exist in this range of velocities, since   $\mc{I}(w)< \bar A$ by 
Proposition~\ref{implicit-speed}. We are going to show that, for a given velocity $c$, the critical wave runs at the lowest height. 

In order to enlighten this fact, we start to investigate the possible heights   which are admissible at a  given speed $c$.



\begin{proposition}\label{pro:normalizations}({\it same speed, different normalization})
Assume that~\eqref{nlKPP} admits a solution $w$.  Then, for any $x_0\in\R$ and $\vt\in(w(x_0),1)$, there exists  
a solution~$\t w$ of~\eqref{nlKPP} satisfying  $\t w(x_0)=\vt$.
Moreover, the function $\t w/w$ is nondecreasing on $\R$.
\end{proposition}

\begin{proof}
	Let $n\in\N$. For $\zeta\geq0$, we consider the initial value problem
	\Fi{ivp-nzeta}
	\begin{cases}
		\displaystyle
		\psi''+c\psi'+\psi\left(\bar A-A\psi+\int_{-\infty}^{-n} A'w
		+\int_{-n}^{x} A'\psi\right)=0,\quad x>-n\\
		\psi(-n)=w(-n-\zeta)\\
		\psi'(-n)=w'(-n-\zeta).
	\end{cases}
	\Ff
	If $\zeta=0$ then the function $w$ is a solution of this problem.
	If $\zeta>0$, we see that the function $w_\zeta$ defined by $w_\zeta:=w(\cdot-\zeta)$
	is a super-solution of this problem. Indeed, calling 
	$A_\zeta:=A(\cdot-\zeta)$, for $x\in\R$ we have that
	\[\begin{split}
	w_\zeta''+cw_\zeta' & = w_\zeta \int_{-\infty}^{x-\zeta} Aw' = -w_\zeta \left( \bar A- A w_\zeta
	+ (A-A_\zeta)w_\zeta+ \int_{-\infty}^{x-\zeta}A'w \right)\\
	& = -w_\zeta \left(   \bar A- A w_\zeta
	-  w_\zeta \int_x^{x-\zeta} A'+ \int_{-\infty}^{-n}A'w + \int_{-n}^{x-\zeta}A'w \right)
	\\
	& \leq  -w_\zeta \left( \bar A- A w_\zeta
	-    \int_x^{x-\zeta} A' w_\zeta + \int_{-\infty}^{-n}A'w + \int_{-n}^{x-\zeta}A'w_\zeta \right)
	\\
	& =  -w_\zeta \left( \bar A- A w_\zeta+ \int_{-\infty}^{-n}A'w + \int_{-n}^{x}A'w_\zeta \right)
%
	\end{split}\]
	On the other hand, the constant $w(-n-\zeta)$ is a sub-solution of the same problem.
	It~follows from Lemma~\ref{lem:CP}
	that~\eqref{ivp-nzeta} admits a unique solution $\psi^\zeta$,
	which is decreasing and for which the ratio $\psi^\zeta/w_\zeta$
	is nondecreasing in $[-n,+\infty)$, whence in particular $\psi^\zeta\geq w_\zeta$.
	In the case $\zeta=0$ we have $\psi^0\equiv w$.	
	Take $x_0\in\R$ and $\vt\in(w(x_0),1)$.
	There holds that 
	$$\psi^0(x_0)=w(x_0)<\vt,\qquad
	\psi^\zeta(x_0)\geq w(x_0-\zeta)\to1\quad\text{as }
	\zeta\to+\infty.$$
	Thus, the continuous dependence of $\psi^\zeta(x_0)$ with respect to $\zeta$
	yields the existence of some $\zeta>0$ such that 
	$\psi^\zeta(x_0)=\vt$.
	We call $\psi_n$ such function~$\psi^\zeta$. The ratio $\psi^\zeta/w_\zeta$
	is nondecreasing in $[-n,+\infty)$ and equal to $1$ at $-n$.
	Then, writing
	$$\frac{\psi^\zeta}{w}=\frac{\psi^\zeta}{w_\zeta}\,\frac{w_\zeta}w,$$
	and observing that 
	$$\left(\frac{w_\zeta}w\right)'=\frac{w_\zeta'w-w'w_\zeta}{w^2}=
	\frac{w_\zeta}{w}\left(\frac{w_\zeta'}{w_\zeta}-\frac{w'}{w}\right)>0
	$$
	due to Proposition~\ref{implicit-speed}-$(iv)$, we find that 
	$\psi^\zeta/w$ is increasing in $[-n,+\infty)$ and larger than~$1$.
	
	The sequence $\seq{\psi}$ converges (up to subsequences)
	to some function $\t w$ in $C^2_{loc}(\R)$.
	This function is nonincreasing, satisfies 
	$\t w(x_0)=\vt$ and in addition
	$\t w/w$ is nondecreasing in $\R$ and larger than or equal to~$1$.
	We infer that $\t w(-\infty)=1$. 
	For every $x\in\R$, there holds
	$$\t w''+c\t w'+\t w(\bar A-A\t w)=
	-\limn\psi_n\int_{-n}^x A'\psi_n
	=-\int_{-\infty}^x A'\t w,$$
	that is, $\t w$ is a solution of~\eqref{eq:integrata}.
	It then readily follows that $\t w(+\infty)=0$, and thus that~$\t w$ satisfies~\eqref{nlKPP}.
%
%
\end{proof}

\begin{corollary}\label{cor:normalizations}
	Assume that~\eqref{nlKPP} admits a solution $w$.
	Then, for any $x_0\in\R$ and $\e>0$, there exists  
	a solution~$\t w$ of~\eqref{nlKPP} satisfying
	$$\t w(x_0)>w(x_0),\qquad
	w\leq\t w<w+\e.$$
\end{corollary}

\begin{proof}
	Let $\delta>0$ be such that $(1+\delta)w(x_0)<1$.
	Applying Proposition~\ref{pro:normalizations} with $\vt=(1+\delta)w(x_0)$ 
	provides us with a solution $\t w$ such that   $\t w(x_0)= (1+\de)w(x_0)$ and $\t w/w$ is nondecreasing. 
	This yields 
	$$
	\forall x\leq x_0,\quad
	w(x) \leq \t w(x)\leq(1+\delta)w(x)<w(x)+\delta.
	$$
	By Lemma \ref{past-future}, given $\vep>0$, we can choose $\de$ small enough so that 
	$\|w-\tilde w\|_\infty < \vep$. 
\end{proof}

We have now the ingredients to show the critical role played by the equality $\mc{I}(w)=c^2/4$. 

\begin{lemma}\label{lem:critical}
	Assume that \eqref{nlKPP}  admits a solution $w$ for which
	$\mc{I}(w)<c^2/4$. Then~\eqref{nlKPP}  admits another decreasing solution 
	$\psi<w$.
\end{lemma}
Before proving this lemma, let us show how it entails the existence of the critical~wave.
\begin{proposition}\label{pro:critical}
	For any $2\sqrt{\ul A}<c<2\sqrt{\bar A}$, there exists a solution $w$ to~\eqref{nlKPP}
	for which $\mc{I}(w)=c^2/4$.
\end{proposition}

\begin{proof}
Consider the maximization problem
$$j^*:=\sup\{\,\mc{I}(w)\ :\ w \text{ is a solution of \eqref{nlKPP}}\}.$$
We know from Proposition~\ref{implicit-speed} that $\ul A<j^*\leq c^2/4<\bar A$.
Let us show that $j^*$ is attained.
Let $\seq{w}$ be a maximizing sequence for $j^*$.
We use the formulation~\eqref{eq:integrata} for 
the equation satisfied by the $w_n$.
Using the $C^3$ estimate of Lemma~\ref{lem:phi'<0}, as well as the dominated convergence theorem,
we infer that
$\seq{w}$ converges (up to subsequences) in $C^2_{loc}(\R)$ to a 
nonincreasing solution $w$ of~\eqref{eq:integrata}. 
Moreover, the second formulation in~\eqref{I} yields
$$j^*=\limn
\mc{I}(w_n)=
\bar A+\int_{\R}A'(y)w(y)dy.$$
This immediately shows that $w\not\equiv0,1$. Therefore, 
Lemma~\ref{lem:phi'<0} implies that $w$ is a decreasing solution to~\eqref{nlKPP}
and in particular that $\mc{I}(w)=j^*$.

Assume by way of contradiction that $\mc{I}(w)=j^*<c^2/4$.
Then by Lemma~\ref{lem:critical}
there exists another solution $\t w<w$ to~\eqref{nlKPP}. The second 
formulation in~\eqref{I} yields $\mc{I}(\t w)>\mc{I}(w)=j^*$,
contradicting the definition of $j^*$.
\end{proof}

It remains to prove Lemma~\ref{lem:critical}.

\begin{proof}[Proof of Lemma~\ref{lem:critical}] We construct the desired wave in two steps.

{\it Step 1}. As a first step, we show that, for any $\zeta\in\R$ and $k\in(0,1)$,

	there exists a solution $\psi_{\zeta,k}$ of~\eqref{eq:integrata} for
	$x<\zeta$ which satisfies 
	$$
	\forall x\leq\zeta,\quad
	\psi_{\zeta,k}'(x)<0,\quad
	1\geq\frac{\psi_{\zeta,k}}{w}(x)\geq k=\frac{\psi_{\zeta,k}}{w}(\zeta).$$
This is essentially a  consequence of Corollary \ref{cor:Dp}. Indeed, 
for $n\in\N$, $n<\zeta$, we consider the problem
\Fi{ivp-n}
\begin{cases}
	\displaystyle
	\psi''+c\psi'+\psi\left(\bar A-A\psi+\int_{-\infty}^{-n} A' w
	+\int_{-n}^{x} A'\psi\right)=0,\quad x\in (-n,\zeta)\\
	\psi(-n)=w(-n)\,,\quad 
	\psi(\zeta)=k  w(\zeta).
\end{cases}
\Ff
We notice that $\bar A+\int_{-\infty}^{-n} A' w=
A(-n)w(-n)+\int_{-\infty}^{-n} A(-w')>A(-n)w(-n),$, and we use Corollary \ref{cor:Dp} with $W= w$. Since the target is smaller than $w(\zeta)$, we obtain the existence of a unique positive and decreasing solution $\psi^n$ of \rife{ivp-n} with $\frac{\psi^n}{w} $ being decreasing 
on $[-n,\zeta]$, whence
$$\forall x\in(-n,\zeta),\qquad
1>\frac{\psi^n }{w}(x)>\frac{\psi^n }{w}(\zeta)=k.
$$
By elliptic estimates, the $\psi^n$ converge (up to subsequences) as $n\to\infty$,
locally uniformly in $(-\infty,\zeta]$, to a solution~$\psi$ of~\eqref{eq:integrata} for
	$x<\zeta$. Moreover, $\psi$ satisfies $kw\leq\psi\leq w$,
	$\psi(\zeta)=kw(\zeta)$ and 
	$\psi/w$ is nonincreasing on $(-\infty,\zeta]$, whence there holds
	$$0\geq\psi'w-w'\psi.
	$$
This is the function $\psi_{\zeta,k}$ we sought for. 

{\it Step 2}.  Now the purpose is to  extend the function $\psi_{\zeta,k}$ to the whole $\R$. 
	Accordingly to~\eqref{eq:integrata}, we extend them 
%
	as the solutions to the problem
	\Fi{ivp-zeta}
	\begin{cases}
		\displaystyle
		\psi''+c\psi'+\psi\left(\bar A-A\psi+\int_{-\infty}^{\zeta} A'\psi_{\zeta,k}
		+\int_{\zeta}^{x} A'\psi\right)=0,\quad x>\zeta\\
		\psi(\zeta)=\psi_{\zeta,k}(\zeta)\\
		\psi'(\zeta)=\psi_{\zeta,k}'(\zeta).
	\end{cases}
	\Ff
	Lemma~\ref{lem:ivp} gives the existence and uniqueness of the positive solution
	in $(\zeta,\xi_{\zeta,k})$, with either $\xi_{\zeta,k}=+\infty$, or
	$\psi_{\zeta,k}(\xi_{\zeta,k})=0$ or $+\infty$. 
	Our aim is to choose $\zeta\in\R$, $k\in(0,1)$ in such a way that 
	$\xi_{\zeta,k}=+\infty$.
		
	Observing that 
	$$\bar A+\int_{-\infty}^{\zeta} A'\psi_{\zeta,k}=
	\bar A(1-\psi_{\zeta,k}(-\infty))+A(\zeta)\psi_{\zeta,k}(\zeta)
	+\int_{-\infty}^{\zeta} A(-\psi_{\zeta,k}')>A(\zeta)\psi_{\zeta,k}(\zeta),$$
	we deduce that the constant function $\psi_1\equiv\psi_{\zeta,k}(\zeta)$ 
	is a sub-solution of the equation in~\eqref{ivp-zeta}. Hence 
	Lemma~\ref{lem:CP} implies that $\psi_{\zeta,k}$ is decreasing 
	in $[\zeta,\xi_{\zeta,k})$. 
	It satisfies there
	\be\label{23}
	\psi_{\zeta,k}''+c\psi_{\zeta,k}'+\psi_{\zeta,k}
	\left(\bar A+\int_{-\infty}^{\zeta} A'\psi_{\zeta,k}\right)\geq0.
	\ee
	Since $\psi_{\zeta,k}\geq kw$ on $(-\infty,\zeta]$, we find that	
	$$\int_{-\infty}^{\zeta} A'\psi_{\zeta,k}\leq k\int_{-\infty}^{\zeta} A' w
	\to k\int_{\R} A'w\quad\text{as }\zeta\to+\infty.$$
	Therefore, by definition of $\mc{I}$, we have that 
	\be\label{24}
	\bar A+\int_{-\infty}^{\zeta} A'\psi_{\zeta,k} \mathop{\to}^{\zeta \to \infty} \bar A + k (\mc{I} -\bar A)\,.
	\ee
	On account of \rife{23} and  \rife{24}, and since $\mc{I} <\frac{c^2}4$, we can find 
%
	$\zeta$ sufficiently large and $k$ sufficiently close to $1$ so that $\psi_{\zeta,k}$ satisfies
%
	\Fi{I+e}
	\psi_{\zeta,k}''+c\psi_{\zeta,k}'+\frac{c^2}4\psi_{\zeta,k}>0
	\quad \;\text{in }\; [\zeta,\xi_{\zeta,k}).
	\Ff

	Next, we apply Proposition~\ref{implicit-speed}-$(iv)$, which yields
	$$\forall x\in\R,\quad\frac{w'}{w}(x)>-\frac c2.$$
	Using the fact that $\psi_{\zeta,k}$ converges to $w$ as 
	$k\nearrow1$, uniformly in $(-\infty,\zeta]$ and therefore by elliptic estimates
	in $C^1_{loc}((-\infty,\zeta])$ (up to subsequences), 
	we deduce that for $k$ sufficiently close to $1$ (depending on $\zeta$) there holds
	\Fi{log'}
	\frac{\psi_{\zeta,k}'}{\psi_{\zeta,k}}(\zeta)>-\frac c2.
	\Ff
	Summing up, we can pick $\zeta$ large enough and then $k$ close enough to $1$
	in such a way that both \eqref{I+e} and \eqref{log'} hold. Therefore, the function  $q:=-\psi_{\zeta,k}'/\psi_{\zeta,k}$ satisfies 
	$$
	q'<\left(q-\frac c2\right)^2 \;\quad\text{in }\; [\zeta,\xi_{\zeta,k}), \,\, q(\zeta)<c/2.
	$$ 
	It follows that
	$q(x)<c/2$ for all $x>\zeta$, i.e.,
	$$\psi_{\zeta,k}(x)>\psi_{\zeta,k}(\zeta)e^{-\frac c2(x-\zeta)}.$$
	This means that $\psi_{\zeta,k}$ remains positive on the whole $\R$.	
	Namely, it is a nontrivial solution
	of~\eqref{eq:integrata} and therefore it solves~\eqref{nlKPP}
	due to Lemma~\ref{lem:phi'<0}.
\end{proof}
	
\begin{remark}
	We could have considered two other natural optimization problems. 
	Namely, for given $x_0\in\R$,
	$$\vt^*:=\inf\{\,w(x_0)\ :\ w \text{ is a solution of \eqref{nlKPP}}\},$$
	or, for given $\vt\in(0,1)$,
	$$\zeta^*:=\inf\{\,w^{-1}(\vt)\ :\ w \text{ is a solution of \eqref{nlKPP}}\}.$$
	Once shown that these infima are actually minima, it  follows from 
	Lemma \ref{lem:critical} that they are both attained by the critical wave 
	(satisfying $\mc{I}(w)=c^2/4$).
	To show that the minima are attained it is sufficient to verify that 
	$\vt^*>0$, $\zeta^*>-\infty$. 
	For this, we consider some corresponding minimizing sequences $\seq{w}$. 
	By Lemma~\ref{lem:phi'<0}, they converge in $C^2_{loc}$ to solutions
	of~\eqref{eq:integrata}.
	On one hand, if $\vt^*=0$, one would have that the limit is identically equal to~$0$,
	whence, thanks to Proposition~\ref{implicit-speed},
	$$c\geq 2\limn\sqrt{\mc{I}(w_n)}=2\sqrt{\bar A},$$ 
	which is a contradiction.
	On the other hand, if $\zeta^*=-\infty$ then the limit $w$ of the
	minimizing sequence would satisfy $w\leq\vt$. Being a nonincreasing solution to~\eqref{eq:integrata},
	we would necessarily have that $w\equiv0$, whence the same contradiction as before.\\
	Let us point out that we do not know if the optimal waves for the above problems are unique,
	nor whether an optimal wave for a problem is also critical for the same problem with another choice 
	of the parameter, or for a problem of the other type (except of course 
	that $w$ is optimal for the first problem at a point $x_0$ if and only if
	it is optimal for the second problem for the value $\vt=w(x_0)$).	
\end{remark}	

	

We now show the uniqueness of the wave for given $\mc{I}$, a crucial step to prove the ordering of waves.

\begin{proposition}\label{pro:uniquenessI}
	For given $c,j\geq0$ there exists at most one solution of~\eqref{nlKPP}
	such~that
	$$\mc{I}(w)=j.$$
\end{proposition}

\begin{proof}
	The proof consists in showing that two distinct solutions of~\eqref{nlKPP} on which
	the operator $\mc{I}$ coincides	cannot intersect on $\R$. 
	One then concludes because the equivalence of $\mc{I}$ prevents two solutions 
	from being strictly ordered, thanks to~\eqref{I} and the fact that $A'\not\equiv0$.	
	
	Assume by way of contradiction that there exist two distinct solutions $w_1$, $w_2$ 
	such that $\mc{I}(w_i)=j$ and they intersect somewhere in $\R$. They are decreasing
	by Lemma~\ref{lem:phi'<0}.
	Let $x_0\in\R$ be such that $A'<0$ in $x_0$. 
	If $w_1\equiv w_2$ in a neighborhood of~$x_0$ then necessarily, by~\eqref{nlKPP},
	$$\int_{-\infty}^{x_0} A'(y) w_1(y)dy=\int_{-\infty}^{x_0} A'(y) w_2(y)dy.$$ 
	It then follows from Lemma~\ref{lem:ivp} that $w_1\equiv w_2$ on $[x_0,+\infty)$.
	Then, since
	$$\int_{+\infty}^{x} A'(y) w_i(y)dy=\int_{-\infty}^{x} A'(y) w_i(y)dy+\bar A-j,$$ 
	applying Lemma~\ref{lem:ivp} to the functions $p_i(-x)$ entails $w_1\equiv w_2$ on $\R$.
	This means that, up to replacing $w_1$ with $w_2$, there exists a point $x_1$ close to $x_0$
	where 
	\Fi{phi1<phi2}
	w_1(x_1)<w_2(x_1),\qquad A'(x_1)<0.
	\Ff

	We claim that the following conditions hold at some $\xi\in\R$:
	\Fi{ww}
	\int_{-\infty}^{\xi} A'(y)\ul w(y)dy>\int_{-\infty}^{\xi} A'(y) \ol w(y)dy,
	\qquad \ul w(\xi)=\ol w(\xi),\qquad \ul w'(\xi)\geq\ol w'(\xi),
	\Ff
	with either $\ul w= w_1$, $\ol w=w_2$, or  
	$\ul w= w_2$, $\ol w=w_1$.
	
	Suppose that $w_2>w_1$ on $(x_1,+\infty)$. Then, calling $\xi$ the largest intersection point
	between $w_1$ and $w_2$, we deduce from~\eqref{I} and \eqref{phi1<phi2} that
	\[\begin{split}
	\int_{-\infty}^{\xi} A'(y) w_1(y)dy &=j-\bar A-\int_{\xi}^{+\infty}A'(y) w_1(y)dy\\
	&<j-\bar A-\int_{\xi}^{+\infty}A'(y) w_2(y)dy=\int_{-\infty}^{\xi} A'(y) w_2(y)dy.
	\end{split}\]
	Then~\eqref{ww} holds with $\ol w= w_1$, $\ul w=w_2$.
	
	Next, consider the case where $w_1$ and $w_2$ intersect somewhere in $(x_1,+\infty)$, 
	and let~$\xi'$ be the smallest intersection point in $(x_1,+\infty)$.
	If $\int_{-\infty}^{\xi'} A'(y)w_2(y)dy<\int_{-\infty}^{\xi'} A'(y)w_1(y)dy$
	then~\eqref{ww} holds with $\ol w= w_1$, $\ul w=w_2$, $\xi=\xi'$.
	Otherwise, 
	calling 
	$$\xi:=\inf\{x<x_1\ :\ w_1<w_2\text{ in }(x,x_1)\},$$
	we derive, using that $w_1<w_2$ in $(\xi,\xi')$ and that
	$A'(x_1)<0$ with $x_1\in(\xi,\xi')$,  
	\[\begin{split}\int_{-\infty}^{\xi} A'(y)w_2(y)dy
	&\geq\int_{-\infty}^{\xi'} A'(y)w_2(y)dy
	+\int_{\xi'}^\xi A'(y)w_2(y)dy\\
	&>\int_{-\infty}^{\xi'} A'(y)w_1(y)dy
	+\int_{\xi'}^\xi A'(y)w_1(y)dy=\int_{-\infty}^{\xi} A'(y)w_1(y)dy,
	\end{split}\]
	that is,~\eqref{ww} holds with $\ul w= w_1$, $\ol w=w_2$.
	
%
%
	
%
	
	We have shown the claim~\eqref{ww}. Integrating by parts yields 
	$$
		\int_{-\infty}^{\xi} A(y)(-\ul w'(y))dy>\int_{-\infty}^{\xi} A(y)(-\ol w'(y))dy.
	$$
	Recalling that $\mc{I}(\ul w)=\mc{I}(\ol w)$, we eventually get
	\Fi{intx1}
		\int_{\xi}^{+\infty} A(y)(-\ul w'(y))dy<\int_{\xi}^{+\infty} A(y)(-\ol w'(y))dy.
	\Ff
	
	We now rewrite the equation in~\eqref{nlKPP} in
	terms of the function $\psi(x):=w(-x)$, that is, calling $B(x):=A(-x)$,
	$$\psi''-c \psi'+\psi\int_x^{+\infty}B(y) \psi'(y)dy=0.$$
%
	Then, using the condition $\mc{I}(w)=j$
	we obtain
	$$\psi''-c \psi'+\psi\Big(j-\int_{-\infty}^x B(y) \psi'(y)dy\Big)=0.$$
	The functions $\ul \psi(x):=\ul w(-x)$, $\ol \psi(x):=\ol w(-x)$
	are increasing and satisfy this equation, and moreover, due to~\eqref{ww},
	$$\ul \psi(-\xi)=\ol \psi(-\xi),\qquad \ul \psi'(-\xi)\leq \ol \psi'(-\xi).$$
	Condition~\eqref{intx1} rewrites
	\Fi{int-x1}
		\int_{-\infty}^{-\xi} B(y) \ul \psi'(y)dy<\int_{-\infty}^{-\xi} B(y) \ol \psi'(y)dy.
	\Ff
	Let us call $\rho:=\ol \psi/\ul \psi$.
	If $\ul \psi'(-\xi)<\ol \psi'(-\xi)$ then $\rho'(-\xi)>0$. Instead, if 
	$\ul \psi'(-\xi)=\ol \psi'(-\xi)$ then $\rho'(-\xi)=0$, but we deduce from 
	condition~\eqref{int-x1} and the equations
	that $\ul \psi''(-\xi)<\ol \psi''(-\xi)$. Hence, in such case we find that
	$$\rho''(-\xi)=
	\left(\frac{\ol \psi''\ul \psi-\ul \psi''\ol \psi}{\ul \psi^2}-2 \frac{\ul \psi'}{\ul \psi}\rho' \right)(-\xi)=
	\left(\frac{\ol \psi''-\ul \psi''}{\ul \psi}\right)(-\xi)>0.$$
	Therefore, in any case, we have that
	$\rho'>0$ in a right neighborhood of $-\xi$. 
	Let $\t x$ be the first point in $(-\xi,+\infty)$ where
	$\rho'(\t x)=0$, which necessarily exists because $\rho(-\xi)=\rho(+\infty)=1$.
	Calling $k:=\rho(\t x)>1$, there holds that $\ol \psi'(\t x)=k\ul \psi'(\t x)$ and	
	$$0\geq\rho''(\t x)=
	\left(\frac{\ol \psi''\ul \psi-\ul \psi''\ol \psi}{\ul \psi^2} \right)(\t x),$$
	that is, $\ol \psi''(\t x)\leq k\ul \psi''(\t x)$.
	Multiplying the equation satisfied by $\ul \psi$ at $\t x$ by $k$ 
	and using these informations we derive
	\[\begin{split} 0 &= k\ul \psi''-c k\ul \psi'+k\ul \psi\Big(j-\int_{-\infty}^{\t x} B(y) \ul \psi'(y)dy\Big)\\
	&\geq \ol \psi''-c \ol \psi'+\ol \psi\Big(j-\int_{-\infty}^{\t x} B(y) \ul \psi'(y)dy\Big),
	\end{split}\]
	whence, using the equation for $\ol \psi$,
	$$\int_{-\infty}^{\t x} B(y) \ul \psi'(y)dy\geq 
	\int_{-\infty}^{\t x} B(y) \ol \psi'(y)dy.$$
	Owing to~\eqref{int-x1}, this entails
	$$\int_{-\xi}^{\t x} B(y) \ul \psi'(y)dy> 
	\int_{-\xi}^{\t x} B(y) \ol \psi'(y)dy.$$
	We finally recall that $\rho'>0$ and $\rho>1$ in $(-\xi,\t x)$, which imply that there holds
	$$\ol \psi'>\frac{\ul \psi'\ol \psi}{\ul \psi}>\ul \psi'.$$
	This contradicts the above integral inequality.	
\end{proof}


\subsection{Decay estimates of the waves}

We now derive some estimates about the convergence at $\pm\infty$ in terms of the 
value of the function at a given point, say, the origin.
They will be used in the study of the waves for the mean-field system.

\begin{lemma}\label{pminfty}
	Let $w$ be a solution of~\eqref{nlKPP}.
Then we have
	\be\label{neginf}
	1-w(x) \leq (1- w(0)) e^{\gamma_0 x}  \quad \forall x\leq 0, \quad \hbox{where \ 
		$\gamma_0:= \frac{A(0)w(0)}{ \sqrt{\bar A}+c }$}
	\ee
	and
	\be\label{posinf}
	w(x) \leq w(0)\, e^{\tilde \gamma_0 (\frac1c-x)}\quad \forall x>0\,,\quad \hbox{where \
		$\tilde \gamma_0 := \frac{A(0)(1-w(0))}c$.}
	\ee

	In addition, if $w$ is a critical wave, i.e.~$\mc{I}(w):= \int_{\R} A(-w')=\frac{c^2}4\;(<\bar A)$, 
	then there exists a constant $K$, only depending on 
	$\bar A$ and positive lower bounds for $c$ and $1-w(0)$, such that 
	\be\label{decay_crit}
	w(x) \leq K  \,  x e^{- \frac c2  x}\quad \forall x\geq 1\,.
	\ee
\end{lemma}


\begin{proof}
	We start with the behavior for $x\to -\infty$.  Similarly as in the proof of Proposition \ref{implicit-speed}, we introduce the function
	$$
	\psi(x):= - \frac{w'}{1-w}\,.
	$$
	We know from Lemma~\ref{lem:phi'<0} that $\psi$ is a positive function. It satisfies
	\be\label{psieq}
	-\psi' - \psi^2- c\psi + \frac w{1-w } \int_{-\infty}^x A(y)(-w'(y))dy =0\,.
	\ee
	In particular, since $A\leq \bar A$ and $w$ satisfies 
	$w <1$ and $w'<0$, we have
	$$
	-\psi' - \psi^2- c\psi + \bar A \geq 0\,.
	$$
	This implies that
	\be\label{psi1}
	\psi \leq \sqrt{\bar A},
	\ee
	because otherwise we would have $\psi'\leq -c\psi\leq-c\sqrt{\bar A}$ 
	in $(-\infty, \bar x)$ for some $\bar x$, which is not possible.
	Coming back to \eqref{psieq}, using that $A, w$ are non increasing, we also have
	$$
	-\psi' - \psi^2- c\psi + A(0) w(0)\leq 0 \qquad \forall x<0\,.
	$$
	Due to \eqref{psi1}, we deduce that
	$$
	\psi' \geq - \psi \left( \sqrt{\bar A}+c \right) + A(0) w(0)\,, \quad x<0\,.
	$$
	Hence 
	$$
	\big( e^{( \sqrt{\bar A}+c ) x} \psi \big) ' \geq  e^{( \sqrt{\bar A}+c ) x} A(0) w(0)\,, \quad x<0\,.
	$$
	Since $\psi$ is bounded above by~\eqref{psi1}, integrating in $(-\infty,x)$ we deduce 
	$$
	\psi(x) \geq \gamma_0:= \frac{A(0)w(0)}{ \sqrt{\bar A}+c }\,.
	$$
	Recalling that $\psi= -\frac{w'}{1-w}$, we readily derive \eqref{neginf}.
	
	A similar statement can be obtained as $x\to +\infty$.  
	As in the proof of Proposition~\ref{implicit-speed}, here
	we consider the function $q:=-w'/w$, which is positive, bounded, and satisfies~\eqref{qeq}.
	In particular, since $A$ is non increasing, for $x>0$  we deduce
	$$
	q' \geq q^2-cq +  \int_{-\infty}^0 A(y)(-w'(y))dy \geq q^2-cq +A(0)(1-w(0))\,. 
	$$
	Hence, always for $x>0$, 
	$$
	q'(x) \geq -cq +c\tilde \gamma_0\,,\quad\text{where }\,
	\tilde \gamma_0:= \frac{A(0)(1-w(0))}c\,.
	$$
	This implies
	$$
	q(x) \geq q(0)e^{-cx} + \tilde \gamma_0 (1-e^{-cx})\geq
	\t\gamma_0(1-e^{-cx}).$$
	Recalling that $q=-\frac{w'}w$ this readily implies  \rife{posinf}.
	
	\vskip0.4em
	Let us prove the last statement. Suppose that $w$ is a  critical wave. 
	In this case, we rewrite \rife{qeq} as
	\be\label{geq2}
	\begin{split}
	q' & =q^2-cq+\frac{c^2}4- \int_x^{+\infty}  A(y)(-w'(y))dy
	\\ & \geq   q^2-cq+\frac{c^2}4-A(x_0) w(x_0),
	\quad x\geq x_0 
	\end{split}
	\ee
	where $x_0$ is any given point and we used  that $A,w$ are non increasing.

We take now a  number $\beta \in (0,\frac c8)$, and we choose $x_0$ such that 
\be\label{xoo}
A(x_0)w(x_0) \leq c\beta \leq\frac{c^2}4- c \beta.
\ee
Notice that, if $A(0)\leq c\beta$, then we can take $x_0=0$. Otherwise we have $A(0)> c \beta$, and we can use \eqref{posinf} to find a value $x_0>0$, only depending on $\bar A$, $\beta$ and a positive lower bound of $1-w(0)$, such that \rife{xoo} holds. As a first consequence, from \rife{geq2} we deduce 
$$
q'\geq -cq + c\beta\,\quad x\geq x_0  
$$ 
which leads, as before, to the exponential estimate
$$
w(x) \leq w(x_0) e^{-\beta(x-x_0-\frac1c)},\quad \forall x\geq x_0.
$$
Coming  back to \rife{geq2},  now we upgrade it into 
\be\label{ode1}
q'\geq \left(q-\frac c2\right)^2- A(x) w(x) \geq \left(q-\frac c2\right)^2- \bar A\,
e^{-\beta(x-x_0-\frac1c)}\,,\quad x>x_0 
\ee
We set
$$
\zeta:= \frac c2-q + B e^{- \beta  (x-x_0)},
$$
where $B$ is sufficiently large, e.g.~take $B=\frac{\bar A}{\beta}e^{1/8}$, so that $B\beta \geq  \bar A e^{\frac \beta c}$. 
Then we get from \rife{ode1}
$$
-\zeta'   = q'+ B  \beta   e^{-\beta( x-x_0)}\geq \left(q-\frac c2\right)^2 \geq \zeta^2 - 
2B e^{-\beta( x-x_0)}\,  \zeta \,.
$$
Notice that $\zeta$ is a positive function since $q<\frac c2$ due to Proposition \ref{implicit-speed}.
Then $\zeta $ satisfies
$$
\left(\frac1\zeta\right)' \geq 1- 2B e^{-\beta( x-x_0)}\, \frac1{\zeta}\,
$$
and we get, integrating and dropping the term in $x_0$, 
\begin{align*}
\frac1\zeta & \geq  \exp\left( \frac{2B}{\beta}e^{-\beta( x-x_0)} \right) \int_{x_0}^x \exp\left( -\frac{2B }{\beta}e^{-\beta( y-x_0)} \right)dy
\\
& \geq   x-M
\end{align*}
for some  constant $M$  only depending on $x_0$, $\beta$ and $B$ (which only depends on  $\bar A$ and~$\beta$). Finally, for $x$ sufficiently large (e.g.~for $x>M+1$), we have $\zeta \leq \frac1{x-M}$ and this implies,  
by definition of $\zeta$,  that
$$
\frac c2-q  \leq \frac 1{x-M} \qquad \forall x>M+1\,.
$$
Recalling that $q=-\frac{w'}w$, by integration we deduce inequality \eqref{decay_crit}, say for $x>M+1$, but then of course for any $x\geq 1$ as well. The constant $K$
depends on $x_0$, $\beta$, $\bar A$, and then, from the above choices of $x_0$ and $\beta$, the constant depends on $\bar A$ and on  positive lower bounds of $c$ and $1-w(0)$.
%
	%
\end{proof}

\subsection{The whole family of waves}

We are now in a position to characterize 
the whole family of waves for any given speed 
$c>2\sqrt{\ul A}$. The key ingredients are Corollary \ref{cor:normalizations} and 
the uniqueness result for any given value of $\mc{I}$, Proposition~\ref{pro:uniquenessI}.
We recall that the operator~$\mc{I}$ is defined on solutions of~\eqref{nlKPP} by the 
two equivalent formulations in~\eqref{I}.
	

\begin{lemma}\label{lem:I<j}
	Assume that~\eqref{nlKPP} admits a solution $w$.
	Then, for any $\ul A<j<\mc{I}(w)$, there exists  
	a solution~$\t w\geq w$ of~\eqref{nlKPP} satisfying
	$\mc{I}(\t w)=j$.\end{lemma}

\begin{proof}
	Consider the family of waves
	$$\mc{F}_{w,j}:=\{\psi \text{ solution of \eqref{nlKPP}}\ :\ \psi\geq w
	\text{ and }\mc{I}(\psi)\geq j\},$$
	and call
	$$j^*:=\inf_{\psi\in\mc{F}_{w,j}}\mc{I}(\psi).$$
	We have that $j^*\geq j$.
	We now show that $j^*$ is attained.
	Let $(\psi_n)_{n\in\N}$ be a minimizing sequence for $\mc{I}$ on $\mc{F}_{w,j}$.
	This sequence converges (up to subsequences)
	in $C^2_{loc}(\R)$ to a nonincreasing solution 
	$w\leq\psi^*\leq1$ of~\eqref{eq:integrata}.
	We see from~\eqref{I} that 
	\Fi{jm}
	j^*=\limn \mc{I}(\psi_n)=
	\bar A+\int_{\R}A'(y)
	\psi^*(y)dy.
	\Ff
	Because $j^*\geq j>\ul A$, we deduce that $\psi^*\not\equiv1$
	and therefore $\psi^*$ is a solution of \eqref{nlKPP}
	thanks to Lemma~\ref{lem:phi'<0}. There holds in particular that 
	$\mc{I}(\psi^*)=j^*$, that is, $j^*$ is~attained.
	
	Next, we assume by contradiction that $j^*>j$. We apply 
	Corollary \ref{cor:normalizations} and deduce that, for any $\e>0$, there exists  
	a solution~$\t\psi$ of~\eqref{nlKPP} satisfying
	$$\t\psi(0)>\psi^*(0),\qquad 	\psi^*\leq\t\psi<\psi^*+\e.$$
	It follows from the second formulation in~\eqref{I} that
	$$j^*=\mc{I}(\psi^*)>\mc{I}(\t\psi)>j^*+\e(\ul A-\bar A).$$
	We can therefore choose $\e$ small enough in such a way that $\mc{I}(\t\psi)>j$,
	whence $\t\psi\in\mc{F}$, and we obtain a contradiction with the definition of $j^*$.
\end{proof}

\begin{proposition}\label{pro:uniqueness}
	Two distinct solutions of~\eqref{nlKPP} are strictly ordered.
\end{proposition}	

\begin{proof}
	Let $w_1$, $w_2$ be two distinct solutions of~\eqref{nlKPP}.
	They satisfy~$\mc{I}(w_i)>\ul A$ by Proposition~\ref{implicit-speed}.
	Proposition \ref{pro:uniquenessI} entails that $\mc{I}(w_1)\neq\mc{I}(w_2)$.
	Suppose to fix the ideas that $\mc{I}(w_1)>\mc{I}(w_2)$, and	
	assume by contradiction that there exists $x_0\in\R$ such that $w_1(x_0)\geq w_2(x_0)$.
	Then, thanks to Corollary~\ref{cor:normalizations} and
	the second formulation in~\eqref{I}, we can find a solution 
	$\t w_1\geq w_1$ which still fulfils $\mc{I}(\t w_1)>\mc{I}(w_2)$, but in addition
	$\t w_1(x_0)>w_2(x_0)$. 
	We can therefore apply Lemma~\ref{lem:I<j} and obtain another
	solution~$\t w\geq\t w_1$ such that 
	$\mc{I}(\t w)=\mc{I}(w_2)$. This contradicts Proposition~\ref{pro:uniquenessI},
	because $\t w(x_0)\geq\t w_1(x_0)>w_2(x_0)$.	
\end{proof}

Gathering together all previous results, we can derive the characterization of the 
family of traveling wave solutions.

\begin{proof}[Proof of \thm{nlKPP}]
	Problem~\eqref{nlKPP} admits solution if and only if $c>2\sqrt{\ul{A}}$
	due to Propositions~\ref{pro:exc<} and~\ref{implicit-speed}.
	Fix $c>2\sqrt{\ul{A}}$ and let $\mc{F}$ be the family of solutions to~\eqref{nlKPP}.
	We know from Proposition~\ref{pro:uniqueness} that functions in $\mc{F}$ are strictly ordered.
	We can therefore parametrise $\mc{F}$ as follows:
	$$\mc{F}=(w_\vt)_{\vt\in\Theta},$$
	with $w_\vt$ satisfying $w_\vt(0)=\vt$,
	for a suitable set of indeces $\Theta\subset(0,1)$. 
	Proposition~\ref{pro:exc>} yields $\Theta=(0,1)$
	when  $c\geq2\sqrt{\bar{A}}$.
	
	Consider the case $2\sqrt{\ul{A}}<c<2\sqrt{\bar{A}}$. Let $w^c$ be the
	critical wave provided by Proposition~\ref{pro:critical}, that is, satisfying
	$\mc{I}(w^c)=c^2/4$. We know from Proposition~\ref{implicit-speed}
	that this realises the maximum of $\mc{I}$ on the family $\mc{F}$.
	As a consequence, since $\mc{I}$ is decreasing due to the second formulation
	in~\eqref{I} and the functions in
	$\mc{F}$ are strictly ordered,~$w^c$ lies below any other function of~$\mc{F}$.
	This means that $\min\Theta=w^c(0)\in(0,1)$; let us call this value $\vt_c$.
	Proposition~\ref{pro:normalizations} eventually shows that $\Theta=[\vt_c,1)$.
	
	Let us show the continuity of 
	the mapping 
	$$\Theta\ni\vt\mapsto w_\vt\in\mc{F},$$
	equipped with the $L^\infty(\R)$ norm, for any given $c>2\sqrt{\ul{A}}$.
	Consider a sequence $(\vt^n)_{n\in\N}$ converging to some $\t\vt\in\Theta$.
	It follows that $(w_{\vt^{n}})_{n\in\N}$ converges (up to subsequences) locally uniformly
	to a solution~$\t w$ of~\eqref{eq:integrata} satisfying $\t w(0)=\t\vt$.
	By Lemma~\ref{lem:phi'<0}, the function $\t w$ satisfies 
	$\t w(-\infty)=1$, $\t w(+\infty)=0$. This means that $\t w$
	solves~\eqref{nlKPP} and therefore~$\t w=w_{\t\vt}$.
	For any $\e>0$, consider $x_\e>0$ for which 
	$$w_{\t\vt}(-x_\e)>1-\e,\qquad
	w_{\t\vt}(x_\e)<\e,$$
	and, owing to the locally uniform convergence, let $n_\e$ be such that 
	$$\forall n\geq n_\e,\ |x|\leq x_\e,\quad
	|w_{\vt^{n}}(x)-w_{\t\vt}(x)|<\e.$$
	This means that
	$$\forall n\geq n_\e,\ x>x_\e,\quad
	|w_{\vt^{n}}(x)-w_{\t\vt}(x)|<\max\{w_{\vt^{n}}(x_\e),w_{\t\vt}(x_\e)\}<2\e$$
	and likewise 
	$$\forall n\geq n_\e,\ x<-x_\e,\quad
	|w_{\vt^{n}}(x)-w_{\t\vt}(x)|< 1-\min\{w_{\vt^{n}}(-x_\e),w_{\t\vt}(-x_\e)\}<2\e$$
	We have thereby shown that $(w_{\vt^{n}})_{n\in\N}$ converges uniformly to $w_{\t\vt}$.
	
	To complete the proof, it remains to analyse the 
	dependence of the critical waves~$w^c$ with respect to $c$.
%
	Let $\seq{c}$ be a sequence 
	converging to some $\t c\in(2\sqrt{\ul{A}},2\sqrt{\bar A})$.
	Then, $(w^{c_{n}})_{n\in\N}$ converges (up to subsequences)
	locally uniformly
	to a solution~$\t w$ of~\eqref{eq:integrata} with $c=\t c$. 
%
	By dominated convergence, we find that
	\Fi{Itphi}
	\int_{\R}A'(y)\t w(y)dy=\lim_{n\to\infty}\int_{\R}A'(y)w^{c_{n}}(y)dy=
	\lim_{n\to\infty}\mc{I}(w^{c_{n}})-\bar A=\frac{\t c^2}4-\bar A.
	\Ff
%
	Because $\t c\in(2\sqrt{\ul{A}},2\sqrt{\bar A})$, we deduce
	that $\t w\not\equiv0,\,1$ and thus that $\t w$ is a solution to~\eqref{nlKPP}
	due to Lemma~\ref{lem:phi'<0}. Hence~\eqref{Itphi} yields $\mc{I}(\t w)=\t c^2/4$,
	that is, $\t w$ is the critical front $w^{\t c}$.
	The same argument as before show that 
	the convergence of (the subsequence of) $(w^{c_{n}})_{n\in\N}$ 
	towards~$w^{\t c}$ is uniform in space.
	This means that the whole sequence 
	$(w^{c_{n}})_{n\in\N}$ converges uniformly to~$w^{\t c}$.
	
	Finally, the limits in \eqref{vtclimits} are deduced from the bounds~\eqref{<phi<} 
	in Proposition~\ref{implicit-speed}. This is immediate if $\ul A<A(0)<\bar A$.
	Otherwise, we need to apply the inequalities~\eqref{<phi<} at a point $x_0$ where 
	$\ul A<A(x_0)<\bar A$, which imply that
	$$
	w^c(x_0)\nearrow1 \as c\searrow2\sqrt{\ul{A}},\qquad
	w^c(x_0)\searrow0 \as c\nearrow 2\sqrt{\bar A}\,.
	$$
Then we can use Harnack's inequalities \rife{Harna}  to transport these limits at the origin. 	
%
%
\end{proof}	
	
\begin{proof}[Proof of \thm{I}]
	The monotonicity and continuity of the mapping
	$\vt\to\mc{I}(w_\vt)$ immediately follow from~\thm{nlKPP} and the second formulation
	of $\mc{I}$ in~\eqref{I}. The image~$J$ of the mapping is an interval contained in 
	$(\ul A,\bar A)\cap(\ul A,c^2/4]$ and with lower bound $\ul A$,
	due to Proposition~\ref{implicit-speed} and Lemma~\ref{lem:I<j}. 
	Then by Proposition~\ref{pro:critical},
	$J=(\ul A,c^2/4]$ if $2\sqrt{\ul{A}}<c<2\sqrt{\bar A}$.
	
	In the case $c\geq2\sqrt{\bar A}$, we consider a sequence of waves
	$(w_{\vt^{n}})_{n\in\N}$ with $(\vt^{n})_{n\in\N}$ converging to $0$,
	and we deduce from~\eqref{Harnack} that $(w_{\vt^{n}})_{n\in\N}$ converges locally uniformly
	to $0$. It then follows from~\eqref{I} that 
	$(\mc{I}(w_{\vt^{n}}))_{n\in\N}$ converges to $\bar A$.
	This means that $J=(\ul A,\bar A)$ in this case.	
\end{proof}	

A question remains open after \thm{nlKPP}:
can two distinct critical waves intersect? 
We are not able to answer this question in general, but only
in the region where $A$ is local. 


\begin{proposition}\label{pro:leftordering}
	Assume that $A(x_0)=\bar A$. 
	Let $w^{c_1}$, $w^{c_2}$ be the critical waves
associated with $2\sqrt{\ul A}<c_1<c_2<2\sqrt{\bar A}$,
	then $w^{c_1}(x_0)>w^{c_2}(x_0)$.	
\end{proposition}

\begin{proof}
	Assume by contradiction that $w^{c_1}(x_0)<w^{c_2}(x_0)$.
	Then by Theorem~\ref{thm:nlKPP} there exists another wave $\t w$ for~\eqref{nlKPP} with $c=c_1$
	satisfying
	$\t w>w^{c_1}$ on $\R$ and $\t w(x_0)=w^{c_2}(x_0)$.
	Observe that $\t w$ is a super-solution of the equation in~\eqref{nlKPP} with $c=c_2$.
	Then, using the fact that $A\equiv \bar A$ on $(-\infty,x_0]$, one checks that necessarily 
	$\t w>w^{c_2}$ on $(-\infty,x_0)$,   whence $\t w'(x_0)\leq w^{c_2}(x_0)$.
	Therefore, Lemma~\ref{lem:CP} yields $\t w\leq w^{c_2}$ on $[x_0,+\infty)$.
	Then we derive from \thm{I} 
	$$\frac{c_1^2}4\geq \mc{I}(\t w)
	=\bar A+\int_{x_0}^{+\infty}A'(y)\t w(y)dy\geq
	\bar A+\int_{x_0}^{+\infty}A'(y)w^{c_2}(y)dy=\frac{c_2^2}4,$$
	which is a contradiction.
\end{proof}


\section{Traveling waves for the mean field game system}
\label{mfg-proofs}

 We now prove the main result of the paper. This will be the outcome of a thorough analysis on the system of traveling waves \rife{zw-waves}.  
Namely, we are going to provide  necessary and sufficient conditions for the existence of traveling waves, as in the following statement. 

\begin{theorem}\label{main-zw}  Assume that hypotheses \rife{alpha1}--\rife{ro} hold true. Then we have:
	
	\begin{enumerate}[$(i)$]
		
		\item  there are no solutions of \rife{zw-waves}  with $c\leq 2\kappa^2$ nor with $c\geq \alpha(1)+\kappa^2$;
		
		\item   there exists a solution $(c, w, z)$  of \rife{zw-waves} such that $c\in (2\kappa^2, 2\kappa\sqrt{\alpha(1)})$ and 
		$$
		\frac{c^2}4 = \kappa^2\int_{\R} \alpha(\sigma(y)) (-w'(y)) dy\,;
		$$ 
		
		\item for every $c\in [2\kappa\sqrt{\alpha(1)}, \alpha(1)+\kappa^2)$, there exist solutions of \rife{zw-waves} (with arbitrary normalization at any given point).
	\end{enumerate}
\end{theorem}

The three statements of this theorem are separately proved in the next subsections.
Then, in Section~\ref{sec:tw-BGP}, we will eventually show the equivalence between BGP solutions of \rife{system}
and traveling wave solutions of~\rife{zw-waves}. The proof of Theorem~\ref{main} will then be achieved.



\subsection{Preliminary properties and necessary conditions }\label{sec:pp}

In this section we derive some necessary conditions for the existence of waves.
This will enlighten in particular the optimality of the assumptions $\rho>\kappa^2$, $\alpha(1)>\kappa^2$,
so those two conditions   (hypotheses~\rife{alpha1'}, \rife{ro})  will not be assumed to  
hold a priori~here.

First of all, it is convenient to observe that, by the concavity of $\alpha$,
the function $\sigma$ associated with a solution $(z,w)$ of~\eqref{zw-waves} can be computed as
\Fi{s*=}
\sigma(x)=
\begin{cases}
	s\in(0,1) & \displaystyle\text{if }\ \int_{x}^{+\infty}z(y) e^y\, w(y)dy=\frac{e^{x}}{\alpha'(s)}\;,\\
	1 & \displaystyle\text{if }\ \int_{x}^{+\infty}z(y) e^y\, w(y)dy\geq \frac{e^{x}}{\alpha'(1)}\;.
\end{cases}
\Ff
We will see in the next proposition 
that $\sigma\in W^{1,\infty}_{loc}(\R)$, and it is positive and nonincreasing.  Hence, when dealing with the first equation of \rife{zw-waves}, we will be allowed to make use of the results of Section~\ref{nonlocalKPP} 
with $A:=\alpha\circ\sigma$ (and the obvious rescaling
by~$1/\kappa^2$).


\begin{proposition}\label{pro:NC}
	Under the assumptions \eqref{alpha1},\eqref{alpha0},
	\eqref{alpha2},\eqref{alpha3}, 
	problem~\eqref{zw-waves} admits solution only~if 
	$$\rho>\kappa^2 \quad\text{and}\quad
	\alpha(1)>\kappa^2.$$

	Moreover, for any solution $(c,w,z)$, the following properties hold:
	$$2\kappa^2<c<\alpha(1)+\kappa^2,$$ 
	$$z(-\infty)=0,\qquad z'>0\text{ in }\;\R,\qquad z(+\infty)=\frac1{\rho-\kappa^2}\,,$$
	and the associated $\sigma$ belongs to $ W^{1,\infty}_{loc}(\R)$, is nonincreasing and satisfies
	\Fi{s*01}
	\exists x_0\in\R,\quad
	\sigma=1 \ \, \text{in }\, (-\infty,x_0],\quad
	0<\sigma<1 \ \,\text{in }\,(x_0,+\infty),\quad
	\sigma(+\infty)=0.
	\Ff
	In particular, we have that $A: =\alpha\circ\sigma\in W^{1,\infty}_{loc}(\R)$ is positive and nonincreasing.
	\end{proposition}

\begin{proof}
	%

	Assume that~\eqref{zw-waves} admits a solution $(c,w,z)$.
	 We preliminarily observe that  $z>0$: indeed,
	being $z$ a super-solution of a linear elliptic equation, the strong maximum principle implies that
	either $z>0$ or $z\equiv0$. But in the latter case, the equation itself 
	yields $\sigma\equiv1$, while, from~\eqref{s*=}, we get
	$\sigma\equiv0$. 
	The strong maximum principle also yields $w>0$.
	We now derive the properties stated in the proposition separately.
	
	\smallskip
	{\em Properties of $\sigma$.}\\
	Owing to the characterization~\eqref{s*=} for the function $\sigma$, properties $zw>0$ and $\alpha'(0)=+\infty$
	entail that $\sigma$ is strictly positive.
	Moreover,~\eqref{s*=} also implies that~$\sigma$ is
	 nonincreasing, because $\alpha''<0$, that $\sigma(x)=1$ for $-x$ large enough, because
	$\alpha'(1)>0$, and that $\sigma(+\infty)=0$, because of the
	condition $z e^x w\in L^1(\R)$ in~\eqref{zw-waves}.
	This proves~\eqref{s*01}. Finally, from \rife{s*=} we deduce, using the regularity of $\alpha(s)$:
\be\label{deriv}
\begin{split}
\sigma'(x) & = \frac1{\alpha''(\sigma (x))}\left( \frac{e^x}{\int_x^{+\infty} e^yz w   \,dy}+   z(x)w(x)\left(\frac{e^x}{\int_x^{+\infty} e^yz w   \,dy}\right)^2 \right)
\\ &  =  \frac{\alpha'(\sigma(x))}{\alpha''(\sigma(x))}\left( 1+    z(x)w(x)\alpha'(\sigma(x)) \right) \qquad \forall x\,:\, \sigma (x)<1\,.
\end{split}
\ee
Since $z,w$ are locally bounded, and $\alpha''(s)<0$, we deduce that $\sigma'(x)$ is locally bounded in the interval $(x_0, +\infty)$ where $0<\sigma (x)<1$, and it admits a finite limit as $x\to x_0^+$. Hence $\sigma\in W^{1,\infty}_{loc}(\R)$. We notice indeed that $\sigma$ is piecewise $C^1$ but it is not differentiable at $x_0$,  because $\lim_{x\to x_0^+}  \sigma'(x) <0$. 
The regularity of $\alpha$ then yields $A=\alpha\circ \sigma\in W^{1,\infty}_{loc}(\R)$.
	
	\smallskip
	{\em The condition $\rho>\kappa^2$.}\\
	Integrating the equation for $z$ in~\eqref{zw-waves} in an interval $(x,y)$ yields
	\Fi{intz}
	\kappa^2 z'(x)-\kappa^2 z'(y)=(c-2\kappa^2 ) (z(x)-z(y)) + \int_{x}^{y}
	\big[1-\sigma+\big(\kappa^2-\rho- \alpha(\sigma)w\big) z\big].
	\Ff
	Supposing by contradiction that $\rho\leq\kappa^2$, using that $\sigma(+\infty)=0$, we find
	that the term under the integral satisfies 
	$$\liminf_{y\to+\infty}\big[1-\sigma+\big(\kappa^2-\rho- \alpha(\sigma)w\big) z\big]\geq1,$$
	hence~\eqref{intz} yields $z'(y)\to-\infty$ as $y\to+\infty$, contradicting the
	boundedness of~$z$. 

	\smallskip
	{\em Properties of $z$.}\\
		Now that we know that $\rho>\kappa^2$, we infer from~\eqref{s*01} 
	that the term under the integral in~\eqref{intz} satisfies 
	$$\limsup_{x\to-\infty}\big[1-\sigma+\big(\kappa^2-\rho- \alpha(\sigma)w\big) z\big]\leq
	(\kappa^2-\rho)\liminf_{x\to-\infty} z(x).$$
	Hence, if $z(x)$ does not tend to $0$ as $x\to-\infty$, using the fact that $z$ is uniformly continuous 
	(by elliptic estimates) we obtain by~\eqref{intz} 
	the contradiction $z'(x)\to-\infty$ as $x\to-\infty$.
	This proves that $z(-\infty)=0$.
	
	Next, the properties of the function $\sigma$ derived above
	allow us to apply the results of Section~\ref{nonlocalKPP}
	with  $A:=\alpha\circ \sigma$. In particular, Proposition~\ref{Harnack} asserts that $w$ is decreasing.
	Then, differentiating the equation for $z$ in~\eqref{zw-waves}, and using that $\sigma$, $A$, $w$ are
	nonincreasing, we find that $z'$ satisfies
	\Fi{eqz'}
	-\kappa^2 (z')''+ (c-2\kappa^2 ) (z')' + (\rho-\kappa^2) z'+ A(x)w\, z' \geq 0,\quad x\in\R.
	\Ff
	Moreover, being~$z$ bounded, there exist two sequences $\seq{x^\pm}$ 
	diverging to $\pm\infty$ respectively, such that $z'(x^\pm_n)\to0$ as $n\to\infty$. 
	Hence, applying the weak maximum principle in the intervals $(x_n^-,x_n^+)$
	and letting $n\to\infty$, we deduce that $z'\geq0$ in $\R$.
	Next, the strong maximum principle yields $z'>0$, because otherwise $z\equiv z(-\infty)=0$,
	while we know that $z>0$. 
	
	The monotonicity and boundedness of $z$ imply that $z(x)$ converges to a positive limit $z(+\infty)$
	as $x\to+\infty$. Due to elliptic estimates, the convergence holds in $C^2_{loc}$ and thus we
	deduce that the value $z(+\infty)$ satisfies 
	$$(\rho-\kappa^2) z(+\infty)=-\lim_{x\to+\infty}\big(A(x)w\, z- 1+\sigma(x)\big)=1.$$

	\smallskip
	{\em The condition $\alpha(1)>\kappa^2$ and the bounds $2\kappa^2<c<\alpha(1)+\kappa^2$.}\\
	First of all,
	Proposition~\ref{Harnack} yields $c>0$ and 
	Proposition~\ref{implicit-speed}-$(iv)$ yields
	$$w(x)\geq w(0)e^{-\lambda x}\quad \forall x\geq0,$$
	with
	$$
	\lambda=\frac1{\kappa^2} \left(\frac c2-\sqrt{\frac{c^2}{4 }-\kappa^2\mc{I}(w)}\right), 
	$$
	where the operator ${\mc I}$ is defined in \rife{I}. 
	If~$c\leq2\kappa^2$, we see that $\lambda\leq1$.
	If $c\geq\alpha(1)+\kappa^2$, the same conclusion follows from
	the inequalities $\mc{I}(w)\leq\bar A:=\alpha(1)$ (and 
	$\mc{I}(w)\leq \frac{c^2}{4\kappa^2}$) provided by Proposition~\ref{implicit-speed}, indeed:
	$$
	\lambda 
	\leq \frac1{\kappa^2} \left(\frac c2-\sqrt{\frac{c^2}{4 }-\kappa^2\alpha(1)}\right)\leq
	\frac1{\kappa^2} \left(\frac c2-\sqrt{\frac{c^2}{4 }-c\kappa^2+\kappa^4}\right)=1.
	$$
	Hence, in both cases, we find that
	$$w(x)\geq w(0)e^{- x}\quad \forall x\geq0.$$
	Since $z$ does not tend to $0$ at $+\infty$, because $z'>0$, the condition
	$ze^x w\in L^1(\R)$ in~\eqref{zw-waves} is violated. 
	As a byproduct, we have shown that necessarily $\alpha(1)>\kappa^2$.
\end{proof}

The first statement of Proposition~\ref{pro:NC} proves Theorem~\ref{main-zw}-$(i)$.

Next, we derive a pointwise lower bound for $z$ only using that it satisfies
the equation in~\eqref{zw-waves}, i.e.,
\Fi{eq:z}
-\kappa^2 z''+ (c-2\kappa^2) z' + (\rho-\kappa^2) z+ A(x)w\, z= 1-\sigma(x),
\Ff
for some given functions $A,w\in L^\infty(\R)$.

\begin{lemma}\label{lem:z>}
	Let $z$ be a nonnegative solution of~\eqref{eq:z} in $[0,2]$,
	with $c\in\R$, $\rho\geq \kappa^2$ and $A,w$ nonnegative and bounded.
	Then
	$$z(1)\geq \frac{1-\sup_{(0,2)}\sigma}{C(1+\kappa^2+|c|+\rho+\sup A\sup w)},$$
	for some universal constant $C>0$.
\end{lemma}

\begin{proof}
	It is sufficient to find a positive sub-solution
	of~\eqref{eq:z} on~$(0,2)$ vanishing at the boundary.
	This is simply provided by $\ul z:=1-(x-1)^2$. It satisfies  on $(0,2)$,
	$$- \kappa^2 \ul z''+ (c-2\kappa^2)\ul z' + (\rho-\kappa^2)\ul z+ A(x)w\,\ul z\leq
	C(1+\kappa^2 + |c|+\rho+\sup A\sup w),$$
	for some universal constant $C>0$.
	Hence, calling
	$$k:=\frac{1-\sup_{(0,2)}\sigma}{C(1+ \kappa^2+ |c|+\rho+\sup A\sup w)},$$
	and supposing that $\sup_{(0,2)}\sigma<1$ (otherwise the result trivially holds),
	we have that $k\ul z$ is a sub-solution
	of~\eqref{eq:z} on $(0,2)$. Observe that the zero-th order coefficient of this equation is
	nonnegative. Thus, 	the standard maximum principle yields
	$$z(1)\geq k\ul z(1)=k.$$
\end{proof}

We conclude this subsection with a stability lemma for problem~\eqref{nlKPP} that will often be used in the sequel.
%
Following the terminology employed for the nonlocal KPP equation, we say that a solution to~\eqref{nlKPP} 
is {\em critical} if
\be\label{w-critical}
{\mc I}(w):= \int_{\R} A(y)(-w'(y))dy=\frac{c^2}4\,.
\ee


\begin{lemma}\label{lem:stability}
	Let $(c_j)_{j\in\N}$ be a positive sequence, 
	$(A_j)_{j\in\N}$ be a sequence of equi-bounded, nonincreasing, 
	nonnegative functions in $W^{1,\infty}_{loc}(\R)$ satisfying 
	$$
	A_j(-\infty)>A_j(+\infty) \quad \forall j\in\N,
	$$
	and for $j\in\N$, let $w_j$ be a solution to~\eqref{nlKPP} 
	with $c=c_j$ and $A=A_j$.
	Assume that~$(c_j)_{j\in\N}$ converges to some $c>0$, that $(A_j)_{j\in\N}$
	converges pointwise to some function $A$ satisfying 
	$A(0)>0$ and that $(w_j(0))_{j\in\N}$ converges to some value in $(0,1)$.
	Then $(w_j)_{j\in\N}$ converges uniformly towards a decreasing solution $w$ of~\eqref{nlKPP}.

	In addition, if the $w_j$ are critical (in the sense of~\eqref{w-critical}) 
	then $w$ is critical~too.
\end{lemma}

\begin{proof}
	Let $\bar A>0$ be such that $A_j\leq\bar A$ for all $j\in\N$.
	By Lemma~\ref{lem:phi'<0} $(w_j)_{j\in\N}$ converges in $C^2_{loc}(\R)$, up to subsequences, towards a function $w$.
	Moreover, since $A_j(0)\to A(0)>0$, Lemma~\ref{pminfty}
	implies that the convergence is also uniform in~$\R$, with $1-w_j(x)\leq K e^{\gamma x}$ and 
	$w_j(x)\leq K e^{-\gamma x}$ for
	some positive $K,\gamma$ independent of $j$. 
	Hence, since the $w_j$ are decreasing by Proposition~\ref{Harnack},
	we find that
	$$\bigg|\int_{-\infty}^x A_j(y)(-w_j'(y))dy\bigg|\leq \bar A(1-w_j(x))\leq\bar A K e^{\gamma x},$$
	which is arbitrarily small up to choosing $-x$ very large. As a consequence, 
	we deduce that, up to subsequences,
	$$\int_{-\infty}^x A_j(y)(-w_j'(y))dy\to \int_{-\infty}^x A(y)(-w'(y))dy,$$
	and thus that $w$ solves~\eqref{nlKPP}. By Theorem~\ref{thm:nlKPP} (or classical
	result if $A$ is constant) there is a unique wave
	for such problem which fulfills $w(0)=\lim_{j\to\infty}w_j(0)$. This shows that the whole sequence
	$w_j$ converges uniformly to $w$. We know from Proposition~\ref{Harnack} that $w$ is decreasing.

	The fact that the criticality condition~\eqref{w-critical} is preserved is obtained by 
	estimating the integral at $+\infty$ using that $w_j(x)\leq K e^{-\gamma x}$.
\end{proof}


\subsection{The critical wave}

We now turn to the proof of Theorem~\ref{main-zw}$(ii)$.
It is divided into two main parts: we first build an approximated  solution $(z_n,w_n,c_n)$ (for a suitably truncated problem) through a fixed point argument, and secondly we pass to the limit on the mean field game system in order to get a solution.

We assume here for the sake of simplicity that $\kappa=1$; the general case is treated in the same way with the
obvious~modifications.

\subsubsection*{Part I.  The approximated problem}
\vskip-3pt
For $n\in\N$, we consider the following approximated problem: 
\be\label{system-nn}
\begin{cases}
\displaystyle
	w''+c w'+w   \int_{-\infty}^x A(y)(-w'(y))dy =0,\quad x\in\R
	\\
	w> 0,\quad w(-\infty)=1,\quad w(+\infty)=0
	\\
	\m
	\displaystyle
	-z''+ (c-2) z' + (\rho-1) z+ A \, w\, z= 1-[\sigma]_n,\quad x\in (-n,n) \phantom{\int_{-\infty}^x} 
	\\
	z=0 \ \text{ in }\;(-\infty,-n], \quad z=\frac1{\rho-1} \ \text{ in }\;[n,+\infty)
	\\
	\m
	\displaystyle [\sigma]_n(x)= \mathop{{\rm argmax}}_{s\in [0,1]} \,  
	\left\{ (1-s)e^x+\alpha(s)\int_x^{+\infty} z(y) e^y\, w(y)\chi_n(y)dy\right\},\quad A:= \alpha([\sigma]_n)\,,
\end{cases}
\ee
where
$$
\chi_n(y)=\min\{1,e^{-2(y-n)}\}\,.
$$
The role of $\chi_n$ is to prevent the case $A\equiv\alpha(1)$, which in the original problem~\eqref{zw-waves}
is excluded by the condition $z e^x w\in L^1(\R)$;
it actually guarantees that  $[\sigma]_n(x),A(x)\to0$ as $x\to+\infty$,
uniformly with respect to $c,w,z$. 

We will find a solution to~\eqref{system-nn} through a fixed point argument on the function~$[\sigma]_n$.
The latter is characterized by
\Fi{s*n=}
\forall s\in(0,1),\qquad
[\sigma]_n(x)=s \ \iff \ \int_{x}^{+\infty}z(y) e^y\, w(y)\chi_n(y)dy=\frac{e^{x}}{\alpha'(s)}.
\Ff
From this we immediately see that $[\sigma]_n$ is nonincreasing and satisfies $[\sigma]_n(x)=1$ for~$-x$ large
and $[\sigma]_n(+\infty)=0$.
We therefore look for the fixed point in the set
\begin{align*}
X:=  \{ \sigma\in W^{1,\infty}_{loc}(\R)\,:\,  
 \hbox{$\sigma$ is nonincreasing, $\sigma(x)=1$ for $-x$ large, $\sigma(+\infty)=0$}\}.
\end{align*}
We will actually restrict at some point to a closed subset of $X$ with respect to the~$L^\infty(\R)$ topology.
%

We now define an operator $T_n$ on $X$.
Given $\sigma \in X$, we set $A:=\alpha\circ\sigma$ and we call $w_c$ the (unique) critical
wave with speed $c\in(0,2\sqrt{\alpha(1)})$ associated with $A$, provided by
Theorem~\ref{thm:nlKPP} ($\bar A=\alpha(1)$). Then we choose 
a speed $c$ through a suitable normalization condition expressed in terms of the functions%
$$
\ul w_c(x):=\inf_{0<c'\leq c}w_{c'}(x).
$$
Namely, we claim that there exists a unique $c\in(0,2\sqrt{\alpha(1)})$ such that
\Fi{normalization}
\int_{-\infty}^0e^y\, \ul w_{c}(y)dy=\frac12.
\Ff
It is clear that the $\ul w_c$ are nonincreasing with respect to $c$.
Moreover, thanks to Proposition~\ref{pro:leftordering},
for any $\ul x\in\R$ where $A(\ul x)=\alpha(1)$ there holds that $w_c(\ul x)$ is decreasing with respect to $c$.
It follows that $\ul w_c(\ul x)=w_c(\ul x)$ and that this value is 
strictly~decreasing with respect to $c$.
As a consequence, the mapping $G:(0,2\sqrt{\alpha(1)})\to\R$ defined by%
$$
G(c):=\int_{-\infty}^0e^y\, \ul w_c(y)dy
$$
is decreasing. In addition, one checks that it is continuous using dominated convergence and the last part of
Theorem~\ref{thm:nlKPP}.
Furthermore, by the properties of $X$, there exists $x_0\in\R$ such that 
$0=A(+\infty)<A(x_0)<A(-\infty)=\alpha(1)$.
Therefore, owing to the criticality condition
$$\mc{I}(w_c):=\int_{\R}A(y)(-w_c'(y))dy=\frac{c^2}4,$$
the two inequalities in~\eqref{<phi<} (applied in 
$x_0$ by translation of the coordinate system, and with $\ul A=0$, $\bar A=\alpha(1)$)
imply that 
$w_c(x_0)\to1$ as $c\searrow 0$  and  
$w_c(x_0)\to0$ as $c\nearrow 2\sqrt{\alpha(1)}$.
We then infer from Harnack's inequalities~\eqref{Harna} 
that these convergences hold true locally uniformly in~$\R$.
The first one then implies that $G(c)\to1$ as $c\searrow 0$, while the second one that
$$G(c)\leq \int_{-\infty}^0e^y\, w_c(y)dy\to0\quad\text{as $c\nearrow 2\sqrt{\alpha(1)}$}.
$$
As a consequence, there exists a unique $c\in(0,2\sqrt{\alpha(1)})$ such that $G(c)=1/2$, i.e.~\rife{normalization} holds true.
This normalization determines the choice of $c$ employed to define~$T_n$. 

Now, given the above speed $c$ and the associated critical wave $w_c$, we consider the solution~$z$ of the second equation in~\eqref{system-nn} with~$\sigma$ in place of $[\sigma]_n$ and 
the prescribed  exterior conditions,
which classically exists and is unique.
The outcome~$T_n(\sigma)$ is the function $[\sigma]_n$ generated by $w_c$ and $z$
as indicated in the last line of~\rife{system-nn}. Summing up, the operator $T_n$ works as follows:
\begin{align*}
\sigma \in X & \ \rightsquigarrow \ A:= \alpha\circ\sigma \ \ \rightsquigarrow
 \ \ (c,w_c): \begin{cases} w''+ c w'+w   \int_{-\infty}^x A(y)(-w'(y))dy =0 
	\\
	w(-\infty)=1,\quad w(+\infty)=0, \quad \int_{\R} A(-w')= \frac{c^2}4 \\ 
	\int_{-\infty}^0e^y\, \ul w_{c}(y)dy=\frac12 \end{cases}  
	\\
	\noalign{\medskip}
	& \ \rightsquigarrow \ z:
	\begin{cases}
	-z''+ (c-2) z' + (\rho-1) z+ A \, w_c\, z= 1-\sigma,\quad x\in (-n,n) \phantom{\int_{-\infty}^x} 
	\\
	z=0 \ \text{ in }\;(-\infty,-n], \quad z=\frac1{\rho-1} \ \text{ in }\;[n,+\infty)
	\end{cases}
	\\ \noalign{\medskip} & 
\ \rightsquigarrow \ T_n(\sigma):= [\sigma]_n(x)= \mathop{{\rm argmax}}_{s\in [0,1]} \,  
	\left\{ (1-s)e^x+\alpha(s)\int_x^{+\infty} z(y) e^y\, w_c(y)\chi_n(y)dy\right\}.
\end{align*}
In the following lemma, we  prove the existence of a fixed point for $T_n$.

\begin{lemma}\label{lem:truncated}
The operator $T_n$ has a fixed point in a closed subset $\tilde X$ of $X$.
 
As a consequence, problem \rife{system-nn} admits 
a solution $(c,w,z)=(c_n, w_n,z_n)$ with $c_n\in (0,2\sqrt{\alpha(1)})$,
$z_n$ increasing in $(-n,n)$, $w_n$ decreasing and satisfying in addition the following properties:

$(i)$ $w_n$ is a critical wave corresponding to $c_n$, i.e.
\be\label{crin}
\int_{\R} A(-w_n') dy= c_n^2 /4;
\ee 

$(ii)$ there exists a constant $\ul\vt>0$, only depending on $\alpha(1)$, such that
\Fi{wn0>}
w_n(0)\geq\ul\vt>0 \quad \forall n\in\N;
\Ff

$(iii)$ the following normalization condition  holds true:
$$
\int_{-\infty}^0e^y\, \ul w_{c_n}(y)dy=\frac12,
\qquad \hbox{ where $\,\, \ul w_{c_n}(x):= \inf\limits_{0<c'\leq c_n} w_{c'}(x)$.}
$$
\end{lemma}

\begin{proof}
The fixed point will be obtained as a consequence of Schauder's theorem. 
Some preliminary observations are in order, concerning the functions $w_c,z$ associated with the 
definition of $T_n(\sigma)$ (c.f.~the previous scheme).
We start with the monotonicities. We know from Proposition \ref{Harnack} 
that $w_c$, as well as any other wave for the KPP equation, is decreasing. On the other hand,
being the constant functions $0$ and $\frac1{\rho-1}$ respectively a sub and a super-solution 
of the equation for $z$, 
the maximum principle yields $0\leq z\leq\frac1{\rho-1}$ in $(-n,n)$; then 
$z'(\pm n)\geq0$ and thus, applying the strong maximum principle to~$z'$, which satisfies~\eqref{eqz'}, 
we infer that $z'>0$ in $(-n,n)$.

Next, we point out a  lower bound for $w_c$. This is a crucial consequence of~\eqref{normalization}, which implies that 
$$
\frac12\leq\frac14+\int_{-\ln4}^0\ul w_c(y)dy\leq \frac14+(\ln4)\ul w_c(-\ln4)
\leq \frac14+(\ln4)w_c(-\ln4).
$$
Hence, by Harnack's inequality \rife{Harna}, for any $R>0$ there exists a positive constant~$C_R$, 
only depending on $R$ and $\alpha(1)$ (recall that $c\in(0,2\sqrt{\alpha(1)})$), such that 
\be\label{lowc}
w_c(x)\geq C_R>0 \ \quad\forall x\in[-R,R].
\ee
In particular, \eqref{wn0>} holds.

We now identify the compact set $\t X\subset X$ where to apply Schauder's theorem and 
we separately check its hypotheses.
For simplicity, we drop hereafter the index of~$T_n$.

\smallskip
{\em The invariant convex, compact set $\t X\subset X$.}\\
Our goal is to show that $T(X)$ is contained in a compact subset of $X$ with respect to the $L^\infty(\R)$ norm; 
this will be our~$\t X$.
Consider as before the functions $w_c,z$ associated with $T(\sigma)$.
We preliminarily observe that 
elliptic boundary estimates imply that the~$C^1$ norm of $z$ is controlled in terms of
$\alpha(1)$ and $\rho$, and thus there exists a 
constant~$\delta\in(0,1)$, depending on $\alpha(1)$,~$\rho$, such that
$z\geq\frac1{2(\rho-1)}$ in $[n-\delta,n]$.
Because of this and the lower bound~\eqref{lowc} for $w_c$, 
we find that, for $x<n-\delta$,
$$e^{-x}\int_{x}^{+\infty}z(y) e^y\, w_c(y)\chi_n(y)dy\geq
e^{-x}\int_{n-\delta}^{n}z(y) e^y\, w_c(y)\chi_n(y)dy\geq
e^{-x}\frac{C_n\delta}{2(\rho-1)},$$
which is larger than $1/\alpha'(1)$ for $x$ smaller than some $x_n$ (possibly smaller than~$-n$)
only depending on $n,\alpha(1),\rho$. 
Owing to the characterization~\eqref{s*n=}, 
we derive
$$T(\sigma)=1\quad\text{on }\;(-\infty,x_n].$$

We are left to show the regularity of $[\sigma]_n= T(\sigma)$ and the uniform estimate as~$x\to+\infty$. 
For these, we rewrite~\eqref{s*n=} as
\be\label{s*n}
[\sigma]_n(x)= (\alpha')^{-1}\left(\frac{e^x}{\int_x^{+\infty} z e^y w_c \chi_n \,dy}\right)
 \qquad \forall x\,:\, [\sigma]_n(x)<1\,.
\ee
Since $(\alpha')^{-1}$ is decreasing, this shows from one hand that $[\sigma]_n$ is strictly positive and nonincreasing, and from the other, using $zw_c\leq\frac1{\rho-1}$, that
$$
[\sigma]_n(x)\leq\omega_n(x):= (\alpha')^{-1}\left(\frac{e^x(\rho-1)}{\int_x^{+\infty} e^y \chi_n(y)dy}\right)
 \qquad \forall x>y_n\,,
$$
where $y_n$ is the unique point where the right-hand side is equal to $1$. 
Observe that~$\omega_n(x) \to 0$ as $x\to+\infty$ because $\alpha'(0)=+\infty$. 

As for the regularity, differentiating~\eqref{s*n} we obtain (exactly as in \rife{deriv})
\be\label{deriv2} 
[\sigma]_n'(x)   =   \frac{\alpha'([\sigma]_n(x))}{\alpha''([\sigma]_n(x))}\left( 1+  \chi_n(x) z(x)w_c(x)\alpha'([\sigma]_n(x)) \right) \quad \forall x\,:\, [\sigma]_n(x)<1\,.
\ee
Then, the strict concavity of $\alpha$ implies that $[\sigma]_n$ is  a locally Lipschitz continuous function,
whose $W^{1,\infty}$ norm remains bounded as long as $[\sigma]_n$ stays bounded away from $0$.
The lower bound follows from~\eqref{s*n}, namely, for any $x\geq n$ where $[\sigma]_n(x)<1$, 
$$[\sigma]_n(x)\geq(\alpha')^{-1}\left(\frac{e^x}{\int_{x}^{x+1} z e^y w_c \chi_n \,dy}\right)\geq
(\alpha')^{-1}\left(\frac{(\rho-1)e^{2(x+1-n)}}{w_c(x+1)}\right),
$$
hence, by \rife{lowc}, $[\sigma]_n(x)\geq C(x)$, where $C(x)$ is a positive 
decreasing function depending on $n,\alpha,\rho,\alpha(1)$. We deduce the existence of another positive 
decreasing function
$\t C(x)$, depending on the same terms, such that
$\|[\sigma]_n\|_{W^{1,\infty}(-\infty,x)}\leq \t C(x)$.
%
Summing up, we have seen that 
\begin{align*}
T(X)\subset  \t X:=\{ \sigma\in W^{1,\infty}_{loc}(\R)\,:\,   
\hbox{$\sigma$ is nonincreasing, $\sigma(x)=1$ for $x\leq x_n$,}& \\
\hbox{$0\leq \sigma(x)\leq \omega_n(x)$ for $x\geq y_n$},\ \|\sigma\|_{W^{1,\infty}(-\infty,x)}\leq \t C(x)&\}\,,
\end{align*}
hence $\t X$ is invariant under $T$.
The set $\t X$ is convex and one readily checks that it is compact in $L^\infty(\R)$
using Ascoli-Arzela theorem and the conditions at~$\pm\infty$.

\smallskip
{\em Continuity of $T$.}\\
Let $(\sigma_j)_{j\in\N}$ in $\t X$ converge uniformly to some~$\sigma$. 
Then we have  that $A_j:= \alpha\circ\sigma_j$ uniformly converges to $A:= \alpha\circ\sigma$; notice that the conditions in $\t X$ imply that $A(x)=\alpha(1)$ for $x\leq x_n$  and  that
$0\leq A(x)\leq\alpha(\omega_n(x)) \to 0$ as $x\to+\infty$, in particular the $A_j$ do not trivialize in the limit. 
Let $(c_j, w_j, z_j)$ be the triplet provided by the construction of $T(\sigma_j)$.
Since the $(c_j)_{j\in\N}$ are in $(0,2\sqrt{\alpha(1)})$, they  converge, up to extraction of a 
subsequence, to some $c\in [0,2\sqrt{\alpha(1)}]$. 
On one hand, the~$w_j$ are locally uniformly equibounded from below away from zero due to~\rife{lowc}.
On the other hand, the values $w_j(x_n)$ are bounded from above away from $1$, because otherwise 
Proposition~\ref{pro:leftordering}
and the second Harnack's inequality in~\eqref{Harna} would yield a contradiction with 
the normalization~\rife{normalization}. 
We deduce that $(w_j(x_n))_{j\in\N}$ is contained in a compact subset of~$(0,1)$.
Then, since by the criticality condition
$$
 c_j^2 /4 =  \int_{\R} A_j(-w_j') dy 
 \geq \int_{0}^{x_n} A_j(-w_j') dy=\alpha(1)(1-w_j(x_n))\,,
$$
we find that $c=\lim_{j\to+\infty}c_j>0$.
We can therefore apply Lemma~\ref{lem:stability} and infer 
that~$(w_j)_{j\in\N}$ converges uniformly, up to subsequences, towards  a 
decreasing solution~$w$ to~\eqref{nlKPP}, which in
addition is critical, i.e., fulfills~\eqref{w-critical}. 
We find as a byproduct that $c<2\sqrt{\alpha(1)}$.
Indeed, by uniqueness of the critical wave,
c.f.~Theorem~\ref{thm:I}, we have that $w_j$ 
converges to $w$ along the whole subsequence on which $c_j\to c$.
%

It remains to check that the normalization condition~\eqref{normalization} is preserved up to subsequences.
By dominated convergence, it is sufficient to show that the functions
$$\ul w_{j,c_j}(x):=\inf_{0<c'\leq c_j}w_{j,c'}(x)$$
converge pointwise to 
$$\ul w_{c}(x):=\inf_{0<c'\leq c}w_{c'}(x),$$
where $w_{j,c'}$ and $w_{c'}$ are the critical waves with speed $c'$ corresponding to the nonlinearity $A_j$ and 
$A$ respectively.

To this purpose, we observe that, for fixed $c'\in (0,2\sqrt{\alpha(1)})$, the  
$w_{j,c'}$ converge uniformly, up to subsequences,
 to $w_{c'}$, thanks to Lemma \ref{lem:stability}
and the bounds~\eqref{<phi<} applied in $x_n$. 
This convergence holds true for the whole sequence $w_{j,c'}$ because of the  uniqueness of the critical wave for fixed $c'$. 
Since $\ul w_{j,c_j}(x)\leq w_{j,c'}(x)$ for any $c'<c_j$, and $c_j\to c$, we deduce that
$$
\limsup_{j\to \infty} \ul w_{j,c_j}(x) \leq w_{c'}(x) \qquad \forall c'<c\,.
$$
Recalling that the values $w_{c'}(x)$  are continuous with respect  to $c'$ owing to the
last part of Theorem~\ref{thm:nlKPP}, this yields
\be\label{lsup}
\limsup_{j\to \infty} \ul w_{j,c_j}(x) \leq \inf_{0<c'\leq c}w_{c'}(x)= \ul w_{c}(x)\,.
\ee
Conversely, we fix $x\in\R$ and, for any $\vep >0$ and $j\in\N$, we find $c_{j}^\e\in(0,c_j)$  such that 
$$
\ul w_{j,c_j}(x) \geq w_{j,c_{j}^\e}(x) -\vep\,.
$$
Without loss of generality, we can suppose that $c_{j}^\e\to \tilde c$ as $j\to\infty$,
for some $\tilde c\in [0,c]$. If $\tilde c>0$, Lemma \ref{lem:stability} and~\eqref{<phi<}
entail that  $w_{j,c_{j}^\e}(x)\to w_{\tilde c}(x)$; while if $\tilde c=0$,~\rife{<phi<} yields $w_{j_\vep,c_{j_\vep}}(x)\to1$. Hence, in any case,
$$
\liminf_{j\to \infty} \ul w_{j,c_j}(x) \geq w_{ \tilde c }(x) -\vep \geq \ul w_{c}(x) - \vep\,.
$$
Since $\vep$ is arbitrary, the previous inequality together with  \rife{lsup} imply that $\ul w_{j,c_j}(x)$ converges to $\ul w_{c}(x)$, for every $x\in \R$.  As we said above, this yields
$$
\int_{-\infty}^0e^y\, \ul w_{c}(y)dy=\lim_{j\to\infty}
\int_{-\infty}^0e^y\, \ul w_{j,c_j}(y)dy=\frac12\,,
$$
i.e., \eqref{normalization} holds.
As we have shown before, this condition uniquely characterizes $c$. 
We deduce that the whole sequence $c_j$ converges to $ c$ (and consequently, the whole sequence $w_{j,c_j}\to w_c$).

The convergences of $\sigma_j$, $c_j$ and $w_j$ now imply, by standard stability in the second equation, that $z_j$ converges uniformly to the unique $z$ which solves $-z''+ (c-2) z' + (\rho-1) z+  A \, w_c\, z= 1-\sigma$ in $(-n,n)$ with the given boundary conditions. Finally, we have proved that $(c_j,w_j,z_j)\to(c,w,z)$ uniformly, 
and the latter is the unique triple associated with 
$T(\sigma)$. We conclude using the characterization~\eqref{s*n=} 
that $T(\sigma_j) \to T(\sigma)$ in $L^\infty(\R)$.

\smallskip

We can now invoke Schauder's theorem which provides us with a fixed point $\sigma_n\in \t X$ such that $T_n(\sigma_n)= \sigma_n$. Associated with this function, we have $A_n:= \alpha\circ\sigma_n$ and a unique triple $(c_n,w_n,z_n)$ which therefore solves system~\eqref{system-nn}.
By construction, we have that $w_n$ satisfies the conditions $(i)$--$(iii)$. 
\end{proof}

\subsubsection*{Part II. Passing to the limit in the approximation}
\vskip-3pt
Now we study the limit of the sequence $(c_n,w_n,z_n)_{n\in\N}$ of solutions to \rife{system-nn} 
provided by Lemma~\ref{lem:truncated}. We call $\sigma_n$ the associated optimal functions $[\sigma]_n$,
and $A_n:=\alpha\circ\sigma_n$. We recall that $w_n$ is decreasing and that $z_n$ is increasing. 

To start with, we show that $\sigma_n$ stays bounded away from $0$. Indeed, by Lemma~\ref{lem:z>},
for any $x\in\R$ we have
$z_n(x+1)\geq C(1-\sigma_n(x))$, for some positive constant $C$ depending on $\kappa,\rho,\alpha(1)$.
Then for any given $x\in\R$, we find for $n>x+2$ that
either $\sigma_n(x)=1$, or by \eqref{s*n=}
\begin{align*}
\frac1{\alpha'(\sigma_n(x))}   
& = e^{-x}\int_{x}^{+\infty}z_n(y) e^y\, w_n(y)\chi_n(y)dy  \\ 
& \geq e^{-x}\int_{x+1}^{x+2} z_n(y) e^{y}\, w_n(y)dy \\
& \geq C(1-\sigma_n(x)) w_n(x+2) (e^2-e).
\end{align*}
Owing to~\eqref{lowc}, this provides a positive lower bound 
for~$\sigma_n(x)$ independent of $n$. We have thereby shown that
\Fi{sn0>}
\liminf_{n\to\infty}\sigma_n(x)>0\quad\forall x\in\R.
\Ff

Next, we derive an upper bound for $w_n$.  Namely, we claim that
up to extraction of a subsequence, there holds that
\Fi{wn0<}
w_n(0)\leq\ol\vt <1\quad\forall n\in\N.
\Ff
To show this, we consider two (mutually excluding) possibilities.  Either there exists some $\beta<1$ such that $\sigma_n(0)\leq\beta<1$ for all~$n\in\N$; in this case
the same computation as before yields, for any $\xi>1$ and $n>\xi$,
\begin{align*}
 \frac1{\alpha'(\beta)} \geq  \frac1{\alpha'(\sigma_n(0))}   
   &\geq \int_{1}^{\xi} z_n(y) e^y\, w_n(y)dy \\
   & \geq C(1-\sigma_n(0)) w_n(\xi) (e^\xi-e)\\
   & \geq C(1-\beta) w_n(\xi) (e^\xi-e),
\end{align*}
that is,
$$
w_n(\xi) \leq  \frac1{e^\xi-e}\, \left(\frac1{C(1-\beta) \alpha'(\beta)}\right)\,.
$$
The right-hand side is smaller than $1$ for $\xi$ sufficiently large, and thus~\rife{wn0<}
follows from Harnack's inequality~\rife{Harna}. 
Alternatively, there exists a subsequence (not relabeled) such that
	$\sigma_n(0) \to 1$ as $n\to\infty$.
If this is the case, 
using the characterization~\eqref{s*n=}  of~$\sigma_n$ we derive
$$\liminf_{n\to\infty}\int_{\ln\frac23}^{+\infty}z_n(y) e^y\, w_n(y)\chi_n(y)dy\geq
\liminf_{n\to\infty}\int_{0}^{+\infty}z_n(y) e^y\, w_n(y)\chi_n(y)dy\geq\frac{1}{\alpha'(1)},$$
which, owing to the same characterization, shows that
$A_n(\ln\frac23)=\alpha(1)$ for $n$ sufficiently large. Hence,
for such values of $n$, Proposition~\ref{pro:leftordering} yields
$w_n=\ul w_{c_n}$ in $(-\infty,\ln\frac23]$ and therefore, by~\eqref{normalization},
$$\frac12 \geq \int_{-\infty}^{\ln\frac23}e^y\, w_n(y)dy\geq \frac23
\textstyle w_n(\ln\frac23).$$
Namely, $w_n(\ln\frac23)\leq\frac34$, whence we deduce \rife{wn0<} because $w_n$ is decreasing.
\vskip0.5em
Henceforth, we reason up to subsequences and we suppose that  $c_n$ converges   to some $c\in[0,2\sqrt{\alpha(1)}]$ 
and that the functions $w_n,z_n,\sigma_n,A_n$ converge, respectively,  
towards  some $w,z,\sigma,A:=\alpha(\sigma)$ locally uniformly in $\R$ 
(observe that the $\sigma_n$ are equicontinuous on compact sets due to~\eqref{deriv2} and \eqref{sn0>}).

We claim that these functions solve~\eqref{zw-waves}. To prove this we need to check that
the various terms do not trivialyze. This is done in the following items.
\begin{enumerate}[\ \ a)]
	\item $0<w<1$.
	\\
	This follows from the bounds~\eqref{wn0>} and  \eqref{wn0<} and Harnack's inequalities~\eqref{Harna}.

	\item $c>0$.
	\\	
	From the criticality condition \rife{crin} we obtain
	\be\label{cn24}
	c_n^2 /4 = \int_{\R} A_n(-w_n') dy \geq  \int_{-\infty}^{x} A_n(-w_n') dy \geq A_n(x) (1-w_n(x)),
	\ee
	which implies $c>0$ due to~\eqref{sn0>} and~$w<1$.
	
	\item $A>0$, $A\not\equiv\alpha(1)$.
	\\	
	We already know from~\eqref{sn0>} that $A>0$. 
	Then the lower bounds on $c_n$ and $1-w_n$ imply that 
	$w_n$ satisfies  the estimate~\eqref{decay_crit} with a constant $K$ independent of $n$, namely
	\be\label{c2dec}
	w_n(x) \leq K x e^{- \frac {c_n}2 x} \qquad \forall x\geq 1, \quad \forall n \in \N.
	\ee
	In particular, we see that
	$w(x) \to 0$ as $x\to \infty$. Now, suppose by contradiction that $A(x) \equiv \alpha(1)$.
	Applying \rife{cn24} with $x_0$ arbitrarily large would show that  
	$c=\lim_{n\to\infty}c_n=2 \sqrt {\alpha(1)}$.  Since $\alpha(1)>1$, together with \rife{c2dec} this would imply that 
 	$w_n(x)e^x$ are equi-integrable at $+\infty$ for $n$ large enough. But then, from the characterization of $\sigma_n$ in~\eqref{s*n=}, we would have $A(x) < \alpha(1)$ for large $x$, which gives a contradiction. 
	
%
%
	
	\item $0<z<\frac1{\rho-1}$.	

	As $\sigma_n$ and $A_n, w_n$ converge locally uniformly to $\sigma,A, w$ respectively, we have 
	by standard stability that the function $z$ satisfies the equation in~\eqref{zw-waves}, i.e.~\eqref{eq:z}. Moreover, $z$ is nondecreasing and satisfies $0\leq z\leq\frac1{\rho-1}$. Indeed, since the right-hand side in the equation is nonnegative,
	the strong maximum principle yields $z\equiv0$ as soon as $z$ vanishes somewhere. 
	But this is impossible because $z\equiv 0$   entails
	$\sigma\equiv1$, i.e., $A\equiv\alpha(1)$, which was already excluded. This means 
	in particular that~$z>0$.
	
	A similar argument applies from above. The constant function $\frac1{\rho-1}$ is a super-solution of~\eqref{eq:z}.
	Hence if $z$, which is less than or equal to $\frac1{\rho-1}$, attains the value $\frac1{\rho-1}$
	somewhere, the strong maximum principle yields $z\equiv\frac1{\rho-1}$.
	Coming back to the equation~\eqref{eq:z}, we see  that this is only possible if $\sigma\equiv0$, that is, $A\equiv0$. 
	But this has already been ruled out.
	

	\item $ 2 < c < 2\sqrt {\alpha(1)}$.\\
	Recall that $0<w<1$, $c>0$ and $A\not\equiv0$. We can then apply Lemma~\ref{lem:stability} and infer that $w$ solves~\eqref{nlKPP} and fulfills the
	critical identity~\eqref{w-critical}. In particular, since $A\not\equiv\alpha(1)$, as seen in c),
	we deduce that $c < 2\sqrt {\alpha(1)}$.

	Finally, we are left to show that $c>2$.  To this purpose, we observe that Proposition~\ref{implicit-speed}, together with~\eqref{wn0>}, implies that 
	$$w_n(x)\geq\ul\vt e^{-\frac{c_n}2 x}\qquad\forall x>0.$$	
	Assume by contradiction that $c\leq2$. We then have that 
	$$w(x)\geq\ul\vt e^{-x}\qquad\forall x>0.$$
	For any arbitrary $x_0<x_1$, we find that 
	$$\int_{x_0}^{x_1} z(y) e^y\, w(y)dy\geq z(x_0) \ul\vt(x_1-x_0).
	$$
	In particular, because $z(x_0)>0$ by b), there exists $x_1$ (depending on $x_0$) such~that
	$$\int_{x_0}^{x_1} z(y) e^y\, w(y)dy>\frac{e^{x_0}}{\alpha'(1)},$$
	and this inequality holds true for $z_n$, $w_n$ when
	$n$ is sufficiently large. It follows from~\eqref{s*n=}
	that, for such values of $n$ (that we can assume without loss of generality 
	being larger than $x_1$), $\sigma_n(x_0)=1$. This means that $A\equiv\alpha(1)$,
	but this case has been excluded in c).		
\end{enumerate}

	Summing up, we have shown that $c,w,z$ solve the equations and constraints  
	in~\eqref{zw-waves} with $A:=\alpha\circ\sigma$.
	It remains to prove that $\sigma$ is indeed the optimal function associated with $w,z$.
	This follows form the fact that, as $n\to\infty$, the integral equivalence 
	in~\eqref{s*n=} reduces by dominated convergence (recall that the  
	decays~\rife{c2dec} hold with $c_n\to c>2$)
	to the characterization~\eqref{s*=} of $\sigma$.
	This concludes the proof of Theorem~\ref{main-zw}$(ii)$.

\subsection{Waves with supercritical speed} 
 
 We now deal with Theorem~\ref{main-zw}$(iii)$, namely, 
 we construct other traveling waves for the system \rife{zw-waves}, with speeds which are faster than 
 $2\kappa\sqrt{\alpha(1)}$.
 As in the previous subsection, we assume for simplicity that $\kappa=1$. 
 For each speed, we are able to attain any arbitrary normalization $\ell_0\in (0,1)$ for
 $w$ at a given point $x_0\in \R$. 



We start with a lemma on the shooting method for the nonlocal problem, similar
to Lemma~\ref{lem:ivp}.

\begin{lemma}\label{CP-coro}
Let $A\in W^{1,\infty}_{loc}(\R)$ be a bounded, nonnegative, nonincreasing function satisfying 
$\bar A:=A(-\infty)>0$.
Consider the problem
\be\label{wgam}
\begin{cases}
	\displaystyle
	w''+cw'+w \left( A(a)- Aw + \int_{a}^x A'(y) w(y)dy\right)=0,\quad x>a\\
	w(a)=\gamma\\
	w'(a)=0\,,
\end{cases}
\ee
with  $c\geq 2\sqrt {\bar A}$, $a\in \R$ and $\gamma\in (0,1)$. 
Then we have:

\begin{enumerate}[$(i)$]
 \item problem \rife{wgam} admits a unique solution $w(\cdot;\gamma)$, which 
 in addition is decreasing and positive in $(0,+\infty)$, with $w \to 0$ as $x\to +\infty$;

 \item for any $x_0>a$ and $\ell_0\in (0,1)$, there exists a unique $\gamma_0$ such that $w(x_0;\gamma_0)= \ell_0$.
 \end{enumerate}
\end{lemma} 

\begin{proof}
For given $\gamma$, local existence and uniqueness of $w$ is provided by Lemma \ref{lem:ivp}, and  $w$ is 
nonincreasing from Lemma \ref{lem:CP}, because the constant $\gamma$ is a sub-solution.  In fact, if we observe that   $w''(x)<0$ if ever  $x\geq a$ and $w'(x)=0$, we conclude that $w$ is actually decreasing.
Consider now the unique
wave $\psi_\gamma$ of the classical KPP equation
\be\label{frontebasso}
\begin{cases}
\displaystyle
\psi''+c\psi'+\bar A\psi(1-\psi)=0,\quad x\in\R \\
\psi(-\infty)=1,\quad \psi(+\infty)=0\,, \quad \psi(a)= \gamma, 
\end{cases}
\ee
which exists because $c\geq 2\sqrt {\bar A}$.  Since we have
\begin{align*}
A(a)- A\psi + \int_{a}^x A'(y) \psi(y)dy & = A(a)(1-\psi(a))- \int_{a}^x A(y) \psi'(y)dy
\\ & \leq 
A(a)(1-\psi) \leq \bar A (1-\psi)
\end{align*}
then  $\psi$ is a super-solution of problem \rife{wgam}, with $\psi'(a)<0$. By the comparison principle of Lemma \ref{lem:CP} we deduce that $w(\cdot; \gamma)\geq \psi_\gamma$. Hence $w$ exists for all times and admits a limit as $x\to \infty$.  We observe that  the equation reads as
$$
(w' e^{c x})'= - we^{cx} \left( A(a)(1-\gamma)  -\int_{a}^x A(y) w'(y)dy\right) = - w e^{cx}\, g(x)
$$ 
where $g(x) $ is an increasing function which admits a  bounded  limit as $x\to \infty$;  then necessarily we deduce that $w' \to 0$ and $w\to 0$ as $x\to \infty$. Indeed, if $w(x) $ has a positive limit at infinity, then $(w' e^{c x})' \sim - \theta e^{cx} $ for some $\theta>0$, in which case $w' $ converges to a negative constant at infinity. But this is impossible, so $w(x) \to 0$ and in turn $w'(x)\to 0$ as well, for $x\to \infty$.
This proves $(i)$. 

Now, for $x_0>a$,  we consider the map  $\gamma \mapsto w(x_0;\gamma)$; this is continuous and  nondecreasing due to Lemma \ref{lem:ivp} and Lemma \ref{lem:CP} respectively. But the monotonicity is actually strict; as we observed before, for $\gamma_1<\gamma_2$, we have $w_{\gamma_1}''(a)<w_{\gamma_2}''(a)$, so $w_{\gamma_1}'<w_{\gamma_2}'$ for $x>a$ and we get $ w(x_0;\gamma_1)<w(x_0, \gamma_2)$ by Lemma \ref{lem:CP} again.
Therefore,  the range of $\gamma \mapsto w(x_0; \gamma)$ is an interval. Clearly we have $w(x_0; \ell_0)<\ell_0$ because $w$ is decreasing. On another hand, there exists a  wave $\psi_0$  for the KPP equation \rife{frontebasso} such that $\psi_0(x_0)=\ell_0$; if we take $\gamma=\psi_0(a)$, by comparison we know that $w(x_0;\gamma)>\psi_0(x_0)=\ell_0$. Therefore,
we deduce the existence of a unique $\gamma_0\in (\ell_0, \psi_0(a))$ such that $w(x_0;\gamma_0)=\ell_0$.
\end{proof}

Similarly to the previous section, we use  
a fixed point argument to build an approximation of the traveling wave in the compact set $[-n,n]$.
However, the approximated problem slightly differs from~\eqref{system-nn}.

\begin{lemma}\label{nk}
Assume that hypotheses \rife{ro}--\rife{alpha2} hold true and let
$c\geq 2\sqrt{\alpha(1)}$. For $ n\in \N$ larger than $\alpha'(1)$,
$|x_0|<n$ and $\ell_0\in (0,1)$, there exists a solution $(w_n,z_n)$ of the 
problem
\be\label{systemnk}
\begin{cases}
\displaystyle
w''+c w'+w \left( A(-n)- A w + \int_{-n}^x A'(y)w(y)dy\right)=0,\quad x\in [-n,n]
\\
0<w<1,\quad w(x_0)=\ell_0,\quad w'<0\,,
\\
\m
-z''+ (c-2) z' + (\rho-1) z+ A \, w\, z= 1-[\sigma]_n,\quad x\in [-n,n]
\\
z(-n)=0, \quad z(n)=\frac1{\rho-1},\quad z'>0\,, 
\\
\m
\displaystyle [\sigma]_n(x)= \mathop{{\rm argmax}}_{s\in [0,1]} \,  \left\{ (1-s)e^x+\alpha(s)\left(\int_x^{n} z(y) e^y\, w(y)dy+\frac1n\right)\right\} \,,\quad A:= \alpha\circ [\sigma]_n\,.
\end{cases}
\ee
\end{lemma}

\begin{proof}  We consider the following subset of $C^0([-n,n])$:
	$$X:=\{\sigma\in W^{1,\infty}([-n,n]),\ :\ 
	\text{$\sigma$ is nonincreasing, $(\alpha')^{-1}(ne^n)\leq\sigma\leq1$}\}.$$
(observe that $(\alpha')^{-1}(ne^n)\leq(\alpha')^{-1}(n)<1$ by hypothesis).
We define a map $T$ on $X$ in the following way. Take $\sigma\in X$ and call $A:=\alpha\circ\sigma$.
First, we let $w$ be the unique solution of the problem
$$
\begin{cases}
\displaystyle
w''+cw'+w \left( A(-n)- Aw + \int_{a}^x A'(y) w(y)dy\right)=0,\quad x\in [-n,n]\\
\displaystyle w(x_0)=\ell_0\\
w'(-n)=0\,.
\end{cases}
$$
Existence and uniqueness of $w$ are given by Lemma \ref{CP-coro}, which additionally ensures that
$w$ is decreasing and satisfies $0<w<1$. Next, given $\sigma,A$ and $w$, we 
consider the unique solution of the linear elliptic problem in~\rife{systemnk},
with~$[\sigma]_n$ replaced by $\sigma$.
We have seen in the proof of Lemma~\ref{lem:truncated} that $z$ is increasing.
We finally define $T(\sigma):=[\sigma]_n$ from the last line of~\rife{systemnk}.
This function is nonincreasing and fulfills 
an analogous characterization to the one derived in the previous~section:
\be\label{tsigma}
[\sigma]_n(x)= (\alpha')^{-1}\left(\frac{e^x}{\int_x^{n} ze^y w \,dy+\frac1n}\right)
\qquad \forall x\in[-n,n]\,:\, [\sigma]_n(x)<1\,.
\ee
This yields~$[\sigma]_n(x)\!\geq\!(\alpha')^{-1}(ne^n)$ and
moreover, by analogous computation as in~\eqref{deriv},
\be\label{deriv2bis} 
[\sigma]_n'(x)   =   \frac{\alpha'([\sigma]_n(x))}{\alpha''([\sigma]_n(x))}\left( 1+   z(x)w(x)\alpha'([\sigma]_n(x)) \right) \quad \forall x\,:\, [\sigma]_n(x)<1\,.
\ee 
By the boundedness of $w,z$ and the regularity of $\alpha$, 
as well as the positive lower bound for $[\sigma]_n$,
we eventually deduce $|[\sigma]_n'|\leq C$ for some positive constant $C$ only depending on
$c,\rho,\alpha,n$.
We have thereby shown that $T(X)\subset X$.

Actually, we have shown that $T(X)\subset \t X$, with
$$\t X:=\{\sigma\in W^{1,\infty}([-n,n]),\ :\ 
\text{$\sigma$ is nonincreasing, $(\alpha')^{-1}(ne^n)\leq\sigma\leq1$, 
$|\sigma'|\leq C$}\}.$$
This is a compact, convex subset of $C^0([-n,n])$.

We now prove that $T$
admits a  fixed point in $\t X$. 
%
%
Let us check the continuity of $T$. 
Consider a sequence $(\sigma_j)_{j\in\N}$ in $\t X$ converging uniformly to some $\sigma$.
Call $A_j:=\alpha\circ\sigma_j$ and $w_j$, $z_j$ the associated functions used in the definition of $T$.
Integrating by parts the term in the equation for $w_j$ 
(in order to get rid of the term $A_j'$), using elliptic estimates, and then integrating back,
 we find a subsequence of $(w_j)_{j\in\N}$ converging uniformly 
 to a solution $w$ of the same equation, with $A:=\alpha\circ\sigma$, which satisfies in addition 
 $0\leq w\leq 1$ and $w'(-n)=0$, $w(x_0)=\ell$. By
 Lemma~\ref{CP-coro} there is a unique of such solutions, hence the whole sequence
 $(w_j)_{j\in\N}$ converges towards it. Likewise, $(z_j)_{j\in\N}$ converges uniformly to the unique solution of
 the corresponding equation with~$A$ and $\sigma$.
 Then, using the characterization~\eqref{tsigma}, we deduce that $(T(\sigma_j))_{j\in\N}$ converges uniformly
 to the function $[\sigma]_n$ defined as in~\rife{systemnk}. This is precisely $T(\sigma)$.

We can therefore invoke
Schauder's theorem and conclude that the map $T$ has a fixed point in $\t X$, 
which is by construction a solution of \rife{systemnk}. 
\end{proof}

We finally analyze the limit as  $n\to \infty$, in order to get the wave for the system~on the whole line. 
It is  here that we face the question whether $w\, e^x $ is integrable at~$+\infty$. 
\vskip1em

\begin{proof}[Proof of Theorem~\ref{main-zw}$(iii)$]
	  Fix $c\in [2\sqrt{\alpha(1)}, \alpha(1)+1)$. Let $w_{n},z_{n}$ be a solution of system \rife{systemnk},
	  provided by Lemma~\ref{nk},
	  and let $[\sigma]_n$ and $A_n:=\alpha\circ[\sigma]_n$ 
	  be the associated function from the last line of~\rife{systemnk}.
   	First of all, by elliptic estimates, both $w_n$ and $z_n$ are  locally bounded in~$C^2$ norm, hence 
   	they converge (up to subsequences) in $C^1_{loc}(\R)$ to some functions $w$, $z$. 
   	We claim~that
\be\label{pk}
\sup_{n\in\N}\,\int_{-n}^n z_n\,w_n\, e^y dy<+\infty.
\ee
To show this, assume first
by contradiction that there exists $x\in\R$ such that (up to subsequences) 
$$p_n(x):= \int_{x}^n z_n\,w_n\, e^y dy\to +\infty.$$ 
Then, since   $p_n'(x)$ is locally uniformly bounded, this must be true for all $x\in\R$.
Recalling the characterization~\ref{tsigma} for $[\sigma]_n$, 
we deduce that $[\sigma]_n(x)= 1$ for $n$ sufficiently large, depending on $x$, 
hence $A_n(x)=\alpha(1)$ for large $n$, and this actually holds uniformly in $(-\infty,x]$,
by monotonicity. 
It follows that
the limit $w$ of the $w_n$ is a solution of
$$ 
w''+c w'+\alpha(1) w(1-w)= 0,\quad  
$$
such that $w(x_0)=\ell_0\in(0,1)$, that is,
$w$ is a wave for the standard KPP equation. We deduce that $w$ tends to $0$
as $x\to +\infty$. In particular, for $\e\in(0,1)$ to be chosen later, we can find $x_1>0$ such that
$w(x_1)<\e/2$. Now, we come back to $w_n$; by pointwise convergence, we take  $n$   large enough so that $w_n(x_1)<\e$. For $x>x_1$, we estimate
\[\begin{split}
0 & = w''_n+c w'_n+w_n(x) \big( A_n(-n)- A_n w + \int_{-n}^x A_n'(y)w(y)dy\big)
\\ &\geq w''_n+c w'_n+w_n(x)  A_n (x)(1-w_n(x))  
\\
&\geq w''_n+c w'_n+ A_n(x)w_n(x)(1-\e).
\end{split}\]
Since we are assuming $A_n(x) \to \alpha(1)$, this implies,   for $n$ sufficiently large, 
$$
w''_n+c w'_n+(\alpha(1)-\e)(1-\e)w_n\leq 0.
$$
Then the  function $q:=-w_n'/w_n$ satisfies
\begin{align*}
q' \geq q^2-cq+(\alpha(1)-\e)(1-\e)=(q-\lambda_\e^-)(q-\lambda_\e^+),
\end{align*}
where
$$\lambda_\e^\pm:=\frac c2\pm\sqrt{\frac{c^2}4-(\alpha(1)-\e)(1-\e)}.$$
We infer that 
$$\liminf_{x\to+\infty}q(x)\geq \lambda_\e^-.$$
But the condition  $2<c<\alpha(1)+1$ implies 
$\lim_{\e\to0^+}\lambda_\e^->1$, hence we can choose $\e$ small enough so that
$\lambda_\e^->1$. 
Reverting to the function $w_n$, we derive
$$
w_n(x) \leq C\, e^{-\lambda x} \qquad \forall x>x_1\,,
$$
for some $C>0$ and $\lambda>1$.
This estimate and the bound on $z$ imply    
$$
p_n(x_1)= \int_{x_1}^{+\infty} z_nw_ne^ydy \leq \frac C{\rho-1} \int_{x_1}^{+\infty} e^{(1-\lambda)y}dy, 
$$
so $p_n(x_1)$ cannot blow-up. This contradicts the fact that $p_n\to+\infty$ pointwise.
We have thus shown that $p_n(x)$ remains bounded at any given $x$. 
We improve this to the bound~\rife{pk} by noticing that
$$
\int_{-n}^n z_n\,w_n\, e^y dy = \int_{-n}^0 z_nw_ne^ydy + p_n(0) \leq \frac1{\rho-1} + p_n(0).
$$

We now refine the above argument to show a uniform decay for $w_n$ at infinity. Observe that
\rife{pk}, together with \eqref{tsigma}, implies that the function $[\sigma]_n$ in~\eqref{systemnk}
associated with $(z_n,w_n)$ does not tend to $1$ as $n\to\infty$ at a sufficiently large point $x$.
Thus, by Lemma~\ref{lem:z>}, the limit $z$ of (a subsequence of)
$z_n$ is positive for~$x$ sufficiently large. 
We infer that $zwe^y \in L^1(\R)$ as a consequence of  \rife{pk} and 
Fatou's lemma.
Now consider the functions $q(x):= -\frac{w'( x)}{w( x)}$ and $q_n(x):= -\frac{w'_n }{w_n }$. Since
$we^y\in L^1$
  there must be  a point~$\bar x$ where  $q(\bar x)>1$ (otherwise, if $-\frac{w'} w\leq 1$ for every $x$, then $we^y $ is not integrable at infinity).  By pointwise convergence, we can assume that $q_n(\bar x) \geq (1+\vep)$ for some $\vep>0$, and for all $n\in \N$. But since
$$
q'_n = q^2_n-cq_n + \big( A_n(-n)- A_n w_n + \int_{-n}^x A_n'(y)w_n(y)dy\big)
$$ 
the same argument used to prove Proposition \ref{implicit-speed} shows that $q_n$ is increasing in $(0,n)$; hence we deduce that $q_n(x)\geq (1+\vep)$ for every $x\in (\bar x,n)$.  Recalling that $q_n= -\frac{w_n'}{w_n}$, integrating we get
$$
w_n(x) \leq w_n(\bar x) e^{-(1+\vep) (x-\bar x)} \leq C \, e^{-(1+\vep) x},\quad x\in (\bar x, n).
$$
Thanks to this estimate, we can use the dominated convergence theorem and we conclude that 
$$
\int_x^n z_n\,w_n\, e^y dy \ \mathop{\to}^{n\to \infty} \ \int_x^{+\infty} z \,w \, e^y dy\,.
$$
This implies that $[\sigma]_n$ characterized by~\eqref{tsigma}
pointwise converges to $\sigma$ characterized by~\eqref{s*=}, that is,
$$ \sigma(x):= \mathop{{\rm argmax}}_{s\in [0,1]} \,\left\{ (1-s)e^x+\alpha(s)\int_x^{+\infty} z(y) e^y\, w(y)dy\right\}.
$$
Call $A:=\alpha\circ \sigma$. 
We immediately deduce that $z$ solves the equation in~\rife{zw-waves}.
Next, the uniform positive lower bound on $z_n(x)$ for $x$ large, which is also true for
$w_n(x)$ due to $w_n(x_0)=\ell_0$ and Harnack inequality, implies that $A_n= A= \alpha(1)$ on some half line
$(-\infty,\bar x]$.
For $-n<\bar x$ we then find that 
\begin{align*}
-w_n''-c w_n' &= \alpha(1)- A_n w_n + \int_{\bar x}^x A_n'(y)w_n(y)dy\\
&= \alpha(1)(1-w_n(\bar x))- \int_{\bar x}^x A_n(y)w_n'(y)dy\\
&\mathop{\to}^{n\to \infty} \ 
\alpha(1)(1-w(\bar x))- \int_{\bar x}^x A(y)w'(y)dy\\
& = \alpha(1)(1-w(-\infty))- \int_{-\infty}^x A(y)w'(y)dy.
\end{align*}
Hence $w$ satisfies the differential inequality~\eqref{phiinfty} used in the proof of
Lemma~\ref{lem:phi'<0} to derive
$w(-\infty)=1$ and $w(+\infty)=0$. This shows that $w$ solves~\rife{nlKPP}.
In the end, $(z,w)$ are proved to be solutions of the problem~\rife{zw-waves}. 
\end{proof}

This concludes the proof of Theorem \ref{main-zw}.


\subsection{From traveling waves to BGP solutions}\label{sec:tw-BGP}

We now deduce Theorem \ref{main} from the results obtained in Theorem \ref{main-zw} on the traveling waves. We only need to establish first  a rigorous  connection between  solutions of~\rife{zw-waves} and BGP solutions of the mean field game system \rife{system}, which are defined as in Definition \ref{BGP}. 

\begin{proposition}\label{bgp-tw}
Assume conditions\eqref{alpha1},\eqref{alpha0},\eqref{alpha2},\eqref{alpha3} hold.
If $(f,v, s^*)$ is a BGP with growth $c$, then $w(x):= \int_x^{+\infty} \vfi(y)dy$ and $z(x):=  \nu'(x) e^{-x}$ are solutions of the traveling wave system~\rife{zw-waves}.

Conversely, let $c<\rho$ and $(w,z)$ be solutions of \rife{zw-waves}. Then the triple given by
\be\label{tobgp}
f= -w'(x-ct)\,;\qquad v= e^{ct} \left( \int_{-\infty}^{x-ct} z(y)e^ydy + K\right) \,; \qquad s^*= \sigma(x-ct)
\ee
is a BGP solution of system \rife{system}, with $K= \frac{\alpha(1)}{\rho-c}\int_{\R} e^y z(y) w(y)dy$.
\vskip0.3em
Finally, there are no possible BGPs with growth  $c\geq \rho$.
 \end{proposition}
 
 \proof  Suppose we are given a BGP solution $(f,v, s^*)$  of \rife{system}, hence 
 $f(t,x)=\varphi(x-ct)$, $v=e^{ct} \nu(x-ct)$, $s^*=\sigma(x-ct)$ for some~$c>0$.
 Thus, as shown in Section~\ref{sec:derivation}, the functions
  $w(x):= \int_x^{+\infty} \vfi(y)dy$, $z(x):=  \nu'(x) e^{-x}$
  are solutions to the traveling wave system~\eqref{zw-waves}.
Finally, according to Definition \ref{BGP}, we also have $we^x \in L^1(\R)$, 
while~$z$ is nonnegative and bounded. 
Hence $zwe^x \in L^1(\R)$ and therefore all conditions in~\rife{zw-waves} are fulfilled.

Conversely, assume that $(w,z,\sigma)$ is a solution of \rife{zw-waves} and define $(f,v,s^*)$ from~\rife{tobgp}.  Set $\vfi (r) := -w'(r) $ and  $\nu(r):= \int_{-\infty}^r z(y) e^y dy + K$, where 
\be\label{K}
K:= \frac{\alpha(1)}{\rho-c}\int_{\R} e^y z(y) w(y)dy .
\ee
It follows that $f= \vfi(x-ct)$, $v= e^{ct} \nu(x-ct)$ with 
$\vfi, \nu\in C^2(\R)$, as required in Definition~\ref{BGP}. 
We further know that $\sigma\in W^{1,\infty}_{loc}(\R)$ by 
Proposition~\ref{pro:NC}. Still from  Proposition \ref{pro:NC}, we know that $z>0$ and it has a bounded positive limit  as $r\to+\infty$; this implies that $\nu$ is increasing, nonnegative and  $\nu' e^{-x} \in  L^\infty(\R)$. Therefore, the condition $zwe^y\in L^1(\R)$ implies that $e^r w(r)= e^r \int_r^{+\infty} \vfi(y)dy\in L^1(\R)$. By  Proposition \ref{implicit-speed}, we also know that
$\frac{-w'(r)}w(r) \to \lambda>0$ as $r\to +\infty$; hence we deduce that $w' e^y\in L^1(\R)$ 
as well.

So far, we have checked that the first three conditions in 
Definition \ref{BGP} hold. We are left to show that $(f,v)$ solve the system \rife{system}. 
We preliminarily observe that the definition of $\sigma$ in \rife{zw-waves} yields through
the computations~\rife{strav}
and~\eqref{hamvx}, that $s^*=\sigma(x-ct)$ satisfies
\[
\begin{split}
	s^* & = \mathop{{\rm argmax}}_{s\in [0,1]} \left\{(1-s)e^x + \alpha(s) \int_x^{+\infty} v_x(t,y)  \int_y^{+\infty} f(t,r)dr dy\right\}
	\\
	& = \mathop{{\rm argmax}}_{s\in [0,1]} \left\{(1-s)e^x + \alpha(s) \int_x^{+\infty}  [v(t,y)-v(t,x)]   f(t,y) dy\right\}\,,
\end{split}
\]
that is, $s^*$ is given by the formula in~\eqref{system}.
As far as $v$ is concerned,
taking the derivative in the equation of $w$, and using that $A(x-ct)=\alpha(s^*(t,x))$,
it readily follows that $f=\vfi(x-ct)$ is a traveling wave solution of the Fokker-Planck equation.  
Finally, in order to derive the equation for $v$, 
we consider the Hamiltonian    
\[\begin{split}
H(t,x;v) := &  \max_{s\in[0,1]} \big[(1-s)e^x + \alpha(s) \int_x^{+\infty}  [v(t,y)-v(t,x)]   f(t,y) dy\big]\\
=  &  \max_{s\in[0,1]} \big[(1-s)e^x+ \alpha(s) \int_x^{+\infty} v_x(t,y)  \int_y^{+\infty} f(t,r)dr dy\big],
\end{split}\]
where the equality follows from~\eqref{hamvx}.
We know that the above maxima are attained at $s=s^*(t,x)$, 
and actually this is the unique maximizer because the expressions are concave in $s$,
because $v_x= e^x z(x-ct)>0$.
Applying the envelope theorem as in Section~\ref{sec:derivation}, we notice that $H$ is differentiable in $x$ with $\partial_x H = (1-s^*) e^x - \alpha(s^*)v_x  \int_x^{+\infty}    f(t,y) dy$. This means that 
$$
(1-\sigma(x-ct)) - \alpha(\sigma(x-ct)) w(x-ct) z(x-ct)= e^{-x }\partial_x H (t,x;v)\,.
$$
Inserting this information into the equation for $z$ in~\rife{zw-waves}, we conclude that $v_x= e^x z(x-ct)$ satisfies the differential equation
$$
-\partial_t v_x- \kappa^2\partial_{xx} v_x + \rho v_x =  \partial_x H (t,x;v)\,.
$$
Now we integrate this equation  in the interval  $(-\infty,x)$. We observe  that
\begin{align*}
H(t,-\infty;v) & = \alpha(1)   \int_{\R}  v_x(t,y)  \int_y^{+\infty} f(t,r)dr dy 
\\ & = e^{ct}\alpha(1) \int_{\R} e^y z(y) w(y)dy = K(\rho-c)e^{ct},
\end{align*}
with $K$ defined by~\eqref{K}, while, by definition of $v$, we have
$$
v(t,-\infty) =e^{ct} K\,, \qquad \partial_t v (t,-\infty)= c e^{ct} K \,. 
$$
Finally, using that $v_{xx}= e^x (z(x-ct) + z'(x-ct)) \to 0$ as $x\to -\infty$, we conclude that 
\begin{align*}
0 & = \int_{-\infty}^x \left\{ -\partial_t v_x- \kappa^2\partial_{xx} v_x + \rho v_x -  \partial_x H (t,x;v)\right\} 
\\
& = -\partial_t v - \kappa^2\partial_{xx} v  + \rho v -    H (t,x;v) 
+ \partial_t v (t,-\infty)- \rho v(t,-\infty)+ H(t,-\infty;v) 
\\ & = -\partial_t v - \kappa^2\partial_{xx} v  + \rho v -    H (t,x;v)
\end{align*}
which means that $v$ is a solution of the Bellman equation. Therefore, we have proved that $(f,v, s^*)$ is a  BGP
solution  of system \rife{system}.

We conclude by observing that $c<\rho$ is necessary for a BGP to exist. Indeed, if a BGP exists, we have established so far that  it is of the form \rife{tobgp} for some $(z,w,\sigma)$ solution of \rife{zw-waves} and for some constant $K\in \R$. By properties of $z$ and  $\sigma$, given in Proposition \ref{pro:NC}, we deduce that $v_{xx}\to 0$ as $x\to -\infty$, owing to elliptic estimates,
while $s^*=\sigma(x-ct)\to 1$ as $x\to -\infty$. 
Then, writing the equation for $z$ in~\rife{zw-waves}
in terms of $v_x= e^x z(x-ct)$, integrating it on $(-\infty,x)$ and using that $v$ is a solution
of~\rife{system}, with the same computation as before we get
$$
(\rho-c) K = \alpha(1) \int_{\R} e^y z(y) w(y)dy\,.
$$
Since the right-hand side is positive, and $K\geq 0$ because otherwise $v$ would be negative for $-x$ large,
we deduce that $\rho>c$.
\qed

\vskip0.5em

We can finally conclude with the proof of our main result. 
\vskip0.3em
{\bf Proof of Theorem \ref{main}.} \quad 

$(i)$ \,
 If there exists  a BGP solution, then we have the necessary condition $c<\rho$ from Proposition \ref{bgp-tw}. In addition,  a BGP solution yields a traveling wave $(z,w)$ solution of \rife{zw-waves}. Hence $c$ must satisfy the conditions in Theorem \ref{main-zw}-$(i)$.

$(ii)$ \,
By Theorem \ref{main-zw}-$(ii)$, there exists  a traveling wave with speed $c\in (2\kappa^2,  2\kappa\sqrt{\alpha(1)})$ which is also critical.  If $\rho \geq 2\kappa\sqrt{\alpha(1)}$, then $c<\rho$ so by Proposition \ref{bgp-tw} this yields  a BGP solution of  \rife{system}, which is critical as well,
 i.e.~fulfills~\rife{ccri}.

$(iii)$ \,
 Putting together Theorem \ref{main-zw}-$(iii)$ and Proposition \ref{bgp-tw}, for every $c\in [2\kappa\sqrt{\alpha(1)}, \alpha(1)+ \kappa^2)$ such that $c<\rho$, we have a BGP solution of  \rife{system} with growth $c$, and with arbitrary normalization at any given point.
\qed

\vskip2em

{\bf Acknowledgement.} We wish to thank Benjamin Moll for stimulating  our interest in the problem as well as for addressing to us interesting related references in the economic literature.

\end{document}